\documentclass[11pt, a4pa  per, twoside]{article}
\usepackage{bbm}
\usepackage{pifont}
\usepackage{bbding}
\usepackage{mathrsfs}
\usepackage{pgf,tikz}
\usepackage{amsfonts,amssymb,amsmath,indentfirst,amsthm}
\usepackage{epsfig}
\usepackage{fancyhdr}
\usepackage{graphicx}
\usepackage{subcaption}
\usepackage{float}
\usepackage{cite}
\usepackage{url}
\usepackage{xcolor}
\usepackage{cite}
\usepackage{calc}
\usepackage{cases}
\usetikzlibrary{patterns}

\def\ess~inf{\mathop{\rm ess~inf}}

\numberwithin{equation}{section}

\newenvironment{key words}{\emph{\texttt{Keywords}}\mbox{  }}{ }
\newtheorem{theorem}{Theorem}[section]
\newtheorem{lemma}[theorem]{Lemma}

\newtheorem{proposition}[theorem]{Proposition}

\newtheorem{definition}[theorem]{Definition}
\renewenvironment{proof}{\noindent{\textbf{Proof.}}}{\hfill$\Box$}
\theoremstyle{remark}
\newtheorem{remark}[theorem]{\textbf{Remark}}

\theoremstyle{plain}

\makeatletter

\newcommand{\Rmnum}[1]{\expandafter\@slowromancap\romannumeral #1@}
\makeatother

\usepackage{comment}

\allowdisplaybreaks

\begin{document}

\fancyhf{}

\fancyhead[EC]{W. Li and H. Wang}

\fancyhead[EL]{\thepage}

\fancyhead[OC]{Maximal functions related to homogeneous hypersurfaces in $\mathbb{R}^3$}

\fancyhead[OR]{\thepage}

\renewcommand{\headrulewidth}{0.5pt}
\renewcommand{\thefootnote}{\fnsymbol {footnote}}

\title
{\bf{Maximal functions
related to homogeneous  hypersurfaces in $\mathbb{R}^3$} }

\author{Wenjuan Li$^{\textrm{a}}$ and Huiju Wang$^{\textrm{b}}$\\
 \small{a. Northwestern Polytechnical University, Xi'an, 710129, China}  \\
\small{b. Henan University, Kaifeng, 475000, China} }


\date{}
 \maketitle

{\bf Abstract:} We study maximal functions related to homogeneous polynomial hypersurfaces in $\mathbb{R}^3$.
In a sense made precise in this paper, the region of $(p,q)$ for which we obtain $L^p\rightarrow L^q$ boundedness is optimal up to the endpoints for the corresponding local maximal operators. The boundedness exponents depend explicitly on both the height of the hypersurface and the type of the curve determined by the level set. As a corollary, we obtain $L^p$-estimates and weighted norm inequalities for the associated global maximal functions. Moreover, we also obtain optimal $L^{p}$-estimates for the global maximal operators associated with homogeneous polynomial hypersurfaces without transversality condition in $\mathbb{R}^{3}$.



 \tableofcontents
\section{Introduction}

Let $S$ be a smooth hypersurface in $\mathbb{R}^n$ and $\eta\in C_0^{\infty}(S)$ be a smooth non-negative function with compact support. Then the associated averaging operator is defined by
\begin{equation}\label{def1}
A_{t}f(y)=\int_{S}f(y-tx)\eta(x)d\sigma(x), \quad t>0,
\end{equation}
where $d\sigma$ denotes the surface measure on $S$. A fundamental and still largely open problem is to characterize the $L^p(\mathbb{R}^n)$-boundedness of the global maximal operator associated to the hypersurface where the Gaussian curvature at some points is allowed to vanish. Iosevich \cite{I} gave a general  answer for the $2$-dimensional case, i.e. $S$ is  a curve of finite type in $\mathbb{R}^2$. For the higher dimensional case, one can see the works by Buschenhenke--Dendrinos--Ikromov--M\"{u}ller \cite{BIM1}, Buschenhenke--Ikromov--M\"{u}ller \cite{BIM2}, Greenleaf \cite{G}, Ikromov--M\"{u}ller \cite{IM,IM1}, Ikromov--Kempe--M\"{u}ller \cite{IKM1,IKM2}, Iosevich--Sawyer \cite{iS1,iS2}, Li \cite{L}, Zimmermann \cite{Z} etc.

However, much less is known about the $L^p \rightarrow L^q$ boundedness $(p\leq q)$ for the local maximal operator $\mathcal{M}_{loc}$ defined by
 \[\mathcal{M}_{loc}f(y)=\sup_{t\in [1,2]} |A_tf(y)|, \quad y\in \mathbb{R}^n\]
 associated with a hypersurface $S$ whose Gaussian curvature vanishes at some points, especially when $S\subset \mathbb{R}^n$, $n\geq 3$. Once we obtain the $L^p \rightarrow L^q$ boundedness for local maximal operators $\mathcal{M}_{loc}$, then the weighted estimates will follow for the corresponding global maximal operators by the methodology of sparse domination (see Section 1.2 in \cite{LWZ}). Specially, Dendrinos--Zimmermann \cite{DZ} and Schwend \cite{Schwend} obtained $L^p \rightarrow L^q$ estimates for averaging operators associated to mixed homogeneous polynomial hypersurfaces in $\mathbb{R}^3$. For more details on the $L^p \rightarrow L^q$  boundedness of averaging operators over hypersurfaces, we refer to the references in \cite{DZ,Schwend}.

In this paper, we   consider maximal operators along the hypersurfaces
\[S:=\{(x_{1},x_{2}, \Phi(x_{1},x_{2}) +c): (x_{1},x_{2}) \in U\},\]
 where $c \in \mathbb{R}$, $U$ is a small neighborhood of the origin in $\mathbb{R}^2$. $\Phi$ is a real-valued polynomial, which is homogeneous of degree $m$. We always assume that $m \ge 2$ in this paper. Then applying polynomial factorization, we write $\Phi$ as
\begin{equation}\label{Phi}
\Phi(x_{1},x_{2})=c_{\Phi}x_{1}^{\nu_{1}}x_{2}^{\nu_{2}} \Pi_{h=1}^{N}(x_{2}-\lambda_{h}x_{1})^{n_{h}},
\end{equation}
where  $N \ge 1$, distinct $\lambda_{h} \in \mathbb{C}  \backslash \{0\}$ and multiplicities $n_{h} \in \mathbb{N}\backslash \{0\}$ with $\nu_{1}, \nu_{2} \in \mathbb{N}$,
$\nu_{1}+\nu_{2}+\sum_{h=1}^{N} n_{h}= m$. We consider the following three cases.

\textbf{Case (i):} One of $\nu_{1},\nu_{2}$ equals $m$. Without loss of generality, we may assume that $\nu_{2}=m$, then   \[\Phi(x_{1},x_{2}) = c_{\Phi}x^{m}_{2}.\]

\textbf{Case (ii):} There exists some $1 \le h\le N$ such that $n_{h} =m$, denote such $h$ by $h_0$,
 then
 \begin{equation}\label{caseii}
 \Phi(x_{1},x_{2}) = c_{\Phi}(x_{2}-\lambda_{h_{0}}x_{1})^{m}.
 \end{equation}
 Since $\Phi$ is a real-valued polynomial, if $\lambda_{h_{0}}$ is complex, then its complex conjugate is also a root of $\Phi$. This contradicts the fact that the polynomial defined by (\ref{caseii}) has exactly one root. Hence, we have $\lambda_{h_{0}} \in \mathbb{R}$. Changing variables, then we have
 \[\Phi(z_{1}, z_{2}+\lambda_{h_{0}}z_{1})=c_{\Phi}z^{m}_{2}.\]

\textbf{Case (iii):} $0 \le \nu_{1} <m$, $0 \le \nu_{2} <m$ and $0 \le n_{h} <m$.

In Case (i) and Case (ii), we reduce the problem to considering the boundedness of the maximal operator related to
 \[S =\{(x_{1},x_{2}, x^{m}_{2}+c): (x_{1},x_{2}) \in U\},\]
which has been extensively  studied in \cite{IKM2,Z,L,LWZ}. Hence, in this paper, we concentrate on the Case (iii).

According to \cite{IM}, when represented in adapted coordinates, the height of an analytic function $\Phi$ can be directly determined from its Newton polyhedron through the quantity known as the Newton distance. In particular, the height in case of a homogeneous polynomial  $\Phi$ can be characterized more clearly by
\[h_{\Phi}:= \max\{\frac{m}{2}, \textsf{ordd}\hspace{0.1cm} \Phi \},\]
where $\textsf{ordd} \hspace{0.1cm}\Phi = \sup_{x \in S^{1}}\textsf{ordd} \hspace{0.1cm}\Phi (x)$, $S^{1}$ is the unit circle in $\mathbb{R}^{2}$ and $\textsf{ordd} \hspace{0.1cm}\Phi (x)$ is the smallest non-negative integer $j$ such that the $j$-th order total derivative of $\Phi$ at $x$ is non-zero. For more details, one can see Corollary 3.4 and the formula (2.11) in \cite{IM}.

\subsection{$L^{p} \rightarrow L^{q}$ estimate of the local maximal operator}
Let $H\Phi$ be the determinant of the Hessian matrix of $\Phi$,  we have
\begin{equation}\label{HPhi0}
H\Phi(x_{1},x_{2}) = x^{\alpha_{1}}_{1}\Pi_{\lambda \in R}(x_{2}-\lambda x_{1})^{w_{\lambda}}Q(x_{1},x_{2}),
\end{equation}
where $R$ denotes all the real roots of $H\Phi$ (including possibly $\lambda =0$), $Q(x_{1},x_{2}) \neq 0$ provided that $(x_{1},x_{2}) \neq (0,0)$. We define
\[L_{\lambda} := \{(x_1,x_2): x_{2}=\lambda x_{1}\},\quad \lambda\in R,\]
and
\begin{equation}
L_{\infty}:=
\begin{cases}
\{(x_1,x_2): x_{1}=0\}, \:\:\:\ \alpha_{1} \neq 0,\\
\emptyset, \:\:\:\:\:\:\:\:\:\:\:\:\:\:\:\:\:\:\:\:\:\:\:\:\:\:\:\:\:\:\:\:\:\:\:\:\:\  \textmd{else}. \\
\end{cases}
\end{equation}
Let $Z_{H\Phi}$ be the zero set of $H\Phi$. Then
\[Z_{H\Phi} =\cup_{\lambda \in R \cup \{\infty\}}L_{\lambda}. \]

For $\epsilon>0$, and $\lambda\in R$, we denote
$$\Gamma(\lambda):=\{(x_1,x_2): \angle((x_1,x_2),(1,\lambda))< \epsilon \} \cup \{(x_1,x_2): \angle((x_1,x_2),(-1,-\lambda))< \epsilon \},$$
where $\angle(V_1,V_2)$ means the angle between the vectors $V_1$ and $V_2$ in $\mathbb{R}^2$, and for $\lambda=\infty$,
$$\Gamma(\lambda):=\{(x_1,x_2): \angle((x_1,x_2),(0,1))< \epsilon \} \cup \{(x_1,x_2): \angle((x_1,x_2),(0,-1))< \epsilon \}.$$
We choose $\epsilon > 0$ sufficiently small such that $\Gamma(\lambda) \cap \Gamma(\lambda^{\prime}) = \emptyset$ whenever $\lambda \neq \lambda^{\prime}$, then
$$\Gamma_{\lambda} \cap Z_{H\Phi} = L_{\lambda} \backslash \{(0,0)\}.$$

We decompose $\mathbb{R}^{2}\backslash \{(0,0)\} $ into conic sectors $\Gamma_{j}$, $j=1,2,\cdot \cdot \cdot, J$, where  either \\
\textbf{(a)} $\Gamma_{j}$ equals one of $\Gamma(\lambda), \lambda\in R\cup \{\infty\}$, or \\
\textbf{(b)} $H\Phi(x_{1},x_{2})$ does not vanish for each $(x_{1},x_{2}) \in \Gamma_{j}$.

 \begin{remark}\label{remarkGammaj}
 For the conic sector $\Gamma_{j}$ in the case (b), we denote the corresponding hypersurface
 \[S_{\Gamma_{j}}=\{(x_{1},x_{2},\Phi(x_{1},x_{2})+c):(x_{1},x_{2}) \in \Gamma_{j}\}.\]
 Since $H\Phi(x_{1},x_{2}) \neq 0$ for each $(x_{1},x_{2}) \in \Gamma_{j}$, by a dyadic decomposition and  scaling arguments, it can be obtained that  the corresponding local maximal operator $\mathcal{M}_{loc,S_{\Gamma_{j}}}$ defined by
 \[\mathcal{M}_{loc,S_{\Gamma_{j}}} = \sup_{t \in [1,2]} \biggl|\int_{S_{\Gamma_{j}}}f(y-tx)d\sigma(x)\biggl|\]
is $L^{p}\rightarrow L^{q}$ bounded if \\
(1) $c=0$, $(\frac{1}{p},\frac{1}{q}) \in \Delta_{0}$ and $\frac{m/2+1}{p}-\frac{m/2+1}{q}-1<0$;  \\
(2) $c \neq 0$, $(\frac{1}{p},\frac{1}{q}) \in \Delta_{0}$ and $\frac{m/2+1}{p}-\frac{1}{q}-1<0$.  \\
Here,  $\Delta_{0}$ is constructed by the interior of the quadrilateral with vertices
\[O_{1}=(0,0),\hspace{0.3cm} O_{2}=(2/3,2/3), \hspace{0.3cm}O_{3}= (2/3,1/3), \hspace{0.3cm}O_{4}=(3/5,1/5),\] and  the segment $O_{1}O_{2}$ with $O_{1}$ included and $O_{2}$ excluded. We will give some details of the proof in Section \ref{proofofmainth}.
\end{remark}

It suffices to consider the conic sectors $\Gamma_{j}$ in the case (a).  For fixed $\lambda$, we consider the local maximal operator $\mathcal{M}_{loc,\lambda}$ defined by
\[\mathcal{M}_{loc,\lambda}f(y) = \sup_{t\in [1,2]} |A_{t,\lambda}f(y) |,\]
where
\begin{equation}\label{avAlambda}
A_{t,\lambda}f(y) =  \int_{S_{\lambda}}f(y-tx)d\sigma_{\lambda}(x),
\end{equation}
and
\[S_{\lambda}= \{(x_{1},x_{2},\Phi(x_{1},x_{2})+c):(x_{1},x_{2}) \in \Gamma_{\lambda}\},\]
 $d\sigma_{\lambda}$ is the surface measure on $S_{\lambda}$.  We split $\Gamma(\lambda)$ into two types
 according to whether $\Phi$ vanishes along $L_{\lambda} \backslash \{(0,0)\}$ or not, and estimate the corresponding local maximal operator $\mathcal{M}_{loc,\lambda}$  respectively.

\subsubsection{Type A: $\Phi$ vanishes along $L_{\lambda}\backslash \{(0,0)\}$}
 We say that $\Gamma(\lambda)$ is of Type A provided that $\Phi$ vanishes along $L_{\lambda}\backslash \{(0,0)\}$. If $ \lambda \neq \infty$, we may assume that
\[\Phi(x_1,x_2)=(x_{2}-\lambda x_{1})^{n_{\lambda}}P(x_{1},x_{2}),\]
where $P$ does not vanish along $L_{\lambda}\backslash \{(0,0)\}$. If $\lambda = \infty$, $\Phi$ vanishes along $x_{1}=0$, we set $n_{\lambda} = \nu_{1}$.
We have the following $L^{p} \rightarrow L^{q}$ estimate for $\mathcal{M}_{loc, \lambda}$.
\begin{theorem}\label{mainth1+}
Denote $h_{\lambda} = \max\{m/2, n_{\lambda}\}$.

(1) When $c = 0$,  then $\|\mathcal{M}_{loc,\lambda}\|_{L^{p}\rightarrow L^{q}} < \infty$ if
\[(\frac{1}{p},\frac{1}{q}) \in   \{(\frac{1}{p},\frac{1}{q}) \in \Delta_{0}: \frac{h_{\lambda}+1}{p} - \frac{h_{\lambda}+1}{q}- 1 <0\};\]

(2) When $c \neq 0$, then  $\|\mathcal{M}_{loc,\lambda}\|_{L^{p}\rightarrow L^{q}} < \infty$ if
\[(\frac{1}{p},\frac{1}{q}) \in   \{(\frac{1}{p},\frac{1}{q}) \in \Delta_{0}: \frac{h_{\lambda}+1}{p}-\frac{1}{q}-1<0\}.\]
\end{theorem}

According to Theorem \ref{apth1} below, the results in Theorem \ref{mainth1+}  are sharp up to endpoints.

\begin{theorem}\label{apth1}
Assume that $\lambda \in R$  such that $\Gamma(\lambda)$ is of Type A.   Then the following conditions are necessary for $\mathcal{M}_{loc,\lambda}$ to be $L^{p} \rightarrow L^{q}$ bounded.

\textbf{(1)} 
$(\frac{1}{p},\frac{1}{q}) \in \overline{\Delta_{0}}$, $\overline{\Delta_{0}}$ is the closure of $\Delta_{0}$.   \\
\textbf{(2)} When $c=0$, $\frac{h_{\lambda}+1}{p}-\frac{h_{\lambda}+1}{q}-1 \le 0$ is necessary for $\mathcal{M}_{loc,\lambda}$ to be $L^{p}  \rightarrow L^{q} $ bounded; when $c \neq 0$, $\frac{h_{\lambda}+1}{p}-\frac{1}{q}-1 \le 0$ is  necessary for $\mathcal{M}_{loc,\lambda}$ to be $L^{p}  \rightarrow L^{q} $ bounded.
\end{theorem}

 Let $Z_{\Phi}$ be the zero set of $\Phi$. If $Z_{H\Phi} \subset Z_{\Phi}$, then the conic sectors $\Gamma(\lambda)$, $\lambda\in R\cup\{\infty\}$, are all of Type  A. We can obtain the following proposition.
\begin{proposition}\label{muphihphi}
If there is no $\Gamma(\lambda)$ that is of Type A, we have $\frac{m}{2} = h_{\Phi}$.
 Otherwise, it can be proved that
\[\max \bigl\{h_{\lambda}: \Gamma(\lambda) \text{ is of Type A} \bigl\} =h_{\Phi}.\]
\end{proposition}

 By Theorem \ref{mainth1+}, Theorem \ref{apth1} and Proposition \ref{muphihphi}, there holds the following Theorem \ref{mainth1}.

\begin{theorem}\label{mainth1}
Assuming that $Z_{H\Phi} \subset Z_{\Phi}$, \\
(1) when $c = 0$,  then $\|\mathcal{M}_{loc}\|_{L^{p}\rightarrow L^{q}} <\infty$ if
\begin{equation}\label{Delta1}
(\frac{1}{p},\frac{1}{q}) \in \widetilde{\Delta_{1}}= \{(\frac{1}{p},\frac{1}{q}) \in \Delta_{0}: \frac{h_{\Phi}+1}{p} - \frac{h_{\Phi}+1}{q}- 1 <0\};
\end{equation}
(2) when $c \neq 0$, then  $\|\mathcal{M}_{loc}\|_{L^{p}\rightarrow L^{q}} < \infty$ if
\begin{equation}\label{Delta2}
(\frac{1}{p},\frac{1}{q}) \in \widetilde{\Delta_{2}}=  \{(\frac{1}{p},\frac{1}{q}) \in \Delta_{0}: \frac{h_{\Phi}+1}{p}-\frac{1}{q}-1<0\}.
\end{equation}
(3) The $L^{p} \rightarrow L^{q}$ boundedness  region in part (1) and part (2) cannot be improved up to the boundary.
\end{theorem}


Notice when $h_{\Phi}=3/2$, the results  in Theorem \ref{mainth1} (1) coincide with that of the local spherical maximal operators (see the article \cite{WS}), and when $h_{\Phi}>3/2$, the results in Theorem \ref{mainth1} depend  heavily on the height of $\Phi$, see Figure 1 and Figure 2.

\begin{figure}[H]
\begin{center}
\begin{tikzpicture}[scale=6.5]
    \draw[->] (0,0) -- (1,0) node[right] {\(\frac1p\)};
    \draw[->] (0,0) -- (0,1) node[above] {\(\frac1q\)};

     \draw[densely dotted] (0,0) -- (0.8,0.8);
      \draw[green,densely dotted] (0.2,0.03) -- (0.85,0.7);
      \draw[red,densely dotted] (0.42,0.1) -- (0.85,0.555);
      \draw[densely dotted] (0,0) -- (0.7,0.3);
     \draw[densely dotted] (0.8,0.8) -- (0.8,0.5);
     \draw[blue,densely dotted] (0.55,0.14) -- (0.78,0.39);
     \draw[densely dotted] (0.7,0.3) -- (0.8,0.5);

     straightlne
     \draw[thick] (0,0) -- (0.8,0.8);

    \fill[blue, opacity=0.8] (0,0) -- (0.8,0.8) -- (0.8,0.65) -- (0.285,0.12) -- cycle;

   \draw[thick] (0.8,0.8) circle (0.3pt);
   \draw[thick] (0.8,0.5) circle (0.3pt);
   \draw[thick] (0.7,0.3) circle (0.3pt);

    \node[below left] at (0,0) {\(O_1\)};
     \node[above] at (0.81,0.805) {\(O_2\)};
    \node[right] at (0.8,0.48) {\(O_3\)};
    \node[below] at (0.7,0.3) {\(O_4\)};
    \node[right] at (0.83,0.75) {\(h_{\Phi}>2\)};
    \node[right] at (0.83,0.58) {\(h_{\Phi}=2\)};
    \node[right] at (0.78,0.39) {\(h_{\Phi}=3/2\)};
    \node[white,right] at (0.38,0.35) {\(\widetilde{\Delta_{1}}\)};
\end{tikzpicture}
\caption{The range of $\widetilde{\Delta_{1}}$ when the height $h_{\Phi}$ changes, and  $\Delta_{0}$ is the range of the boundedness for local spherical maximal operators, which
is constructed by the interior of the quadrilateral with vertices
$O_{1},O_{2}, O_{3}, O_{4}$, and  the segment $O_{1}O_{2}$ with $O_{1}$ included and $O_{2}$ excluded.}
\end{center}
\end{figure}

\begin{figure}[H]
\begin{center}
\begin{tikzpicture}[scale=6.5]
    \draw[->] (0,0) -- (1,0) node[right] {\(\frac1p\)};
    \draw[->] (0,0) -- (0,1) node[above] {\(\frac1q\)};

     \draw[densely dotted] (0,0) -- (0.8,0.8);
     \draw[green,densely dotted] (0.26,0.03) -- (0.32,0.42);
      \draw[red,densely dotted] (0.42,0.1) -- (0.58,0.7);
      \draw[densely dotted] (0,0) -- (0.7,0.3);
     \draw[densely dotted] (0.8,0.8) -- (0.8,0.5);
     \draw[blue,densely dotted] (0.55,0.14) -- (0.86,0.95);
     \draw[densely dotted] (0.7,0.3) -- (0.8,0.5);

     \draw[thick] (0,0) -- (0.8,0.8);

    \fill[blue, opacity=0.8] (0,0) -- (0.3,0.3) -- (0.275,0.12)  -- cycle;

   \draw[thick] (0.8,0.8) circle (0.3pt);
   \draw[thick] (0.8,0.5) circle (0.3pt);
   \draw[thick] (0.7,0.3) circle (0.3pt);

    \node[below left] at (0,0) {\(O_1\)};
     \node[right] at (0.81,0.805) {\(O_2\)};
    \node[right] at (0.8,0.48) {\(O_3\)};
    \node[below] at (0.7,0.3) {\(O_4\)};
    \node[above] at (0.32,0.42) {\(h_{\Phi}>2\)};
    \node[above] at (0.58,0.7) {\(h_{\Phi}=2\)};
    \node[above] at (0.86,0.95) {\(h_{\Phi}=3/2\)};
    \node[white,right] at (0.18,0.16) {\(\widetilde{\Delta_{2}}\)};
\end{tikzpicture}
\caption{The range of $\widetilde{\Delta_{2}}$ when the height $h_{\Phi}$ changes.}
\end{center}
\end{figure}

\subsubsection{Type B: $\Phi$ does not vanish along $L_{\lambda}\backslash \{(0,0)\}$}\label{CasB}
We say that $\Gamma(\lambda)$ is of Type B if $\Phi$ does not vanish along $L_{\lambda}\backslash \{(0,0)\}$. In this case, the  type of the curves determined by the level set will play an essential role in the $L^{p}\rightarrow L^{q}$ estimate for $\mathcal{M}_{loc,\lambda}$. We will now provide a detailed description of the curve determined by the level set.

We choose $(x_{1}^{0},x_{2}^{0}) \in L(\lambda)\backslash\{(0,0)\}$ such that $H\Phi(x^{0}_{1},x^{0}_{2}) = 0$ and $\Phi(x^{0}_{1},x^{0}_{2}) = 1$.

$\bullet$ When $\lambda <\infty$, we change variables  $x_{1} \rightarrow z_{1}$, $x_{2} \rightarrow z_{2}+\lambda z_{1}$, and denote $z^{0}_{1}=x^{0}_{1}$, $z^{0}_{2}=0$, $\Phi_{\lambda}(z_{1},z_{2}):= \Phi(z_{1},z_{2}+\lambda z_{1})$.

$\bullet$  When $\lambda = \infty$, we change variables $x_{1} \rightarrow z_{2}$, $x_{2} \rightarrow z_{1}$, and denote $z^{0}_{1}=x^{0}_{2}$, $z^{0}_{2}=0$, $\Phi_{\infty}(z_{1},z_{2}):= \Phi(z_{2},z_{1})$.

For fixed $\lambda$, since $\Phi_{\lambda}(rz_{1},rz_{2}) = r^{m}\Phi_{\lambda}(z_{1},z_{2})$ for each $r>0$,
\[z^{0}_{1}\partial_{1}\Phi_{\lambda}(z^{0}_{1},0) = m\Phi_{\lambda}(z^{0}_{1},0) = m \Phi(x^{0}_{1}, x^{0}_{2}) \neq 0, \]
which implies $\partial_{1}\Phi_{\lambda}(z^{0}_{1},0) \neq 0$. By implicit function theorem, there is a unique local $C^{\infty}$-curve $\gamma:z_{2}\rightarrow (\Psi(z_{2}),z_{2}), |z_{2}|< \epsilon_{0} \ll 1$ such that $\Psi(0) = z^{0}_{1}$ and
\[\Phi_{\lambda}(\Psi(z_{2}),z_{2}) =\Phi_{\lambda}(z^{0}_{1},0)=\Phi(x^{0}_{1}, x^{0}_{2})=1, \quad |z_{2}|< \epsilon_{0}.\]
We observe that the conic sector $\Gamma(\lambda)$ has changed into $\Gamma(0)$. If appropriate $\epsilon$ (in the definition of $\Gamma(0)$) is sufficiently small such that $\epsilon<\epsilon_0$, then we have
\[\{\Phi_{\lambda}(z_{1},z_{2}) = 1\}\cap \Gamma(0)= \{(\Psi(z_{2}),z_{2}): |z_{2}| < \epsilon\}. \]
We can obtain the following proposition.
\begin{proposition}\label{level set type}
There exists the smallest integer $M^{\lambda}$ which is no less than $2$ such that $\Psi^{(M^{\lambda}) } (0) \neq 0$. More concretely, by (\ref{HPhi0}), if $\lambda \neq \infty$, then $M^{\lambda} = \omega_{\lambda}+2$; if $\lambda=\infty$, then $M^{\lambda}=\alpha_{1}+2$.
\end{proposition}

 \begin{definition}\label{DeltaM}
 For convenience, we define the region $\Delta_{M}$  for each integer $M \ge 3$  as follows.
 (1) $M \ge 6$, $\Delta_{M}$ is constructed by the interior of the quadrilateral with vertices $P_{1}=(0,0)$, $P_{2}=(\frac{M+1}{2M},\frac{M+1}{2M})$, $P_{3}= (\frac{4}{M+2},\frac{2}{M+2})$, $P_{4}=(\frac{3}{M+2},\frac{1}{M+2})$, and   the segment $P_{1}P_{2}$ with $P_{1}$ included and $P_{2}$ excluded; \\
 (2) $M =5$, $\Delta_{M}$ is constructed by the interior of the pentagon  with vertices $P_{1}=(0,0)$, $P_{2}=(\frac{M+1}{2M},\frac{M+1}{2M})$, $P_{3}= (\frac{2M+4}{5M+2},\frac{M+2}{5M+2})$, $P_{4}=(\frac{3}{M+2},\frac{1}{M+2})$, $P_{5}=(\frac{1}{2},\frac{1}{2}\cdot \frac{M-2}{M+2})$, and  the segment $P_{1}P_{2}$ with $P_{1}$ included and $P_{2}$ excluded;  \\
(3) $M =4$, $\Delta_{M}$ is constructed by the interior of the quadrilateral with vertices $P_{1}=(0,0)$, $P_{2}=(\frac{M+1}{2M},\frac{M+1}{2M})$, $P_{3}= (\frac{2M+4}{5M+2},\frac{M+2}{5M+2})$, $P_{4}=(\frac{3}{M+2},\frac{1}{M+2})$, and  the segment $P_{1}P_{2}$ with $P_{1}$ included and $P_{2}$ excluded; \\
(4) $M =3$, $\Delta_{M}$ is constructed by the interior of the quadrilateral with vertices $P_{1}=(0,0)$, $P_{2}=(\frac{M+1}{2M},\frac{M+1}{2M})$, $P_{3}= (\frac{2M+4}{5M+2},\frac{M+2}{5M+2})$,
$P_{4}=(\frac{3M+6}{8M+4},\frac{M+2}{8M+4})$, and   the segment $P_{1}P_{2}$ with $P_{1}$ included and $P_{2}$ excluded.
\end{definition}

We can obtain the following $L^{p}\rightarrow L^{q}$ estimate for $\mathcal{M}_{loc, \lambda}$.
\begin{theorem}\label{mainth2+}
We have

(1) when $c = 0$,  then $\|\mathcal{M}_{loc,\lambda}\|_{L^{p}\rightarrow L^{q}} < \infty$ if
\[(\frac{1}{p},\frac{1}{q}) \in   \{(\frac{1}{p},\frac{1}{q}) \in \Delta_{M^{\lambda}}: \frac{m/2+1}{p} - \frac{m/2+1}{q}- 1 <0\};\]

(2) when $c \neq 0$,  then $\|\mathcal{M}_{loc,\lambda}\|_{L^{p}\rightarrow L^{q}} < \infty$ if
\[(\frac{1}{p},\frac{1}{q}) \in   \{(\frac{1}{p},\frac{1}{q}) \in \Delta_{M^{\lambda}}: \frac{m/2+1}{p}-\frac{1}{q}-1<0\}.\]
\end{theorem}

\begin{remark}\label{equi-reg}
The proof of Theorem \ref{apth2} below implies that $M^{\lambda} \le m$, see the inequality (\ref{Mm}). So we can obtain that, in the part (2) of Theorem \ref{mainth2+}, the $L^p\rightarrow L^q$ boundedness region is equivalent to
\[\{(\frac{1}{p},\frac{1}{q}): \frac{1}{q}\leq\frac{1}{p}<\frac{3}{q}, \quad \frac{m/2+1}{p}-\frac{1}{q}-1<0\}.\]
\end{remark}

For convenience, we denote some lines that play a crucial role in constructing the region $\Delta_{M^{\lambda}}$ for each integer $M^{\lambda} \ge 3$. Set
\begin{equation*}
\mathfrak{L}_1:=\{(x,y):x=\frac{M^{\lambda}+1}{2M^{\lambda}}\},\quad \mathfrak{L}_2:=\{(x,y):y=\frac{3M^{\lambda}+2}{M^{\lambda}+2}x-1\},
\end{equation*}

\begin{equation*}
\mathfrak{L}_3:=\{(x,y):(\frac{M^{\lambda}}{2}+1)x-(\frac{M^{\lambda}}{2}+1)y-1=0\},\quad \mathfrak{L}_4:=\{(x,y):x=2y\},
\end{equation*}

\begin{equation*}
\mathfrak{L}_5:=\{(x,y):x=3y\}, \quad \mathfrak{L}_6:=\{(x,y):(\frac{m}{2}+1)x-(\frac{m}{2}+1)y-1=0\}.
\end{equation*}

\begin{figure}[H]
\begin{center}
\begin{subfigure}[b]{0.45\textwidth}
\begin{center}
\begin{tikzpicture}[scale=7]
    \draw[->] (0,0) -- (1,0) node[right] {\(\frac1p\)};
    \draw[->] (0,0) -- (0,1) node[above] {\(\frac1q\)};

     \draw[densely dotted] (0,0) -- (0.75,0.75);
     \draw[green,densely dotted] (0.6,0.6) -- (0.49,0.17);
     \draw[purple,densely dotted] (0.09,-0.02) -- (0.75,0.62);
      \draw[red,densely dotted] (0.25,-0.05) -- (0.74,0.405);
     \draw[densely dotted] (0,0) -- (0.75,0.15);
      \draw[densely dotted] (0,0) -- (0.75,0.26);
     \draw[densely dotted] (0.6,0.7) -- (0.6,0.05);

     \draw[thick] (0,0) -- (0.6,0.6);

    \fill[blue, opacity=0.8] (0,0) -- (0.6,0.6) -- (0.555,0.43) -- (0.14,0.028) -- cycle;
    \fill[cyan, opacity=0.8] (0.14,0.028) -- (0.555,0.43) --(0.49,0.174)-- (0.39,0.08) -- cycle;

   \draw[thick] (0.6,0.6) circle (0.3pt);
   \draw[thick] (0.49,0.174) circle (0.3pt);
   \draw[thick]  (0.39,0.08) circle (0.3pt);

    \node[below left] at (0,0) {\(P_1\)};
     \node[right] at (0.6,0.6) {\(P_2\)};
    \node[below ] at (0.53,0.19) {\(P_3\)};
    \node[below] at (0.43,0.09) {\(P_4\)};
   \node[above] at (0.6,0.7) {\(\mathfrak{L}_1\)};
   \node[right] at (0.75,0.62) {\(\mathfrak{L}_6\)};
   \node[right] at (0.74,0.405) {\(\mathfrak{L}_3\)};
   \node[right] at (0.75,0.26) {\(\mathfrak{L}_4\)};
   \node[right] at (0.75,0.18) {\(\mathfrak{L}_5\)};
   \node[white,right,font=\tiny] at (0.275,0.16) {\(m=M^{\lambda}\)};
    \node[white,right,font=\tiny] at (0.34,0.35) {\(m>M^{\lambda}\)};
\end{tikzpicture}
\caption{The case $M^{\lambda}\geq 6$.}
\end{center}
\end{subfigure}
    \hfill
    \begin{subfigure}[b]{0.45\textwidth}
\begin{center}
\begin{tikzpicture}[scale=7]
    \draw[->] (0,0) -- (1,0) node[right] {\(\frac1p\)};
    \draw[->] (0,0) -- (0,1) node[above] {\(\frac1q\)};

     \draw[densely dotted] (0,0) -- (0.75,0.75);
     \draw[green,densely dotted] (0.6,0.6) -- (0.497,0.174);
     \draw[purple,densely dotted] (0.09,-0.02) -- (0.75,0.62);
      \draw[red,densely dotted] (0.26,-0.08) -- (0.75,0.4);
       \draw[blue,densely dotted] (0.36,-0.09) -- (0.75,0.65);
     \draw[densely dotted] (0,0) -- (0.75,0.15);
      \draw[densely dotted] (0,0) -- (0.75,0.26);
     \draw[densely dotted] (0.6,0.7) -- (0.6,0.05);

     \draw[thick] (0,0) -- (0.6,0.6);

    \fill[blue, opacity=0.8] (0,0) -- (0.6,0.6) -- (0.555,0.43) -- (0.14,0.028) -- cycle;
    \fill[cyan, opacity=0.8] (0.14,0.028) -- (0.555,0.43) --(0.497,0.174)-- (0.477,0.13)-- (0.425,0.085) -- cycle;

   \draw[thick] (0.6,0.6) circle (0.3pt);
   \draw[thick] (0.497,0.174) circle (0.3pt);
   \draw[thick]  (0.425,0.085) circle (0.3pt);
   \draw[thick]  (0.477,0.13) circle (0.3pt);

    \node[below left] at (0,0) {\(P_1\)};
     \node[right] at (0.6,0.6) {\(P_2\)};
    \node[left ] at (0.51,0.21) {\(P_3\)};
    \node[below] at (0.43,0.08) {\(P_4\)};
     \node[below] at (0.52,0.165) {\(P_5\)};
     \node[right] at (0.75,0.69) {\(\mathfrak{L}_2\)};
   \node[above] at (0.6,0.7) {\(\mathfrak{L}_1\)};
   \node[right] at (0.75,0.6) {\(\mathfrak{L}_6\)};
    \node[right] at (0.74,0.405) {\(\mathfrak{L}_3\)};
   \node[right] at (0.75,0.26) {\(\mathfrak{L}_4\)};
   \node[right] at (0.75,0.18) {\(\mathfrak{L}_5\)};
   \node[white,right,font=\tiny] at (0.275,0.16) {\(m=M^{\lambda}\)};
    \node[white,right,font=\tiny] at (0.34,0.35) {\(m>M^{\lambda}\)};
\end{tikzpicture}
\caption{The case $M^{\lambda}=5$.}
\end{center}
 \end{subfigure}
\end{center}
\end{figure}

\begin{figure}[H]
\begin{center}
\begin{subfigure}[b]{0.45\textwidth}
\begin{center}
\begin{tikzpicture}[scale=7]
    \draw[->] (0,0) -- (1,0) node[right] {\(\frac1p\)};
    \draw[->] (0,0) -- (0,1) node[above] {\(\frac1q\)};

     \draw[densely dotted] (0,0) -- (0.75,0.75);
     \draw[green,densely dotted] (0.6,0.6) -- (0.497,0.174);
     \draw[purple,densely dotted] (0.09,-0.02) -- (0.75,0.62);
      \draw[red,densely dotted] (0.26,-0.1) -- (0.75,0.38);
       \draw[blue,densely dotted] (0.35,-0.09) -- (0.75,0.61);
     \draw[densely dotted] (0,0) -- (0.75,0.15);
      \draw[densely dotted] (0,0) -- (0.75,0.26);
     \draw[densely dotted] (0.6,0.7) -- (0.6,0.05);

     \draw[thick] (0,0) -- (0.6,0.6);

    \fill[blue, opacity=0.8] (0,0) -- (0.6,0.6) -- (0.558,0.435) -- (0.14,0.028) -- cycle;
    \fill[cyan, opacity=0.8] (0.14,0.028)--(0.558,0.435)--(0.498,0.172)-- (0.45,0.089) -- cycle;

   \draw[thick] (0.6,0.6) circle (0.3pt);
   \draw[thick] (0.497,0.174) circle (0.3pt);
   \draw[thick]  (0.45,0.089) circle (0.3pt);

    \node[below left] at (0,0) {\(P_1\)};
     \node[right] at (0.6,0.6) {\(P_2\)};
    \node[left ] at (0.55,0.23) {\(P_3\)};
    \node[below] at (0.48,0.09) {\(P_4\)};
     \node[right] at (0.75,0.67) {\(\mathfrak{L}_6\)};
   \node[above] at (0.6,0.7) {\(\mathfrak{L}_1\)};
   \node[right] at (0.75,0.58) {\(\mathfrak{L}_2\)};
    \node[right] at (0.74,0.405) {\(\mathfrak{L}_3\)};
   \node[right] at (0.75,0.26) {\(\mathfrak{L}_4\)};
   \node[right] at (0.75,0.18) {\(\mathfrak{L}_5\)};
   \node[white,right,font=\tiny] at (0.275,0.16) {\(m=M^{\lambda}\)};
    \node[white,right,font=\tiny] at (0.34,0.35) {\(m>M^{\lambda}\)};
\end{tikzpicture}
\caption{The case $M^{\lambda}=4$.}
\end{center}
\end{subfigure}
    \hfill
    \begin{subfigure}[b]{0.45\textwidth}
\begin{center}
\begin{tikzpicture}[scale=7]
    \draw[->] (0,0) -- (1,0) node[right] {\(\frac1p\)};
    \draw[->] (0,0) -- (0,1) node[above] {\(\frac1q\)};

     \draw[densely dotted] (0,0) -- (0.75,0.75);
     \draw[green,densely dotted] (0.6,0.6) -- (0.52,0.175);
     \draw[purple,densely dotted] (0.09,-0.02) -- (0.75,0.62);
      \draw[red,densely dotted] (0.255,-0.12) -- (0.75,0.32);
       \draw[blue,densely dotted] (0.35,-0.11) -- (0.75,0.57);
     \draw[densely dotted] (0,0) -- (0.75,0.15);
      \draw[densely dotted] (0,0) -- (0.75,0.26);
     \draw[densely dotted] (0.6,0.7) -- (0.6,0.05);

     \draw[thick] (0,0) -- (0.6,0.6);

    \fill[blue, opacity=0.8] (0,0) -- (0.6,0.6) -- (0.57,0.444) -- (0.14,0.028) -- cycle;
    \fill[cyan, opacity=0.8](0.14,0.028)--(0.57,0.444)--(0.52,0.18) --(0.47,0.093) -- cycle;

   \draw[thick] (0.6,0.6) circle (0.3pt);
   \draw[thick] (0.52,0.18) circle (0.3pt);
   \draw[thick]  (0.47,0.093) circle (0.3pt);

    \node[below left] at (0,0) {\(P_1\)};
     \node[right] at (0.6,0.6) {\(P_2\)};
    \node[above] at (0.52,0.2) {\(P_3\)};
    \node[below] at (0.49,0.09) {\(P_4\)};
     \node[right] at (0.75,0.67) {\(\mathfrak{L}_6\)};
   \node[above] at (0.6,0.7) {\(\mathfrak{L}_1\)};
   \node[right] at (0.75,0.57) {\(\mathfrak{L}_2\)};
    \node[right] at (0.74,0.35) {\(\mathfrak{L}_3\)};
   \node[right] at (0.75,0.26) {\(\mathfrak{L}_4\)};
   \node[right] at (0.75,0.18) {\(\mathfrak{L}_5\)};
    \node[white,right,font=\tiny] at (0.275,0.16) {\(m=M^{\lambda}\)};
    \node[white,right,font=\tiny] at (0.34,0.35) {\(m>M^{\lambda}\)};
\end{tikzpicture}
\caption{The case $M^{\lambda}=3$.}
\end{center}
 \end{subfigure}
\caption{For $m > M^{\lambda}$, the $L^p\rightarrow L^q$ boundedness region in part (1) of Theorem \ref{mainth2+} is represented by the blue region; for the critical case $m = M^{\lambda}$, it is given by the union of the blue and green regions.}
\end{center}
\end{figure}

Theorem \ref{apth2} gives necessary conditions for $\mathcal{M}_{loc,\lambda}$ to be $L^{p}\rightarrow L^{q}$ bounded.
\begin{theorem}\label{apth2}
Suppose that $\lambda \in R$ such that $\Gamma(\lambda)$ is of Type B. Then the following are necessary conditions for $\mathcal{M}_{loc,\lambda}$ to be $L^{p} \rightarrow L^{q}$ bounded.\\
 \textbf{(1)} For $c \neq 0$, $\frac{m/2+1}{p}-\frac{1}{q}-1 \le 0$, and $\frac 1q\leq \frac 1p\leq\frac 3q$. \\
\textbf{(2)}  For $c =0$, $\frac{m/2+1}{p}-\frac{m/2+1}{q}-1 \le 0$. Moreover,\\
\indent\textbf{ (i)} When $M^{\lambda} \ge 6$, $(\frac{1}{p},\frac{1}{q})$ belongs to the closed quadrilateral constructed by the lines $\mathfrak{L}_{1}$,  $\mathfrak{L}_{3}$, $\mathfrak{L}_{5}$, and the diagonal line. \\
\indent\textbf{ (ii)} When $M^{\lambda} = 5$, $(\frac{1}{p},\frac{1}{q})$ belongs to the closed pentagon  constructed by the lines $\mathfrak{L}_{1}$, $\mathfrak{L}_{2}$, $\mathfrak{L}_{3}$, $\mathfrak{L}_{5}$, and the diagonal line. \\
\indent\textbf{ (iii)} When $M^{\lambda} = 4$ or $M^{\lambda} =3$, $(\frac{1}{p},\frac{1}{q})$ belongs to the closed quadrilateral constructed by the lines $\mathfrak{L}_{1}$, $\mathfrak{L}_{2}$,  $\mathfrak{L}_{5}$, and the diagonal line.
\end{theorem}


Notice when $Z_{H\Phi} \cap Z_{\Phi} \subsetneqq Z_{H\Phi}$, there exists at least one of  $\{\Gamma(\lambda)\}_{\lambda\in R\cup\{\infty\}}$,  such that it belongs to Type B. We define
\[ M^{\Phi} := \max\{M^{\lambda}: \Gamma(\lambda) \text{ is of Type B}\}.\]
Then
\[\Delta_{ M^{\Phi}} = \cap_{\{\lambda: \Gamma(\lambda) \text{ is of Type B}\}} \Delta_{M^{\lambda}}. \]
Note that $\Delta_{M^{\Phi}} \subset \Delta_{0}$. Then by Theorem \ref{mainth2+}, Theorem \ref{apth2}, Proposition \ref{muphihphi} and Theorem \ref{mainth1+}, we can obtain the following theorem.

\begin{theorem}\label{mainth2}
Suppose that $Z_{H\Phi} \cap Z_{\Phi} \subsetneqq Z_{H\Phi}$, we have the following results. \\
(1) If $c=0$,  $\|\mathcal{M}_{loc}\|_{L^{p} \rightarrow L^{q}} < \infty$ if
\begin{equation}\label{Delta3}
(\frac{1}{p},\frac{1}{q}) \in \widetilde{\Delta_{3}}=  \{(\frac{1}{p},\frac{1}{q}) \in \Delta_{M^{\Phi}}: \frac{h_{\Phi}+1}{p} - \frac{h_{\Phi}+1}{q}- 1 <0\};
\end{equation}
(2) If $c \neq 0$, $\|\mathcal{M}_{loc}\|_{L^{p}\rightarrow L^{q}} < \infty$ if
\begin{equation}\label{Delta4}
(\frac{1}{p},\frac{1}{q}) \in \widetilde{\Delta_{4}}=  \{(\frac{1}{p},\frac{1}{q}) \in  \Delta_{M^{\Phi}}:  \frac{h_{\Phi}+1}{p}-\frac{1}{q}-1<0\}.
\end{equation}
(3) The region of $L^{p} \rightarrow L^{q}$ boundedness  in the  part (2) is sharp up to the endpoints. Moreover,  the region obtained in the part (1) admits no further extension (excluding extensions to the boundary), with the only possible exception occurring on the open segment connecting $P_{2}$ and $P_{3}$ (with $M:=M^{\Phi}$).
\end{theorem}

 In 2019, Dendrinos and Zimmermann \cite{DZ} established $L^p\rightarrow L^q$ estimates for averaging operators associated with mixed homogeneous polynomial hypersurfaces through oscillatory integral estimates. However, the $L^p\rightarrow L^q$ boundedness ranges they obtained were not optimal in certain cases. Recently, Schwend \cite{Schwend} improved these results to near optimality (except at the endpoints) using geometric methods such as Christ's method of refinement. In \cite{G1}, Greenblatt characterized $L^p$-improving properties of the averaging operator along real-analytic surfaces $\{(t_{1},t_{2},\Phi(t_{1},t_{2})): (t_{1},t_{2}) \in U\subset  \mathbb{R }^{2}\}$, where the determinant of the Hessian matrix of  $\Phi(t_{1},t_{2})$ plays a critical role.

Nevertheless, when considering local maximal operators, the presence of supremum exhibits additional challenges that prevent direct application of the methods from these previous works. In this paper, we overcome these obstacles and establish the $L^p\rightarrow L^q$ boundedness of local maximal operators $\mathcal{M}_{loc}$  by employing resolution of singularities techniques combined with local smoothing estimates. Based on the research presented in this paper, we will further investigate the $L^{p} \rightarrow L^{q}$  boundedness of local maximal operators associated with more general surfaces in our subsequent papers.


\subsection{Weighted $L^{p}$-estimate of the global maximal operator}
Define the global maximal operator  by
\[\mathcal{M}f(y):=\sup_{t >0} |A_{t}f(y)|,\]
where $A_{t}$ is given by (\ref{def1}). The corresponding weighted estimates follow from Theorem \ref{mainth1}, Theorem \ref{mainth2} and the methodology of sparse domination.  One can see \cite{LWZ} and the appendix in this paper for more details.

\begin{theorem}\label{mainth7}
Let $0< p < r <q$, $\omega \in A_{\frac{r}{p}} \cap RH_{\left(\frac{q}{r}\right)'}$, $\alpha=\max \left(\frac{1}{r-p}, \frac{q-1}{q-r} \right)$.
Then there holds the  weighted estimate
\[  \| \mathcal{M} \|_{L^{r}(\omega) \rightarrow L^{r}(\omega)} \le C \left([\omega]_{A_{\frac{r}{p}}} [\omega]_{RH_{\left(\frac{q}{r}\right)'}}\right)^\alpha \]
in each of the following cases: \\
(1) $c =0$, $Z_{H\Phi} \subset Z_{\Phi}$, and  $(\frac{1}{p}, \frac{1}{q}) \in \widetilde{\Delta_{1}}$ given in (\ref{Delta1}); \\
(2)  $c \neq 0$, $Z_{H\Phi} \subset Z_{\Phi}$, and $(\frac{1}{p}, \frac{1}{q}) \in \widetilde{\Delta_{2}}$ given in (\ref{Delta2}); \\
(3)  $c =0$, $Z_{H\Phi} \cap Z_{\Phi} \subsetneqq Z_{H\Phi}$, and $(\frac{1}{p}, \frac{1}{q}) \in \widetilde{\Delta_{3}}$ given in (\ref{Delta3}); \\
(4)  $c \neq 0$, $Z_{H\Phi} \cap Z_{\Phi} \subsetneqq Z_{H\Phi}$, and $(\frac{1}{p}, \frac{1}{q}) \in \widetilde{\Delta_{4}}$ given in (\ref{Delta4}).
\end{theorem}



We note that the unweighted $L^{p}$-estimates for the global maximal operators associated with mixed homogeneous hypersurfaces in $\mathbb{R}^{3}$ were established in \cite{IKM1}. There, the hypersurface is represented as the graph of $\Phi + c$, where $c \in \mathbb{R}$ and $\Phi$ is a mixed homogeneous function. It was shown that the global maximal operator is $L^{p}$-bounded for $p > h_{\Phi} \geq 2$. These results are optimal when $c \neq 0$, but the sharp bounds for $h_{\Phi} < 2$ remain an open question. When $\omega = 1$ and $c \neq 0$ in Theorem \ref{mainth7}, our results coincide with those of \cite{IKM1} for homogeneous functions $\Phi$ with $h_{\Phi} \geq 2$, and also extend to certain cases where $h_{\Phi} < 2$.

 Moreover, when $\omega = 1$ and $c = 0$ in Theorem \ref{mainth7}, we obtain optimal $L^{p}$-estimates for the global maximal operators associated with homogeneous hypersurfaces in $\mathbb{R}^{3}$. The $L^{p} \to L^{q}$ boundedness ($p \leq q$) for maximal operators related to mixed homogeneous hypersurfaces will be addressed in our subsequent work.

\subsection{Organization of this paper}
In order to prove  Theorem \ref{mainth2+}, we are reduced to  considering the $L^{p} \rightarrow L^{q}$ boundedness of  a local maximal operator $\widetilde{\mathcal{M}}_{loc}$ related to
 \begin{equation}\label{tildeS}
 \widetilde{S}= \{(r(1+s^{M}g(s)),rs,r^{d}+c): r \in [1,2], s \in (0,\epsilon_{0})\},  M \ge 3, d \ge 2,
 \end{equation}
where $g$ is a smooth function with $g(0) \neq 0$. The boundedness of $\widetilde{M}_{loc}$ is stated in Theorem \ref{mainth3}. We will prove Theorem \ref{mainth3} through Section \ref{proofofmainth3}.

Section \ref{proofofmainth} is devoted to obtain the $L^{p}\rightarrow L^{q}$ bounded result of $\mathcal{M}_{loc,\lambda}$ in Theorem \ref{mainth1+} and Theorem \ref{mainth2+}. In the proof of Theorem  \ref{mainth1+}, we can reduce the problem to considering the local maximal operator related to hypersurfaces with non-vanishing Gaussian curvature, then we apply the method developed in \cite{WS}  to finish the proof of Theorem \ref{mainth1+}. The proof of   Theorem \ref{mainth2+}  is based on Theorem \ref{mainth3}.

Theorem \ref{apth1} and Theorem \ref{apth2} will be proved in Section \ref{nececondition}. The geometric properties of the hypersurface $S_{\lambda}$ play an essential role in the proofs. In particular, during the proof of Theorem \ref{apth2},
we obtain the relationship $M^{\lambda} \le m$.

In Section \ref{proofofproperties}, we give the details for the proofs of Proposition \ref{muphihphi} and Proposition \ref{level set type}.

In the appendix,  we will first
  explain how to prove the H\"{o}lder continuity property, which is a crucial step in obtaining weighted estimates for the global maximal operator. Secondly, for completeness, we give a complementary discussion of the proof of Theorem \ref{mainth1+}.

We will omit the proofs of Theorems \ref{mainth1} and \ref{mainth2}. Parts (1) and (2) of Theorem \ref{mainth1} follow directly from Theorem \ref{mainth1+} and Proposition \ref{muphihphi}. Parts (1) and (2) of Theorem \ref{mainth2} can be immediately obtained from Theorem \ref{mainth2+}, Proposition \ref{muphihphi}, and Theorem \ref{mainth1+}. Since the $L^{p} \rightarrow L^{q}$ boundedness of $\mathcal{M}_{loc}$ implies that of $\mathcal{M}_{loc,\lambda}$, part (3) of Theorems \ref{mainth1} and \ref{mainth2} follows from Theorems \ref{apth1} and \ref{apth2}, respectively.

\textbf{Conventions}: Throughout this article, we shall use the notation $A\ll B$, which means that  there is a sufficiently large constant $G$, which is much larger than $1$ and does not depend on the relevant parameters arising in the context in which
the quantities $A$ and $B$ appear, such that $G A\leq B$. Conversely, $A\gg B$ means $ A\ge GB$. By
$A\lesssim B$ we mean that $A \le CB $ for some constant $C$ independent of the parameters related to  $A$ and $B$.
Given $\mathbb{R}^{n}$, we write $B(0,1)$ instead of the unit ball $B^{n}(0,1)$ in $\mathbb{R}^{n}$ centered at the origin for short, and the same notation is valid for $B(x_{0},r)$.
For $i,j=1,2$, we denote  $\partial_{x_{i}}\partial_{x_{j}}\Phi(x_{1},x_{2})$ by $\partial_{ij}\Phi(x_{1},x_{2})$,
and $\partial_{x_{i}}\Phi(x_{1},x_{2})$  by $\partial_{i}\Phi(x_{1},x_{2})$. The symbol $H\Phi(x_{1},x_{2})$ denotes the determinant of the Hessian matrix of $\Phi$, i.e.,
\[H\Phi(x_{1},x_{2}) =\partial_{11}\Phi(x_{1},x_{2}) \partial_{22}\Phi(x_{1},x_{2})-[\partial_{12}\Phi(x_{1},x_{2})]^{2}.\]
For any positive integer $k > 2$, by $\partial^{(k)}_{i}\Phi(x_{1},x_{2})$ we mean the  $k$-th order partial derivative of $\Phi(x_{1},x_{2})$ with respect to $x_{i}$, where $i=1,2$, and  for a one-variable function $\Psi(s)$, by $\Psi^{(k)}(s)$ we mean the $k$-th order derivative of $\Psi(s)$.

\section{$L^{p}\rightarrow L^{q}$ estimate for $\widetilde{\mathcal M}_{loc}$}\label{proofofmainth3}

\begin{theorem}\label{mainth3}
Let
\begin{equation}\label{dmu}
\widehat{d\mu}(\xi)= \int_{\mathbb{R}^{2}}e^{-i\bigl(r(1+s^{M}g(s))\xi_{1}+rs\xi_{2}+r^{d}\xi_{3}\bigl)} \psi(s) \rho(r)dsdr,
\end{equation}
where $g$ is a smooth function with $g(0) \neq 0$, $\rho$ and $\psi$ are smooth cutoff functions, $\mathrm{supp}$ $ \rho \subset [1,2]$ and  $\mathrm{supp}$ $\psi \subset (0,\epsilon_{0})$ with $\epsilon_{0}$ sufficiently small, $d \ge 2$ and $M \ge 3$ are positive integers. Given the local maximal function
\begin{equation}\label{def4}
\widetilde{\mathcal{M}}_{loc}f(y):=\sup_{t\in[1,2]}| \widetilde{A_{t}}f(y)|, \nonumber
 \end{equation}
 where
 \begin{equation}
  \widetilde{A_{t}}f(y) =  \int_{\mathbb{R}^{3}} e^{i \bigl(y \cdot \xi  + tc\xi_{3}\bigl)} \widehat{d\mu}(t\xi) \hat{f}(\xi)d\xi \nonumber
 \end{equation}
with some constant $c \ge 0$, then we have the following estimate
\begin{equation}\label{levelsetestimate}
 \|\widetilde{\mathcal{M}}_{loc}f  \|_{L^{q}} \lesssim (1+c^{\frac{1}{q}}) \|f\|_{L^{p}}, \nonumber
\end{equation}
provided that $(\frac{1}{p}, \frac{1}{q})$ belongs to $\Delta_{M}$, which is defined by Definition \ref{DeltaM}.
\end{theorem}

$\bullet$ \textbf{Sketch of the proof:} In the proof of Theorem \ref{mainth3}, we perform a dyadic decomposition and localize the frequencies $\xi=(\xi_{1},\xi_{2},\xi_{3})$ where $|\xi| \sim 2^{j} \gg 1$. For fixed $j\gg 1$, we consider $\widehat{d\mu}(\xi)$ defined in (\ref{dmu}) when $\xi\in \{\xi:|\xi| \sim 2^{j}\}$. By some preliminary reductions, we reduce the problem to considering the case when $|\xi_{1}| \sim |\xi_{3}| \ge |\xi_{2}|$. Since $r\in [1,2]$, the phase function in $\widehat{d\mu}(t\xi)$ along the $r$-direction has non-degenerate critical point. Employing stationary phase method in the $r$-direction, we get the contribution term (\ref{r-mainpart}) for $\widehat{d\mu}(\xi)$. Furthermore,
in (\ref{r-mainpart}), to  control the integration in $|s| \ll 1$, we perform the dyadic decomposition into $s$-intervals of length $2^{-l}$, $l \ge l_{0} \gg 1$. For fixed $l$, we further decompose the variable $\xi_{2}$ according to the oscillation in $\widehat{d\mu_{l}}(\xi)$ defined by (\ref{dmul}) below. Finally, we are reduced to estimating the truncated maximal operator $\mathcal{M}_{j,l,k}$, see (\ref{Mjlk}). The $L^{p} \rightarrow L^{q}$ estimates of $\mathcal{M}_{j,l,k}$ relies on the kernel estimates in Theorem \ref{kernel estimate3}. We will first apply Theorem \ref{kernel estimate3} to finish the proof of Theorem \ref{mainth3} in Subsection \ref{mainproof}, then state the proof of Theorem \ref{kernel estimate3} in Subsection \ref{proofofmainth4}.

\subsection{Proof of Theorem \ref{mainth3}}\label{mainproof}
We choose a smooth cut-off function  $\beta$ supported in $[1/2,2]$ such that
\begin{equation}\label{LPDecoposition}
\sum_{j \in \mathbb{Z}}\beta(2^{-j}r)=1 \hspace{0.5cm}\textmd{for each} \hspace{0.5cm} r>0.
\end{equation}
Then we perform a dyadic decomposition to obtain that
\[\widetilde{A_{t}}f(y)= \sum_{j \ge 1} \widetilde{A_{t,j}}f(y)+ \widetilde{A_{t,0}}f(y), \quad y \in \mathbb{R}^{3},\]
 where the averaging operator
\begin{equation}\label{annulusj}
\widetilde{A_{t,j}}f(y)= \int_{\mathbb{R}^{3}} e^{i \bigl(y \cdot \xi  + tc\xi_{3} \bigl)} \widehat{d\mu}(t\xi) \beta(2^{-j}|\xi|) \hat{f}(\xi)d\xi,
\end{equation}
for $j \ge 1$,  and
\[\widetilde{A_{t,0}}f(y)= \int_{\mathbb{R}^{3}} e^{i \bigl(y \cdot \xi  + tc\xi_{3}\bigl)} \widehat{d\mu}(t\xi) \biggl(\sum_{j \le 0}\beta(2^{-j}|\xi|)\biggl) \hat{f}(\xi)d\xi .\]
 For each $j \ge 0$, we define  the local maximal function
\[\widetilde{\mathcal{M}}_{loc,j}f(y): = \sup_{t \in [1,2]} |\widetilde{A_{t,j}}f(y)|.\]
 It suffices to prove that for each $(\frac{1}{p}, \frac{1}{q}) \in \Delta_{M}$, $M\geq 3$, there exists $\epsilon(p,q)>0$ such that
\begin{equation}\label{goal1.4}
 \|\widetilde{\mathcal{M}}_{loc,j}  \|_{L^{p} \rightarrow L^{q}} \lesssim (1+c^{\frac{1}{q}}) 2^{-\epsilon(p,q)j}.
\end{equation}
We notice that when $j \lesssim 1$, (\ref{goal1.4}) can be proved for each $q \ge p \ge 1$ by integration by parts with respect to $\xi$ and Young's inequality. We will prove (\ref{goal1.4}) for $j \gg 1$ in what follows.

We choose a suitable smooth cut-off function $\tilde{\chi}$ with sufficiently small compact support, which vanishes near the origin and is identically one near $1$,  and decompose
\begin{equation}\label{mainterm}
\widetilde{A_{t,j}}f(y) =  A_{t,j} f(y) + R_{t,j}f(y),
\end{equation}
where
\[ A_{t,j} f(y)= \int_{\mathbb{R}^{3}} e^{i \bigl(y \cdot \xi  + tc\xi_{3} \bigl)} \widehat{d\mu}(t\xi) \beta(2^{-j}|\xi|) \tilde{\chi}\biggl(\frac{|\xi_{1}| + |\xi_{2}|}{|\xi_{3}|}\biggl) \hat{f}(\xi)d\xi,\]
and
\[ R_{t,j} f(y)= \int_{\mathbb{R}^{3}} e^{i \bigl(y \cdot \xi  + tc\xi_{3} \bigl)} \widehat{d\mu}(t\xi) \beta(2^{-j}|\xi|) \biggl[1-\tilde{\chi}\biggl(\frac{|\xi_{1}| + |\xi_{2}|}{|\xi_{3}|}\biggl)\biggl] \hat{f}(\xi)d\xi.\]
We define the local maximal function
\[\mathcal{M}_{loc,j}f(y) = \sup_{t\in[1,2]}|A_{t,j}f(y)|\]
and
\[\mathcal{M}^{0}_{loc,j}f(y) = \sup_{t\in[1,2]}|R_{t,j}f(y)|.\]
Then
\begin{equation}
\|\widetilde{\mathcal{M}}_{loc,j}   \|_{L^{p} \rightarrow L^{q}} \lesssim \| \mathcal{M}_{loc,j}   \|_{L^{p} \rightarrow L^{q}}  + \| \mathcal{M}^{0}_{loc,j}  \|_{L^{p} \rightarrow L^{q}} .
\end{equation}

We first consider   $\mathcal{M}_{loc,j}^{0}$. Let
\begin{equation}\label{phixi}
\phi_{\xi}(r,s) =  r(1+s^{M}g(s))\xi_{1} +  rs\xi_{2} +  r^{d}\xi_{3}, \hspace{0.2cm}r\in [1,2] \hspace{0.1cm}\textmd{and}\hspace{0.1cm}  s\in (0,\epsilon_0).
\end{equation}
Then
\[\nabla_{r,s} \phi_{\xi}(r,s) = ( s\xi_{2} +  (1+s^{M}g(s))\xi_{1} + d r^{d-1}\xi_{3}, r \xi_{2}+r (Ms^{M-1}g(s) +  s^{M}g^{\prime}(s))\xi_{1}).\]
If $|\xi_{3}| \gg |\xi_{1}| + |\xi_{2}|$, then $|\partial_{r}\phi_{\xi}(r,s)| \sim |\xi_{3}| \sim |\xi|$. If $|\xi_{3}| \ll |\xi_{1}| + |\xi_{2}|$ and $|\xi_{2}| \le |\xi_{1}|$, then $|\partial_{r}\phi_{\xi}(r,s)| \sim |\xi_{1}| \sim |\xi|$. If $|\xi_{3}| \ll |\xi_{1}| + |\xi_{2}|$ and $|\xi_{2}| > |\xi_{1}|$, then $|\partial_{s}\phi_{ \xi}(r,s)| \sim |\xi_{2}| \sim |\xi|$.
Hence, there holds
\begin{equation}\label{rapiddecay}
|\widehat{d\mu}(t\xi)| \le \frac{C_{N}}{(1+t|\xi|)^{N}}, \quad \forall N \in \mathbb{N}^{+}.
\end{equation}
  Then integrating by parts implies that
\[|R_{t,j}f(y)| \le C_{N}2^{-Nj} \int_{\mathbb{R}^{3}}\frac{|f(z)|}{(1+2^{j}|y-z-(0,0,ct)|)^{N}}dz, \quad \forall N \in \mathbb{N}^{+}. \]
Therefore, Sobolev  embedding and Young's inequality yield that
\begin{equation}\label{remainderterm}
 \|\mathcal{M}_{loc,j}^{0}f  \|_{L^{q}} \lesssim (1 + c^{\frac{1}{q}}) 2^{-jN} \|f\|_{L^{p}}, \quad \forall N \in \mathbb{N}^{+}
\end{equation}
whenever $q \ge p \ge 1$.  So we are reduced to investigating the main contribution part $\mathcal{M}_{loc,j}$, where we localize the frequencies to the following range
\begin{equation}\label{Frequency1}
2^j \sim|\xi_{3}| \sim |\xi_{1}| + |\xi_{2}|.
\end{equation}

Notice that in the phase function $\phi_{\xi}(r,s)$ denoted by (\ref{phixi}),  $s^{M}g(s)$ is a small perturbation of $s^{M}$. We first compute the critical points for the following function
\[\varphi_{\xi}(r,s) =  r(1+s^{M} )\xi_{1} +  rs\xi_{2} +  r^{d}\xi_{3}, \quad r \sim 1, |s| \ll 1. \]
We have
\[\partial_{r} \varphi_{\xi}(r,s) =  s\xi_{2} +  (1+s^{M})\xi_{1} + d r^{d-1}\xi_{3}, \]
\[\frac{1}{r} \cdot \partial_{s} \varphi_{\xi}(r,s)=  \xi_{2}+ Ms^{M-1}  \xi_{1}.\]
Then  $\nabla_{r,s} \varphi_{\xi}(r,s) = (0,0)$ provided that $\xi_{2} = -Ms^{M-1}\xi_{1}$, and  $\xi_{1}+dr^{d-1}\xi_{3} = -s\xi_{2} + s^{M}\xi_{1} = (1-M)s^{M}\xi_{1}  $.
Hence,
the phase function $\phi_{\xi}(r,s)$ can have  critical points only if
\[|\xi_{1}| \sim |\xi_{3}|, \quad |\xi_{2}| \sim s^{M-1} |\xi_{1}| . \]
Based on this observation, we consider (Case I): $|\xi_{1}| < |\xi_{2}| $ and (Case II): $|\xi_{1}| \ge |\xi_{2}| $, respectively.

\textbf{ Case I:} $|\xi_{1}| < |\xi_{2}| $. Combining (\ref{Frequency1}) with this inequality, the relation $|\xi_2| \sim |\xi_{3}| $ is obtained. Subsequently, integration by parts in $s$ leads to (\ref{rapiddecay}). Then the corresponding local maximal operator can be handled by the method to obtain the inequlity (\ref{remainderterm}).

\textbf{ Case II:} $|\xi_{1}| \ge |\xi_{2}| $. Combining this with (\ref{Frequency1}), we obtain
\begin{equation}\label{localfre1}
2^j \sim|\xi_{3}| \sim  |\xi_{1}|  \ge |\xi_{2}|.
\end{equation}
We concentrate  on  this  case in what follows.

Now we apply stationary phase method in the variable $r$. The equation
\[\partial_{r}\phi_{\xi}(r,s)=0\]
has a unique solution
\[r_{c}= \biggl(-\frac{[1+s^{M}g(s)]\xi_{1}  + s\xi_{2}}{d\xi_{3}} \biggl)^{\frac{1}{d-1}}. \]
We denote the phase function
\[\phi_{2}(s,\xi) = \phi_{\xi}(r,s)|_{r=r_{c}}=c_{d} \xi_{3}  \biggl(-\frac{[1+s^{M}g(s)]\xi_{1}  + s\xi_{2}}{d\xi_{3}} \biggl)^{\frac{d}{d-1}}. \]
 Then we have
\begin{equation}\label{dmiu}
\widehat{d\mu}(\xi)= 2^{-\frac{j}{2}} \int_{\mathbb{R}}e^{-i\phi_{2}(s,\xi) }a_j(s,\xi) \psi(s)ds + R_{j}(\xi).
\end{equation}

 We first consider the remainder term $R_{j}$ in (\ref{dmiu}). For each non-negative multi-index $\alpha$ and positive integer $N$,  $R_{j}$ satisfies the estimate
\[| \partial_{\xi}^{\alpha}R_j(t\xi)|\leq C 2^{-Nj}, \hspace{0.2cm}t\in [1,2].\]
Hence, the corresponding local maximal operator can be handled by the method to obtain the inequality (\ref{remainderterm}).

 Next we only consider the main contribution term of $\widehat{d\mu}(\xi)$ given by
\begin{equation}\label{r-mainpart}
2^{-\frac{j}{2}} \int_{\mathbb{R}}e^{-i\phi_{2}(s,\xi) }a_j(s,\xi) \psi(s)ds,
\end{equation}
  where the  symbol $a_j$ satisfies for arbitrary non-negative multi-index $\alpha$ and integer $\beta$,
\begin{equation}\label{symbol}
|\partial_s^{\beta}\partial_{\xi}^{\alpha}a_j(s,\xi)|\leq C 2^{-j|\alpha|}.
\end{equation}

In order to  control the integration over $|s| \ll 1$, we now perform a dyadic decomposition into $s$-intervals of length $2^{-l}$, $l \ge l_{0} \gg 1$,  scale in $s$ by $2^{-l}$, and scale in $\xi_{2}$ by $2^{-(M-1)l}$. Then the bounded estimate for the  corresponding maximal operator $\mathcal{M}_{loc,j}$ can be reduced to obtaining the $L^{p}\rightarrow L^{q}$ boundedness for the maximal operator $\mathcal{M}_{j,l}$, i.e.,
\begin{equation}\label{frequencydecompose1}
 \|\mathcal{M}_{loc,j}   \|_{L^{p}\rightarrow L^{q}} \le \sum_{l \ge l_0} 2^{-l}2^{(M-1)l(\frac{1}{q}-\frac{1}{p})}  \| \mathcal{M}_{j,l}
 \|_{L^{p}\rightarrow L^{q}}.
\end{equation}
Here, the local maximal operator $\mathcal{M}_{j,l}$ is defined by
\begin{equation}\label{Mjl}
\mathcal{M}_{j,l}f(y)= \sup_{t \in [1,2]} |A_{t,j,l}f(y)|,
\end{equation}
and
\[ A_{t,j,l} f(y)= \int_{\mathbb{R}^{3}} e^{i \bigl(y \cdot \xi  + tc\xi_{3}\bigl)} \widehat{d\mu_{l}}(t\xi) \chi_{1} \biggl(\frac{|\xi_{1}| }{2^{j}} \biggl)  \chi_{1} \biggl(  \frac{\xi_{1}}{ \xi_{3} }  \biggl) \chi_{0}(2^{-(M-1)l-j}|\xi_{2}|) \hat{f}(\xi)d\xi,\]
$\chi_{0}$ and $\chi_{1}$ are smooth cut-off functions, $\chi_{0}$ is supported in $[-1,1]$, whereas $\chi_{1}$ is supported in $[1,2]$,
\begin{equation}\label{dmul}
\widehat{d\mu_{l}}(\xi)= 2^{-\frac{j}{2}} \int_{\mathbb{R}}e^{-i\phi_{3}(s,\xi) } a_j(2^{-l}s,\xi_1,2^{-{(M-1)l}}\xi_2,\xi_3)\chi_{1}(s)ds,
\end{equation}
\[\phi_{3}(s,\xi) = \xi_{3} \biggl(- \frac{\xi_{1}}{d\xi_{3}}- \delta^{M} \cdot \frac{ s^{M}g(\delta s) \xi_{1}  + s\xi_{2}}{d\xi_{3}} \biggl)^{\frac{d}{d-1}}, \quad \delta  = 2^{-l}. \]
We notice that the localized frequencies defined in (\ref{localfre1}) has transformed into
\begin{equation}\label{region1}
2^j \sim|\xi_{3}| \sim  |\xi_{1}|  \ge2^{-(M-1)l}|\xi_{2}|.
\end{equation}
We may assume that $g(0)=1$, it is easy to get $\partial_{s}\phi_{3}(s,\xi)=0$ if  $s=s_{c}:= (-\frac{\xi_{2}}{M\xi_{1}})^{\frac{1}{M-1}} + \tilde{R}(\frac{\xi_{2}}{\xi_{1}},\delta) $. The remainder term $\tilde{R}(\frac{\xi_{2}}{\xi_{1}},\delta)$ is homogeneous of degree zero in $\xi$ and has at least $1$ power of $\delta$. Combining with $s \in [1,2]$, the phase function $\phi_{3}(s,\xi)$ can have critical points  only if $|\xi_{1}| \sim |\xi_{2}|$, and
\[\bigl|\partial^{2}_{s}\phi_{3}(s,\xi)|_{s=s_{c}} \bigl| \sim  2^{j-Ml}. \]
Hence, we will consider   $|\xi_{2}| \gg |\xi_{1}| $,   $|\xi_{2}| \sim |\xi_{1}| $ and   $|\xi_{2}| \ll |\xi_{1}| $, respectively. More concretely,   we take a  dyadic decomposition
\begin{equation}
\widehat{d\mu_{l}}(t\xi) = \sum_{k \in \mathbb{Z}} \beta(2^{-k-j}|\xi_{2}|)\widehat{d\mu_{l}}(t\xi),
\end{equation}
where $\beta$ satisfies (\ref{LPDecoposition}).
By (\ref{region1}), we only need to consider the case  $k \le (M-1)l$. We choose
\begin{equation}\label{defN1}
N_{1} = \sup_{s \in [1,2]}   \lfloor \log|Ms^{M-1}g(\delta s) +s^{M} \delta g^{\prime}(\delta s)|  \rfloor +10,
\end{equation}
$\lfloor \cdot \rfloor $ is the floor function, and
\begin{equation}\label{defN0}
N_{0} = \sup_{s \in [1,2]}   \lfloor \log|Ms^{M-1}g(\delta s) +s^{M} \delta g^{\prime}(\delta s) |^{-1} \rfloor +10.
\end{equation}
We will consider the contribution of each $\beta(2^{-k-j}|\xi_{2}|)\widehat{d\mu_{l}}(t\xi)$ in the following three cases. \textbf{Case (A)}:  $N_{1}<k \le (M-1)l$; \textbf{Case (B)}: $-N_{0}\le k \le N_{1}$.\textbf{ Case (C)}: $k<-N_{0}$.

\textbf{Case (A):} $N_{1}<k \le (M-1)l$. In this case, we have
\[|\xi_{1}| \sim |\xi_{3}|  \sim 2^{j},\quad |\xi_{2}| \sim 2^{j+k}. \]
 Notice that in the expression of $\phi_{3}(s,\xi)$, we have
\[\biggl|\delta^{M} \cdot \frac{ s^{M}g(\delta s) \xi_{1}  + s\xi_{2}}{d\xi_{3}}\biggl| \lesssim 2^{-l},\]
by Taylor expansion,
\begin{align}
\phi_{3}(s,\xi) &=  \xi_{3} \biggl(- \frac{\xi_{1}}{d\xi_{3}} \biggl)^{\frac{d}{d-1}}
  + \delta^{M} \xi_{3} \biggl(\bigl[-  \frac{  s^{M}g(\delta s) \xi_{1}  + s\xi_{2}}{d\xi_{3}} \bigl] \cdot \frac{d}{d-1} \bigl(- \frac{\xi_{1}}{d\xi_{3}} \bigl)^{\frac{1}{d-1}} +  \delta^{M} R_{\delta}\bigl(s, \frac{\xi_{1}}{\xi_{3}}, \frac{\xi_{2}}{\xi_{3}}\bigl) \biggl),\nonumber
\end{align}
where we have put all the remainder term into $ \delta^{M} R_{\delta}\bigl(s, \frac{\xi_{1}}{\xi_{3}}, \frac{\xi_{2}}{\xi_{3}}\bigl)$.  By the mean value theorem, we can obtain that
\begin{align}\label{remaindercond1}
&\delta^{M} \biggl|\partial_{s} R_{\delta}\bigl(s, \frac{\xi_{1}}{\xi_{3}}, \frac{\xi_{2}}{\xi_{3}}\bigl) \biggl| \nonumber\\
&\lesssim    \biggl|\frac{ [Ms^{M-1}g(\delta s)+ s^{M}\delta  g^{\prime}(\delta s)]\xi_{1}  +  \xi_{2}}{d\xi_{3}}\biggl| \cdot \biggl| \biggl(- \frac{\xi_{1}}{d\xi_{3}}- \delta^{M} \cdot \frac{ s^{M}g(\delta s) \xi_{1}  + s\xi_{2}}{d\xi_{3}} \biggl)^{\frac{1}{d-1}}- \biggl(- \frac{\xi_{1}}{d\xi_{3}} \biggl)^{\frac{1}{d-1}} \biggl| \nonumber\\
&\lesssim   \delta^{M}  \biggl|\frac{ [Ms^{M-1}g(\delta s)+ s^{M}\delta  g^{\prime}(\delta s)]\xi_{1}  +  \xi_{2}}{d\xi_{3}}\biggl| \cdot \biggl|\frac{ s^{M}g(\delta s) \xi_{1}  + s\xi_{2}}{d\xi_{3}}\biggl|.
\end{align}
We denote
\begin{equation}\label{phi4}
\phi_{4}(s,\xi)= \delta^{M} \xi_{3} \biggl(\bigl[-  \frac{  s^{M}g(\delta s) \xi_{1}  + s\xi_{2}}{d\xi_{3}} \bigl] \cdot \frac{d}{d-1} \bigl(- \frac{\xi_{1}}{d\xi_{3}} \bigl)^{\frac{1}{d-1}} +  \delta^{M} R_{\delta}\bigl(s, \frac{\xi_{1}}{\xi_{3}}, \frac{\xi_{2}}{\xi_{3}}\bigl) \biggl).
\end{equation}
Then we have
\begin{align}\label{mujlk}
 \beta (2^{-j-k}|\xi_{2}|) \widehat{d\mu_{l}}(t\xi)&= 2^{-\frac{j}{2}} e^{-it\xi_{3}\bigl(-\frac{\xi_{1}}{d\xi_{3}}\bigl)^{\frac{d}{d-1}}}\beta(2^{-j-k}|\xi_{2}|) \nonumber\\
 &\quad \times  \int_{\mathbb{R}}e^{-it \phi_{4}(s,\xi) } a_j(2^{-l}s,t\xi_1,t2^{-{(M-1)l}}\xi_2,t\xi_3) \chi_{1}(s)ds.
 \end{align}
In order to estimate
\[  \beta(2^{-j-k}|\xi_{2}|) \int_{\mathbb{R}}e^{-it \phi_{4}(s,\xi) } a_j(2^{-l}s,t\xi_1,t2^{-{(M-1)l}}\xi_2,t\xi_3) \chi_{1}(s)ds,\]
 we need to consider the phase function $\phi_{4}(s,\xi)$. Let $\xi_{1} =  2^{j}\xi^{\prime}_{1}$, $\xi_{2} = 2^{j+k}\xi^{\prime}_{2}$,  $\xi_{3} = 2^{j}\xi^{\prime}_{3}$,  $\xi^{\prime}_{j,k}=2^{j}(\xi_{1}^{\prime},2^{k}\xi_{2}^{\prime},\xi_{3}^{\prime})$.  We have
\begin{align}
\phi_{4}(s,\xi^{\prime}_{j,k}) &= 2^{j} \delta^{M} \xi^{\prime}_{3} \biggl(\bigl[-  \frac{  s^{M}g(\delta s) \xi^{\prime}_{1}  + 2^{k} s \xi^{\prime}_{2}}{d\xi^{\prime}_{3}} \bigl] \cdot \frac{d}{d-1} \bigl(- \frac{\xi^{\prime}_{1}}{d\xi^{\prime}_{3}} \bigl)^{\frac{1}{d-1}} +  \delta^{M} R_{\delta}\bigl(s, \frac{\xi^{\prime}_{1}}{\xi^{\prime}_{3}}, \frac{2^{k}\xi^{\prime}_{2}}{\xi^{\prime}_{3}}\bigl) \biggl) \nonumber\\
&= 2^{j+k}\delta^{M} \xi^{\prime}_{3} \biggl(\bigl[-  \frac{ 2^{-k}s^{M}g(\delta s) \xi^{\prime}_{1}  + s\xi^{\prime}_{2}}{d\xi^{\prime}_{3}} \bigl] \cdot \frac{d}{d-1} \bigl(- \frac{\xi^{\prime}_{1}}{d\xi^{\prime}_{3}} \bigl)^{\frac{1}{d-1}}   +  2^{-k}\delta^{M} R_{\delta}\bigl(s, \frac{\xi^{\prime}_{1}}{\xi_{3}^{\prime}}, \frac{2^{k}\xi^{\prime}_{2}}{\xi_{3}^{\prime}}\bigl) \biggl). \nonumber
\end{align}
Then
\begin{equation}
|\partial_{s} \phi_{4} (s,\xi^{\prime}_{j,k})| \sim 2^{j+k-Ml}  \biggl|     \frac{   \xi^{\prime}_{2}}{ \xi^{\prime}_{1}}   +   2^{-k} \bigl[  Ms^{M-1}g(\delta s) +s^{M}\delta g^{\prime}(\delta s)   \bigl]   + 2^{-k}\delta^{M}  \partial_{s}R_{\delta}\bigl(s, \frac{\xi^{\prime}_{1}}{\xi_{3}^{\prime}}, \frac{2^{k}\xi^{\prime}_{2}}{\xi_{3}^{\prime}}\bigl)  \biggl|. \nonumber
\end{equation}
Since $|\xi_{i}^{\prime}| \sim 1$ for $i=1,2,3$, by (\ref{remaindercond1}), there holds
\[ \biggl|2^{-k}\delta^{M}  \partial_{s}R_{\delta}\bigl(s, \frac{\xi^{\prime}_{1}}{\xi_{3}^{\prime}}, \frac{2^{k}\xi^{\prime}_{2}}{\xi_{3}^{\prime}}\bigl) \biggl| \lesssim 2^{k}\delta^{M}.   \]
Combining with $2^{k} \delta^{M} \le 2^{-l} \ll 1$,  $k >N_{1}$, and $|\frac{\xi^{\prime}_{2}}{\xi^{\prime}_{1}}| \sim 1$, it follows that
\begin{equation}\label{lowerbound1}
|\partial_{s} \phi_{4} (s,\xi^{\prime}_{j,k})| \sim 2^{j+k-Ml}.
\end{equation}

Depending on whether the oscillation factor can be derived from the estimate of
$\beta (2^{-j-k}|\xi_{2}|) \cdot$ $\widehat{d\mu_{l}}(t\xi)$, we divide the analysis into the following two subcases.

$\bullet$ \textbf{(A-i)} When $j >Ml$, $N_{1}<k \le (M-1)l$, or when $j \le Ml$, $Ml-j < k \le (M-1)l$,
we denote
\begin{equation}\label{ajlk}
 \beta(2^{-j-k}|\xi_{2}|) \int_{\mathbb{R}}e^{-it \phi_{4}(s,\xi) } a_j(2^{-l}s,t\xi_1,t2^{-{(M-1)l}}\xi_2,t\xi_3) \chi_{1}(s)ds:=a_{j,l,k} (\xi,t).
\end{equation}
According to (\ref{lowerbound1}),  integration by parts in $s$ implies that
\begin{equation}\label{symbolb1}
|D_{\xi^{\prime}}^{\alpha} a_{j,l,k} (\xi^{\prime}_{j,k},t)| \lesssim  2^{-(j+k-Ml)N}
\end{equation}
for each multi-index $\alpha$ and any positive integer $N$.  Then combining with (\ref{mujlk}), we have
\begin{align}
&\beta(2^{-j-k}|\xi_{2}|) \widehat{d\mu_{l}}(t\xi)   =   2^{-\frac{j}{2}} e^{-it\xi_{3}\bigl(-\frac{\xi_{1}}{d\xi_{3}}\bigl)^{\frac{d}{d-1}}}   a_{j,l,k}(\xi,t). \nonumber
\end{align}

$\bullet$ \textbf{(A-ii)} When $j \le Ml$ and $k \le Ml-j$,  then $2^{j+k-Ml} \le 1$. Therefore,
 \begin{align}
&\sum_{ N_{1}< k \le Ml-j}\beta(2^{-j-k}|\xi_{2}|) \widehat{d\mu_{l}}(t\xi) \nonumber\\ &= 2^{-\frac{j}{2}} e^{-it\xi_{3}\bigl(-\frac{\xi_{1}}{d\xi_{3}}\bigl)^{\frac{d}{d-1}}}
   \biggl(\sum_{N_{1} < k \le Ml-j} \beta(2^{-j-k}|\xi_{2}|) \int_{\mathbb{R}}e^{-it \phi_{4}(s,\xi) } a_j(2^{-l}s,t\xi_1,t2^{-{(M-1)l}}\xi_2,t\xi_3) \chi_{1}(s)ds\biggl), \nonumber
\end{align}
and the term
\[\sum_{N_{1} < k \le Ml-j} \beta(2^{-j-k}|\xi_{2}|) \int_{\mathbb{R}}e^{-it \phi_{4}(s,\xi) } a_j(2^{-l}s,t\xi_1,t2^{-{(M-1)l}}\xi_2,t\xi_3) \chi_{1}(s)ds\]
can be put into the amplitude.


\textbf{Case (B)}: $-N_{0}\le k \le N_{1}$. We have
\[|\xi_{1}| \sim |\xi_{3}|  \sim 2^{j}, \quad \quad 2^{-N_{0}+j} \le |\xi_{2}| \le 2^{N_{1}+j} .\]
  In this case, the equation $\partial_{s}\phi_{3}(s,\xi) =0$ has a unique solution $s_{c}= (-\frac{\xi_{2}}{M\xi_{1}})^{\frac{1}{M-1}} + \tilde{R}(\frac{\xi_{2}}{\xi_{1}},\delta) $.

$\bullet$ \textbf{(B-i)} When $j>Ml$, we rewrite
\begin{align*}
&\beta(2^{-j-k}|\xi_{2}|)\widehat{d\mu_{l}}(t\xi) \nonumber\\
&=   2^{-\frac{j}{2}}  e^{-it\xi_{3} \bigl(-\frac{\xi_{1}}{d\xi_{3}}\bigl)^{d/(d-1)} }\beta(2^{-j-k}|\xi_{2}|)
\int_{\mathbb{R}}e^{-it\phi_{4}(s,\xi) } a_j(2^{-l}s,\xi_1,2^{-{(M-1)l}}\xi_2,\xi_3)\chi_{1}(s)ds,
\end{align*}
we apply the analysis of stationary phase in $s$ to obtain
\[ \beta(2^{-j-k}|\xi_{2}|) \int_{\mathbb{R}}e^{-it\phi_{4}(s,\xi) } a_j(2^{-l}s,\xi_1,2^{-{(M-1)l}}\xi_2,\xi_3)\chi_{1}(s)ds=  e^{-it\phi_{4}(s_{c},\xi)}  \widetilde{a_{j,l,k}}(\xi,t) + \widetilde{r_{j,l,k}}(\xi,t),\]
where
\begin{equation}\label{symbolb2}
|D^{\alpha}_{\xi^{\prime}}\widetilde{a_{j,l,k}}(2^{j}\xi^{\prime},t)| \lesssim 2^{-\frac{j-Ml}{2}}, \quad \xi^{\prime}=2^{-j}\xi,
\end{equation}
\begin{equation}\label{symbolb2+}
|D^{\alpha}_{\xi^{\prime}}\widetilde{r_{j,l,k}}(2^{j}\xi^{\prime},t)| \lesssim 2^{-(j-Ml)N},  \quad  \forall N \in \mathbb{N}^{+},
\end{equation}
and
\begin{equation}\label{Rl}
\phi_{4}(s_{c},\xi)= \xi_{3} \biggl[c_{d,M} 2^{-Ml}\bigl(-\frac{\xi_{1}}{d\xi_{3}}\bigl)^{d/(d-1)} \bigl(-\frac{\xi_{2}}{M\xi_{1}}\bigl)^{M/(M-1)} + 2^{-(M+1)l}  R_{l}(\frac{\xi_{1}}{\xi_{3}}, \frac{\xi_{2}}{\xi_{1}}) \biggl].
\end{equation}
Here, in fact, we use Taylor expansion to show the expression of $\phi_{4}(s_{c},\xi)$.

To facilitate the subsequent discussion, we denote
\begin{equation}\label{ajl1}
  \sum_{-N_{0} \le k \le N_{1}} \biggl( e^{-it\phi_{4}(s_{c},\xi)} \widetilde{a_{j,l,k}}(\xi,t)+ \widetilde{r_{j,l,k}}(\xi,t) \biggl) = : a_{j,l,1}(\xi,t).
\end{equation}
Then
 \begin{align}
&\sum_{-N_{0} \le k \le N_{1}}\beta(2^{-j-k}|\xi_{2}|) \widehat{d\mu_{l}}(t\xi) =   2^{-\frac{j}{2}} e^{-it\xi_{3}\bigl(-\frac{\xi_{1}}{d\xi_{3}}\bigl)^{\frac{d}{d-1}}}   a_{j,l,1}(\xi,t).\nonumber
\end{align}

$\bullet$ \textbf{(B-ii)} When $j \le Ml$,
we consider (\ref{mujlk}) and $\phi_{4}(s,\xi)$ in (\ref{phi4}),
\begin{equation}\label{tildephi3}
 \phi_{4}(s,2^{j}\xi^{\prime})= 2^{j}\delta^{M} \xi^{\prime}_{3} \biggl(\bigl[-  \frac{  s^{M}g(\delta s) \xi^{\prime}_{1}  + s\xi^{\prime}_{2}}{d\xi^{\prime}_{3}} \bigl] \cdot \frac{d}{d-1} \bigl(- \frac{\xi^{\prime}_{1}}{d\xi^{\prime}_{3}} \bigl)^{\frac{1}{d-1}} + \delta^{M} R_{\delta}\bigl(s, \frac{\xi^{\prime}_{1}}{\xi_{3}^{\prime}}, \frac{\xi^{\prime}_{2}}{\xi_{3}^{\prime}}\bigl) \biggl),
\end{equation}
by $| \xi^{\prime}_{i} |   \sim 1$ for $i=1,2,3$, we have
\[ | \phi _{4}(s,2^{j}\xi^{\prime})| \lesssim 2^{j-Ml} \lesssim 1 .\]
Then
\begin{align}
&\sum_{-N_{0} \le k \le N_{1}}\beta(2^{-j-k}|\xi_{2}|) \widehat{d\mu_{l}}(t\xi) \nonumber\\
&= 2^{-\frac{j}{2}} e^{-it\xi_{3}\bigl(-\frac{\xi_{1}}{d\xi_{3}}\bigl)^{\frac{d}{d-1}}}
\biggl(\sum_{-N_{0} \le k \le N_{1}} \beta(2^{-j-k}|\xi_{2}|) \int_{\mathbb{R}}e^{-it \phi_{4}(s,\xi) } a_j(2^{-l}s,t\xi_1,t2^{-{(M-1)l}}\xi_2,t\xi_3) \chi_{1}(s)ds\biggl), \nonumber
\end{align}
and the terms
\[\sum_{-N_{0} \le k \le N_{1}} \beta(2^{-j-k}|\xi_{2}|) \int_{\mathbb{R}}e^{-it \phi_{4}(s,\xi) } a_j(2^{-l}s,t\xi_1,t2^{-{(M-1)l}}\xi_2,t\xi_3) \chi_{1}(s)ds\]
can be put into the amplitude.

\textbf{Case (C)}: $k < -N_{0}$. In this case,
\[|\xi_{1}| \sim |\xi_{3}|  \sim 2^{j}, \quad |\xi_{2}| <2^{j-N_{0}},\]
 $N_{0}$ is defined by (\ref{defN0}) above.
We consider (\ref{mujlk}) and $\phi_{4}(s,\xi)$ in (\ref{phi4}).
Combining  (\ref{remaindercond1}) with the fact that $2^{-l} \ll 1$, we conclude  that
\begin{align}\label{lowerbound2}
|\partial_{s}\phi _{4}(s,2^{j}\xi^{\prime})| \sim 2^{j -Ml} \biggl| [Ms^{M}g(\delta)s + \delta s^{M}g^{\prime}(\delta s)] + \frac{\xi^{\prime}_{2}}{\xi^{\prime}_{1}} + \delta^{M}  \partial_{s}R_{\delta}\bigl(s, \frac{\xi^{\prime}_{1}}{\xi_{3}^{\prime}}, \frac{ \xi^{\prime}_{2}}{\xi_{3}^{\prime}}\bigl)  \biggl| \sim 2^{j-Ml}.
\end{align}

$\bullet$ \textbf{(C-i)} When $j> Ml$, we denote
\begin{equation}\label{ajl0}
\sum_{k < -N_{0}}\beta(2^{-j-k}|\xi_{2}|)\int_{\mathbb{R}}e ^{-it\phi_{4}(s,\xi)} a_j(2^{-l}s,t\xi_1,t2^{-{(M-1)l}}\xi_2,t\xi_3)  \chi_{1}(s)ds = : a_{j,l,0}(\xi,t).
\end{equation}
By (\ref{lowerbound2}), integration by parts in $s$ implies that
\begin{equation}\label{symbolb3}
|D^{\alpha}_{\xi^{\prime}}a_{j,l,0}(2^{j}\xi^{\prime},t)| \le  2^{-(j-Ml)N}
\end{equation}
holds for every multi-index $\alpha$ and every positive integer $N$. Then combining with (\ref{mujlk}), we get
 \begin{align}
&\sum_{ k <-N_{0}}\beta(2^{-j-k}|\xi_{2}|) \widehat{d\mu_{l}}(t\xi) =   2^{-\frac{j}{2}} e^{-it\xi_{3}\bigl(-\frac{\xi_{1}}{d\xi_{3}}\bigl)^{\frac{d}{d-1}}}   a_{j,l,0}(\xi,t). \nonumber
\end{align}

$\bullet$ \textbf{(C-ii)} When $j \le Ml$,  then
\begin{align}
&\sum_{  k <-N_{0}}\beta(2^{-j-k}|\xi_{2}|) \widehat{d\mu_{l}}(t\xi)\nonumber\\
&= 2^{-\frac{j}{2}} e^{-it\xi_{3}\bigl(-\frac{\xi_{1}}{d\xi_{3}}\bigl)^{\frac{d}{d-1}}}
\biggl(\sum_{k <-N_{0}} \beta(2^{-j-k}|\xi_{2}|) \int_{\mathbb{R}}e^{-it \phi_{4}(s,\xi) } a_{j}(2^{-l}s,t\xi_1,t2^{-{(M-1)l}}\xi_2,t\xi_3) \chi_{1}(s)ds\biggl), \nonumber
\end{align}
the term
 \[\sum_{ k < -N_{0} } \beta(2^{-j-k}|\xi_{2}|) \int_{\mathbb{R}}e^{-it \phi_{4}(s,\xi) } a_{j}(2^{-l}s,t\xi_1,t2^{-{(M-1)l}}\xi_2,t\xi_3) \chi_{1}(s)ds.\]
can be put into the amplitude.

Based on the above Case (A), Case (B), and Case (C), we present the asymptotic expansion of $\widehat{d\mu_{l}}(t\xi)$ below.

When $j>Ml$, by the arguments in (A-i), (B-i) and (C-i), we may decompose
\begin{equation}\label{dmudec1}
\widehat{d\mu_{l}}(t\xi) = 2^{-\frac{j}{2}} e^{-it\xi_{3}\bigl(-\frac{\xi_{1}}{d\xi_{3}}\bigl)^{\frac{d}{d-1}}} \bigl[a_{j,l,0}(\xi,t)+  a_{j,l,1}(\xi,t) + \sum_{k \ge N_{1}}  a_{j,l,k}(\xi,t)\bigl],
\end{equation}
where $a_{j,l,0}$ was defined by (\ref{ajl0}), $a_{j,l,1}$ was defined by (\ref{ajl1}), and $a_{j,l,k}$ was defined by (\ref{ajlk}) for $N_{1} < k \le (M-1)l $.

When $j \le Ml$, we divide $\widehat{d\mu_{l}}(t\xi)$ as
\[\widehat{d\mu_{l}}(t\xi)=\sum_{ Ml-j<k \le (M-1)l}\beta(2^{-j-k}|\xi_{2}|) \widehat{d\mu_{l}}(t\xi)+ \sum_{ k \le Ml-j}\beta(2^{-j-k}|\xi_{2}|) \widehat{d\mu_{l}}(t\xi).\]
According to the discussion in (A-ii), (B-ii) and (C-ii),
 \begin{align}
 &\sum_{ k \le Ml-j}\beta(2^{-j-k}|\xi_{2}|) \widehat{d\mu_{l}}(t\xi)\nonumber\\
 & = 2^{-\frac{j}{2}} e^{-it\xi_{3}\bigl(-\frac{\xi_{1}}{d\xi_{3}}\bigl)^{\frac{d}{d-1}}}
\biggl(\sum_{   k \le Ml-j} \beta(2^{-j-k}|\xi_{2}|) \int_{\mathbb{R}}e^{it \phi_{4}(s,\xi) } a_j(2^{-l}s,t\xi_1,t2^{-{(M-1)l}}\xi_2,t\xi_3) \chi_{1}(s)ds\biggl). \nonumber
\end{align}
We denote
\begin{equation}\label{ajlMl-j}
a_{j,l,Ml-j}(\xi,t):  =\sum_{  k \le Ml-j} \beta(2^{-j-k}|\xi_{2}|) \int_{\mathbb{R}}e^{it \phi_{4}(s,\xi) } a_j(2^{-l}s,t\xi_1,t2^{-{(M-1)l}}\xi_2,t\xi_3) \chi_{1}(s)ds,
\end{equation}
then we have
\begin{equation}\label{symbolb4}
|D^{\alpha}_{\xi^{\prime}}a_{j,l,Ml-j}( \xi_{j,Ml-j}^{\prime},t)| \lesssim 1
\end{equation}
with $\xi^{\prime}_{j,Ml-j}=2^{j}(\xi^{\prime}_{1},2^{Ml-j}\xi^{\prime}_{2},\xi^{\prime}_{3})$, $|\xi^{\prime}_{i}| \sim 1$ for $i=1,2,3$. Then combining with the subcase  $j\leq Ml$ and $k>Ml-j$ in (A-i), we obtain
\begin{equation}\label{dmudec2}
\widehat{d\mu_{l}}(t\xi)=\sum_{Ml-j \le k \le (M-1)l} 2^{-\frac{j}{2}} e^{-it\xi_{3}\bigl(-\frac{\xi_{1}}{d\xi_{3}}\bigl)^{\frac{d}{d-1}}}  a_{j,l,k}(\xi,t),
\end{equation}
where $a_{j,l,k}$ was defined by (\ref{ajlk}) for $Ml-j < k \le (M-1)l $.

Now we are ready to decompose $\mathcal{M}_{j,l}$ defined by (\ref{Mjl}) according to (\ref{dmudec1}) and (\ref{dmudec2}). For fixed $j,l,k$, we define the truncated maximal operator $\mathcal{M}_{j,l,k}$ by
\begin{align}\label{Mjlk}
\sup_{t\in[1,2]}\biggl|\int_{\mathbb{R}^{3}}e^{i \bigl[y \cdot \xi+ct\xi_{3} - t\xi_{3}\bigl(-\frac{\xi_{1}}{d\xi_{3}}\bigl)^{\frac{d}{d-1}} \bigl]} a_{j,l,k }(\xi,t)\chi_{1} \biggl(\frac{|\xi_{1}| }{2^{j}} \biggl)  \chi_{1} \biggl(\frac{\xi_{1}}{ \xi_{3} }\biggl)  \chi_{0}(2^{-j-(M-1)l }|\xi_{2}|) \hat{f}(\xi) d\xi\biggl|.
\end{align}
By  (\ref{dmudec1}),
for $j>Ml$, the $L^{p} \rightarrow L^{q}$ norm of $\mathcal{M}_{j,l}$ is dominated  by
\begin{align}\label{jlkdecompose2}
\|\mathcal{M}_{j,l}\|_{L^{p} \rightarrow L^{q}} &\le \|\mathcal{M}_{j,l,0}\|_{L^{p} \rightarrow L^{q}}+ \|\mathcal{M}_{j,l,1}\|_{L^{p} \rightarrow L^{q}} + \sum_{N_{1} < k \le (M-1)l} \|\mathcal{M}_{j,l,k}\|_{L^{p}\rightarrow L^{q}}.
\end{align}
For $j \le Ml$, by (\ref{dmudec2}),
the $L^{p} \rightarrow L^{q}$ norm of $\mathcal{M}_{j,l}$ is dominated by
\begin{align}\label{jlkdecompose1}
\|\mathcal{M}_{j,l}\|_{L^{p} \rightarrow L^{q}}
\le \sum_{Ml-j \le k \le (M-1)l} \|\mathcal{M}_{j,l,k}\|_{L^{p} \rightarrow L^{q}}.
\end{align}
So we first do some reduction to estimate   $\mathcal{M}_{j,l,k}$. For each $1 \le q \le \infty$,   by Sobolev  embedding, we have
\begin{align}\label{sobolevemb}
 \|  \mathcal{M}_{j,l,k }   \|_{L^{p} \rightarrow L^{q}}  \lesssim (1+c^{\frac{1}{q}})2^{\frac{j}{q}-\frac{j}{2}}  \|F_{j,l,k }\|_{  L^{p}(\mathbb{R}^3) \rightarrow L^{q} (\mathbb{R}^3\times [1/2,4]) },
\end{align}
where
\begin{equation}\label{IFO}
F_{j,l, k}f(y,t) = \rho(t) K_{j,l,k}(\cdot, t) * f(y),
\end{equation}
and $\rho(t)$ is a smooth cut-off function supported in $[1/2,4]$ and $\rho(t)=1$ on $[1,2]$,
\begin{equation}\label{kerneljlk}
K_{j,l,k }(y,t) := \int_{\mathbb{R}^{3}}e^{i \bigl[y \cdot \xi - t\xi_{3}\bigl(-\frac{\xi_{1}}{d\xi_{3}}\bigl)^{\frac{d}{d-1}} \bigl]} a_{j,l,k }(\xi,t)\chi_{1} \biggl(\frac{|\xi_{1}| }{2^{j}} \biggl)  \chi_{1} \biggl(\frac{\xi_{1}}{ \xi_{3} }\biggl)  \chi_{0}(2^{-j-(M-1)l }|\xi_{2}|)  d\xi.
\end{equation}



The following theorem is helpful  to handle $F_{j,l,k}$, and its proof will be left to Subsection \ref{proofofmainth4}.

\begin{theorem}\label{kernel estimate3}
 Let $K_{j,l,k}$ be defined by (\ref{kerneljlk}). Define $\widetilde{K_{j,l,k }}(y,t,t^{\prime})$ by
\begin{align}\label{tildekerneljlk}
\int_{\mathbb{R}^{3}}e^{i \bigl[y \cdot \xi +(t^{\prime}- t)\xi_{3}\bigl(-\frac{\xi_{1}}{d\xi_{3}}\bigl)^{\frac{d}{d-1}} \bigl]} a_{j,l,k }(\xi,t) \overline{a_{j,l,k }(}\xi,t^{\prime}) \chi^{2}_{1} \biggl(\frac{|\xi_{1}| }{2^{j}} \biggl)  \chi^{2}_{1} \biggl(\frac{\xi_{1}}{ \xi_{3} }\biggl)  \chi^{2}_{0}(2^{-j-(M-1)l }|\xi_{2}|)  d\xi.
\end{align}
Then we have the following estimates. \\
(I) When $j \le Ml$, $k\in \mathbb{Z}\bigcap [Ml-j,(M-1)l]$, or when  $j > Ml$, $k\in \biggl(\mathbb{Z}\bigcap (N_{1},(M-1)l]\biggl)\bigcup\{0\}$,
there hold
\begin{equation}\label{kerneli1}
\sup_{t \in [1/2,4]}\|K_{j, l,k}(\cdot,t)\|_{L^{1}} \lesssim 2^{\frac{j}{2}} 2^{-(j+k-Ml)N};
\end{equation}
\begin{equation}\label{kerneli2}
\sup_{t \in [1/2,4]} \sup_{y \in \mathbb{R}^{3}}|K_{j,l,k}(y,t)| \lesssim 2^{\frac{5j}{2}+k}2^{-(j+k-Ml)N};
\end{equation}
\begin{equation}\label{kerneli3}
\sup_{y \in \mathbb{R}^{3}}\bigl|\widetilde{K_{j,l,k}}(y,t,t^{\prime})\bigl| \lesssim 2^{3j+k}2^{-2(j+k-Ml)N} (1+2^{j}|t-t^{\prime}|)^{-1/2}, \quad t,t^{\prime} \in [1/2,4].
\end{equation}
(II) When $j>Ml$ and $k=1$,
\begin{equation}\label{kernelii1}
\|K_{j, l,1}(\cdot,t)\|_{L^{1}} \lesssim 2^{\frac{j}{2}};
\end{equation}
\begin{equation}\label{kernelii2}
\sup_{t \in [1/2,4]} \sup_{y \in \mathbb{R}^{3}}|K_{j,l,1}(y,t)| \lesssim 2^{\frac{3}{2}j + Ml};
\end{equation}
\begin{equation}\label{kernelii3}
\sup_{y \in \mathbb{R}^{3}} \bigl|\widetilde{K_{j,l,1}}(y,t,t^{\prime})\bigl|  \lesssim 2^{2j+Ml} (1+2^{j}|t-t^{\prime}|)^{-1/2}(1+2^{j-Ml}|t-t^{\prime}|)^{-1/2},  \quad t,t^{\prime} \in [1/2,4].
\end{equation}
\end{theorem}

We now briefly explain how we obtain Theorem \ref{mainth3}. For fixed $j,l,k$, we will obtain the $L^{p}  \rightarrow L^{q} $ estimate for $\mathcal{M}_{j,l,k}$ at $(\frac{1}{2},\frac{1}{2})$, $(0,0)$, $(1,1)$, $ (1,0)$, $ (\frac{1}{2},\frac{1}{6})$, $ (\frac{1}{2},\frac{1}{4})$, respectively. Then we will sum over $k$ and get the corresponding $L^{p}  \rightarrow L^{q} $ estimate for $\mathcal{M}_{j,l}$. Interpolating  between these estimates and summing  over $l$ will yield the $L^{p} \rightarrow L^{q} $ estimate for $\mathcal{M}_{loc,j}$ at $P_{1},P_{2},P_{3},P_{4}$, and $P_{5}$ (when $M=5$). Finally, by interpolating the estimates for  $\mathcal{M}_{loc,j}$  at $P_{1},P_{2},P_{3},P_{4}$,$P_{5}$ and the $L^{2}  \rightarrow L^{2} $ estimate for $\mathcal{M}_{loc,j}$, we will  prove that for each $(\frac{1}{p}, \frac{1}{q}) \in \Delta_{M}$, $M\geq 3$, there exists $\epsilon(p,q)>0$ such that (\ref{goal1.4}) remains valid.

The idea we employed to prove $L^{2} \rightarrow L^{4}$ and $L^{2} \rightarrow L^{6}$ estimates for $\mathcal{M}_{j,l,k}$ was inspired by local smoothing estimates appeared in Proposition 3.4 in \cite{MSS}. However, in this context, we cannot directly apply the conclusion therein, so we will provide a complete proof.

Next, we  elaborate individually on the proof for establishing $L^{p} \rightarrow L^{q} $  estimates of $\mathcal{M}_{loc,j}$ at $(\frac{1}{2},\frac{1}{2})$, $P_{1}$, $P_{2}$, $P_{3}$, and $P_{4}$.

\textbf{$\bullet$ $L^{2} \rightarrow L^{2}$ estimate for $\mathcal{M}_{loc,j}$. }
The Plancherel's theorem implies that
\begin{equation}
\|F_{j,l,k}\|_{ L^{2}(\mathbb{R}^{3}) \rightarrow L^{2}(\mathbb{R}^{3} \times [1/2,4])  } \lesssim \|a_{j,l,k}\|_{L^{\infty}_{\xi,t}(\mathbb{R}^{3} \times [1/2,4])}.\nonumber
\end{equation}
When $j \le Ml$, $k\in \mathbb{Z}\bigcap [Ml-j,(M-1)l]$, or when  $j > Ml$, $k\in \biggl(\mathbb{Z}\bigcap (N_{1},(M-1)l]\biggl)\bigcup\{0\}$, according to (\ref{sobolevemb}),  (\ref{symbolb4}),  (\ref{symbolb1}) and  (\ref{symbolb3}), we have
\begin{equation}
\|\mathcal{M}_{j,l,k}\|_{L^{2} \rightarrow L^{2}} \lesssim (1+c^{\frac{1}{2}}) 2^{-(j+k-Ml)N}. \nonumber
\end{equation}
When $j>Ml$, by (\ref{sobolevemb}), (\ref{symbolb2}) and (\ref{symbolb2+}), we get
\begin{equation}
\|\mathcal{M}_{j,l,1}\|_{L^{2} \rightarrow L^{2}} \lesssim (1+c^{\frac{1}{2}}) 2^{- \frac{j-Ml}{2}}. \nonumber
\end{equation}
We sum over $k$ according to (\ref{jlkdecompose1}) and (\ref{jlkdecompose2}) respectively to obtain the following results,
\begin{align}\label{L2L21}
\|\mathcal{M}_{j,l}\|_{L^{2} \rightarrow L^{2}} \lesssim (1+c^{\frac{1}{2}}) \sum_{Ml-j \le k \le (M-1)l}2^{-(j+k-Ml)N} \lesssim (1+c^{\frac{1}{2}}), \quad j \le Ml,
\end{align}
and
\begin{align}\label{L2L22}
\|\mathcal{M}_{j,l}\|_{L^{2}\rightarrow L^{2}} &\lesssim (1+c^{\frac{1}{2}}) \biggl( 2^{-(j-Ml)N} + 2^{-\frac{j-Ml}{2}} + \sum_{N_{1} < k \le (M-1)l, k\neq 1}2^{-(j+k-Ml)N} \biggl) \nonumber\\
&\lesssim (1+c^{\frac{1}{2}}) 2^{-\frac{j-Ml}{2}} , \quad j > Ml.
\end{align}
By (\ref{frequencydecompose1}), we  then get
\begin{align}\label{maximalL2}
 \|\mathcal{M}_{loc,j}   \|_{L^{2} \rightarrow L^{2}} \lesssim  \sum_{l_{0}\le l < \frac{j}{M}}(1+c^{\frac{1}{2}}) 2^{-\frac{j-Ml}{2}} 2^{-l}  + \sum_{ l\geq \frac{j}{M}}  (1+c^{\frac{1}{2}}) 2^{-l} \lesssim (1+c^{\frac{1}{2}}) 2^{-\frac{j}{M}}  .
\end{align}

\textbf{$\bullet$ Estimation for $\mathcal{M}_{loc,j}$ at $P_{1}$. } By Young's inequality, we have
\begin{equation}
\|F_{j,l,k}\|_{ L^{\infty}(\mathbb{R}^{3}) \rightarrow L^{\infty}(\mathbb{R}^{3} \times [1/2,4])  } \lesssim \sup_{t \in [1/2,4]} \|K_{j,l,k}(\cdot, t)\|_{L^{1}_{y}(\mathbb{R}^{3})}. \nonumber
\end{equation}
Combining with (\ref{kerneli1}) and (\ref{kernelii1}) in Theorem \ref{kernel estimate3}, we get
 \begin{equation}
\|\mathcal{M}_{j,l,k}\|_{L^{\infty}\rightarrow L^{\infty}} \lesssim 2^{-(j+k-Ml)N} \nonumber
\end{equation}
for $j \le Ml$, $k\in \mathbb{Z}\bigcap [Ml-j,(M-1)l]$,  along with $j>Ml$, $k\in \biggl(\mathbb{Z}\bigcap (N_{1},(M-1)l]\biggl)\bigcup\{0\}$, and
 \begin{equation}
\|\mathcal{M}_{j,l,1}\|_{L^{\infty}\rightarrow L^{\infty}} \lesssim  1, \quad j >Ml. \nonumber
\end{equation}
 Summing over $k$ by (\ref{jlkdecompose1}) and (\ref{jlkdecompose2})  respectively, we obtain
\begin{align}
\|\mathcal{M}_{j,l}\|_{L^{\infty} \rightarrow L^{\infty}} \lesssim  \sum_{Ml-j \le k \le (M-1)l}2^{-(j+k-Ml)N} \lesssim 1, \quad j \le Ml, \nonumber
\end{align}
and
\begin{align}
\|\mathcal{M}_{j,l}\|_{L^{\infty} \rightarrow L^{\infty}} &\lesssim 2^{-(j-Ml)N}   + 1 +\sum_{N_{1} < k \le (M-1)l}2^{-(j+k-Ml)N}    \lesssim 1 , \quad j > Ml. \nonumber
\end{align}
Hence, for each $j,l$, there holds
\begin{equation}\label{Linftyinftyjl}
\|\mathcal{M}_{j,l}\|_{L^{\infty}\rightarrow L^{\infty}}\lesssim 1,
\end{equation}
then by (\ref{frequencydecompose1}), we get the uniform estimate for $\mathcal{M}_{loc,j}$,
\begin{align}\label{maximalLinfty}
\|\mathcal{M}_{loc,j}   \|_{L^{\infty} \rightarrow L^{\infty}} &\lesssim  \sum_{l \ge l_{0} } 2^{-l} \lesssim 1.
\end{align}

\textbf{$\bullet$ Estimation for $\mathcal{M}_{loc,j}$ at $P_{2}$. } By Young's inequality,
there holds
\begin{equation}
\|F_{j,l,k}\|_{ L^{1}(\mathbb{R}^{3}) \rightarrow L^{1}(\mathbb{R}^{3} \times [1/2,4])  } \lesssim  \|K_{j,l,k}\|_{L^{1}_{y,t}(\mathbb{R}^{3} \times [1/2,4])}.
\end{equation}
By (\ref{kerneli1}) and (\ref{kernelii1}) in Theorem \ref{kernel estimate3}, when $j \le Ml$, $k\in \mathbb{Z}\bigcap [Ml-j,(M-1)l]$, or when  $j > Ml$, $k\in \biggl(\mathbb{Z}\bigcap (N_{1},(M-1)l]\biggl)\bigcup\{0\}$, we can obtain
 \begin{equation}
\|\mathcal{M}_{j,l,k}\|_{L^{1} \rightarrow L^{1}} \lesssim (1+c) 2^{j} \cdot 2^{-(j+k-Ml)N}, \nonumber
\end{equation}
and when $j>Ml$, we get
 \begin{equation}
\|\mathcal{M}_{j,l,1}\|_{L^{1} \rightarrow L^{1}} \lesssim  (1+c) 2^{j}. \nonumber
\end{equation}
Summing over $k$, we obtain
\begin{align}
\|\mathcal{M}_{j,l}\|_{L^{1}\rightarrow L^{1}} \lesssim  \sum_{Ml-j \le k \le (M-1)l}(1+c) 2^{j} 2^{-(j+k-Ml)N} \lesssim (1+c) 2^{j} , \quad j \le Ml, \nonumber
\end{align}
and
\begin{align}
\|\mathcal{M}_{j,l}\|_{L^{1} \rightarrow L^{1}} &\lesssim  (1+c) 2^{j}  \biggl(2^{-(j-Ml)N} + 1 +\sum_{N_{1} < k \le (M-1)l}2^{-(j+k-Ml)N} \biggl)  \nonumber\\
&\lesssim (1+c) 2^{j}  , \quad j > Ml. \nonumber
\end{align}
Hence, for each $j$ and $l$, we have
\begin{equation}\label{L1L1}
\|\mathcal{M}_{j,l}\|_{L^{1} \rightarrow L^{1}}  \lesssim (1+c) 2^{j}.
\end{equation}
To obtain the estimate for   $\mathcal{M}_{loc,j}$ at $P_{2}$, we interpolate between (\ref{L2L21}) and (\ref{L1L1}) when $j \le Ml$, and   between (\ref{L2L22}) and (\ref{L1L1}) when $j > Ml$, then apply (\ref{frequencydecompose1}) to obtain that
\begin{align}\label{MlocjP2}
& \|\mathcal{M}_{loc,j}  \|_{L^{\frac{2M}{M+1}}\rightarrow L^{\frac{2M}{M+1}}} \nonumber\\
 &\le (1+c^{\frac{M+1}{2M}}) \biggl( \sum_{ l \ge \frac{j}{M}} 2^{-l}  2^{\frac{j}{M}} +  \sum_{l_{0} \le l < \frac{j}{M}} 2^{-l} 2^{\frac{M-1}{M} \cdot (-\frac{j}{2}+\frac{Ml}{2})} 2^{\frac{j}{M}} \biggl)  \nonumber\\
 & \lesssim (1+c^{\frac{M+1}{2M}}).
\end{align}

\textbf{$\bullet$ Estimation for $\mathcal{M}_{loc,j}$ at $P_{3}$. }Recall that $P_{3}= (\frac{4}{M+2},\frac{2}{M+2})$ when $M \ge 6$, and $P_{3}= (\frac{2M+4}{5M+2},\frac{M+2}{5M+2})$ when $3 \le M \le 5$. We first outline the proof for the two cases separately. For $M \ge 6$, where $0<\frac{2}{M+2} \le \frac{1}{4}$, we will build the $L^{2} \rightarrow L^{4}$ estimate of $\mathcal{M}_{j,l}$,   then interpolate between this estimate and the $L^{\infty} \rightarrow L^{\infty}$ estimate of $\mathcal{M}_{j,l}$. Applying  (\ref{frequencydecompose1}), we obtain the  estimate of $\mathcal{M}_{loc,j}$ at $P_{3}$. When $3 \le M \le 5$, where $\frac{1}{4} <\frac{M+2}{5M+2} < \frac{1}{3}$, we will  establish the $L^{\frac{3}{2}}  \rightarrow L^{3} $ estimate of $\mathcal{M}_{j,l}$,   then interpolate between  this estimate   and the $L^{2}  \rightarrow L^{4} $ estimate of $\mathcal{M}_{j,l}$. By applying  (\ref{frequencydecompose1}), we can obtain the  estimate of $\mathcal{M}_{loc,j}$ at $P_{3}$. Next, we give more details of  the proof.

$\blacklozenge$ \textbf{Case $M \ge 6$}. For $2<q<\infty$, in order to get $L^{2}(\mathbb{R}^{3}) \rightarrow L^{q}(\mathbb{R}^{3} \times [1/2,4])$ estimate for the operator defined by (\ref{IFO}), by Plancherel's theorem, it suffices to consider $L^{2}(\mathbb{R}^{3}) \rightarrow L^{q}(\mathbb{R}^{3} \times [1/2,4])$ estimate for the operator
\begin{equation*}
T_{j,l,k }f(y,t):=\rho(t) \int_{\mathbb{R}^{3}}e^{i \bigl[y \cdot \xi - t\xi_{3}\bigl(-\frac{\xi_{1}}{d\xi_{3}}\bigl)^{\frac{d}{d-1}} \bigl]} a_{j,l,k }(\xi,t)\chi_{1} \biggl(\frac{|\xi_{1}| }{2^{j}} \biggl)  \chi_{1} \biggl(\frac{\xi_{1}}{ \xi_{3} }\biggl)  \chi_{0}(2^{-j-(M-1)l }|\xi_{2}|) f(\xi) d\xi.
\end{equation*}
 By duality,
\[\|T_{j,l,k}\|_{L^{2}(\mathbb{R}^{3}) \rightarrow L^{q}(\mathbb{R}^{3} \times [1/2,4]) } =\| T^{*}_{j,l,k}\|_{L^{q^{\prime}}(\mathbb{R}^{3} \times [1/2,4]) \rightarrow L^{2}(\mathbb{R}^{3})}, \quad 1/q + 1/q^{\prime} = 1, \]
where $T^{*}_{j,l,k}$ is defined by
\[T^{*}_{j,l,k}G(\xi):= \int_{\mathbb{R}^{4}}e^{i \bigl[-y \cdot \xi+ t\xi_{3}\bigl(-\frac{\xi_{1}}{d\xi_{3}}\bigl)^{\frac{d}{d-1}} \bigl]}\overline{ a_{j,l,k }(\xi,t)}\chi_{1} \biggl(\frac{|\xi_{1}| }{2^{j}} \biggl)  \chi_{1} \biggl(\frac{\xi_{1}}{ \xi_{3} }\biggl)  \chi_{0}(2^{-j-(M-1)l }|\xi_{2}|) \rho(t) G(y,t) dydt,\]
and $G \in L^{q^{\prime}}(\mathbb{R}^{3} \times [1/2,4])$. Moreover, since
\begin{align}
\| T^{*}_{j,l,k}G\|^{2}_{L^{2}(\mathbb{R}^{3})} \le \|G\|_{L^{q^{\prime}}(\mathbb{R}^{3} \times [1/2,4])}  \|T_{j,l,k}T^{*}_{j,l,k}G\|_{L^{q }(\mathbb{R}^{3} \times [1/2,4])},
\end{align}
it follows that
 \begin{equation}\label{reducetodual}
  \| T^{*}_{j,l,k}\|_{L^{q^{\prime}}(\mathbb{R}^{3} \times [1/2,4]) \rightarrow L^{2}(\mathbb{R}^{3})} \le \|T_{j,l,k}T^{*}_{j,l,k}\|^{\frac{1}{2}}_{L^{q^{\prime} }(\mathbb{R}^{3} \times [1/2,4]) \rightarrow L^{q }(\mathbb{R}^{3} \times [1/2,4])}.
  \end{equation}
Notice that
\begin{equation}\label{dualestimate1}
  T_{j,l,k}T^{*}_{j,l,k}G(z,t^{\prime})  = \int_{\mathbb{R}} \widetilde{K_{j,l,k}}(\cdot,t, t^{\prime})  *_{y} G(\cdot,t) (z) \rho(t) \rho(t^{\prime})dt  ,
\end{equation}
where $*_{y}$ denotes convolution with respect to the variable $y$, and $\widetilde{K_{j,l,k}}(y,t, t^{\prime})$ is defined by (\ref{tildekerneljlk}).

When $j \le Ml$, $k\in \mathbb{Z}\bigcap [Ml-j,(M-1)l]$, or when  $j > Ml$, $k\in \biggl(\mathbb{Z}\bigcap (N_{1},(M-1)l]\biggl)\bigcup\{0\}$, by Plancherel's theorem and  (\ref{symbolb4}), (\ref{symbolb1}), (\ref{symbolb3}), there holds
\begin{equation}\label{dualestimate2}
\biggl\|  \widetilde{K_{j,l,k}}(\cdot,t, t^{\prime})  *_{y} G(\cdot,t) (z)   \biggl\|_{L_{z}^{2}(\mathbb{R}^{3})} \lesssim  2^{-2(j+k-Ml)N}  \|G(\cdot,t)\|_{L^{2}(\mathbb{R}^{3})}
\end{equation}
for fixed $t$,  $t^{\prime}$.
The inequality (\ref{kerneli3}) in Theorem \ref{kernel estimate3}  implies that for each $t \neq t^{\prime}$, the following estimate holds true
\begin{equation}\label{dualestimate3}
\biggl\|  \widetilde{K_{j,l,k}}(\cdot,t, t^{\prime})  *_{y} G(\cdot,t) (z)   \biggl\|_{L_{z}^{\infty}(\mathbb{R}^{3})} \lesssim 2^{\frac{5j}{2}+k} 2^{-2(j+k-Ml)N}  |t-t^{\prime}|^{-\frac{1}{2}}  \|G(\cdot,t)\|_{L^{1}(\mathbb{R}^{3})}.
\end{equation}
Interpolation between (\ref{dualestimate2}) and (\ref{dualestimate3}) yields
\begin{equation}\label{dualestimate6}
\biggl\|  \widetilde{K_{j,l,k}}(\cdot,t, t^{\prime})  *_{y} G(\cdot,t) (z)   \biggl\|_{L_{z}^{4}(\mathbb{R}^{3})} \lesssim 2^{\frac{5j}{4}+\frac{k}{2}}  2^{-2(j+k-Ml)N}  |t-t^{\prime}|^{-\frac{1}{4}}  \|G(\cdot,t)\|_{L^{\frac{4}{3}}(\mathbb{R}^{3})}.
\end{equation}
By  (\ref{dualestimate6}) and the Hardy-Littlewood-Sobolev inequality, we obtain
\begin{align}\label{dualestimate5}
& \biggl\|  \int_{\mathbb{R}} \widetilde{K_{j,l,k}}(\cdot,t, t^{\prime})  *_{y} G(\cdot,t) (z)\rho(t) \rho(t^{\prime}) dt \biggl\|_{L_{z,t^{\prime}}^{4}(\mathbb{R}^{3}\times [1/2,4])}\nonumber\\
&\le \biggl\| \int_{\mathbb{R}} \biggl\| \widetilde{K_{j,l,k}}(\cdot,t, t^{\prime}) *_{y} G(\cdot,t) (z) \biggl\|_{L_{z}^{4}(\mathbb{R}^{3})} \rho(t) \rho(t^{\prime})dt \biggl\|_{L_{t^{\prime}}^{4}([1/2,4] )} \nonumber\\
&\lesssim  2^{\frac{5j}{4}+\frac{k}{2}} 2^{-2(j+k-Ml)N} \biggl\| \int_{\mathbb{R}}  |t-t^{\prime}|^{-\frac{1}{4}}  \|G(\cdot,t) \|_{L^{\frac{4}{3}}(\mathbb{R}^{3})} \rho(t) \rho(t^{\prime}) dt \biggl\|_{L_{t^{\prime}}^{4}([1/2,4])} \nonumber\\
&\lesssim  2^{\frac{5j}{4}+\frac{k}{2}} 2^{-2(j+k-Ml)N} \biggl\| \int_{\mathbb{R}}  |t-t^{\prime}|^{-\frac{1}{2}}  \|G(\cdot,t) \|_{L^{\frac{4}{3}}(\mathbb{R}^{3})} \rho(t) \rho(t^{\prime}) dt \biggl\|_{L_{t^{\prime}}^{4}([1/2,4])} \nonumber\\
&\lesssim 2^{\frac{5j}{4}+\frac{k}{2}} 2^{-2(j+k-Ml)N}\|G\|_{L^{4/3}(\mathbb{R}^{3} \times [1/2,4])}.
\end{align}
Combining (\ref{dualestimate5}), (\ref{dualestimate1}) with (\ref{reducetodual}), we obtain
 \[\|F_{j,l,k}\|_{L^{2}(\mathbb{R}^{3}) \rightarrow L^{4}(\mathbb{R}^{3} \times [1/2,4]) }=\|T_{j,l,k}\|_{L^{2}(\mathbb{R}^{3}) \rightarrow L^{4}(\mathbb{R}^{3} \times [1/2,4]) } \lesssim 2^{\frac{5j}{8}+\frac{k}{4}} 2^{-(j+k-Ml)N}.\]
 It follows from (\ref{sobolevemb}) that
\begin{align}
\|\mathcal{M}_{j,l,k}\|_{L^{2}  \rightarrow L^{4} } \lesssim (1+c^{\frac{1}{4}}) 2^{-\frac{j}{4}}2^{\frac{3j}{8}+\frac{j+k}{4}} 2^{-(j+k-Ml)N}.\nonumber
\end{align}
When $j \le Ml$,  we apply (\ref{jlkdecompose1}) to sum over $k$, and obtain that
\begin{align}\label{L2L41}
\|\mathcal{M}_{j,l}\|_{L^{2}  \rightarrow L^{4} } &\lesssim \sum_{Ml-j \le k \le (M-1)l}(1+c^{\frac{1}{4}}) 2^{-\frac{j}{4}}2^{\frac{3j}{8}+\frac{j+k}{4}} 2^{-(j+k-Ml)N} \nonumber\\
&\lesssim (1+c^{\frac{1}{4}}) 2^{-\frac{j}{4}}2^{\frac{3j}{8}+\frac{Ml}{4}}.
\end{align}
When $j > Ml$, we still need to consider the case where $k=1$. By Plancherel's theorem, (\ref{symbolb2}) and (\ref{symbolb2+}), we have
\begin{equation}\label{tildekernelL2L2}
\biggl\|  \widetilde{K_{j,l,1}}(\cdot,t, t^{\prime})  *_{y} G(\cdot,t) (z)   \biggl\|_{L_{z}^{2}(\mathbb{R}^{3})} \lesssim  2^{-(j-Ml)}  \|G(\cdot,t)\|_{L^{2}(\mathbb{R}^{3})}.
\end{equation}
The inequality (\ref{kernelii3}) in Theorem \ref{kernel estimate3}  implies that for each $t \neq t^{\prime}$, the following estimate holds true
\begin{equation}\label{tildekernelL2L4}
\biggl\|  \widetilde{K_{j,l,1}}(\cdot,t, t^{\prime})  *_{y} G(\cdot,t) (z)   \biggl\|_{L_{z}^{\infty}(\mathbb{R}^{3})} \lesssim 2^{j+ \frac{3Ml}{2}}  |t-t^{\prime}|^{-1}  \|G(\cdot,t)\|_{L^{1}(\mathbb{R}^{3})}.
\end{equation}
Interpolation between (\ref{tildekernelL2L2}) and (\ref{tildekernelL2L4}) implies
\begin{equation}\label{dualestimate4++}
\biggl\|  \widetilde{K_{j,l,1}}(\cdot,t, t^{\prime})  *_{y} G(\cdot,t) (z)   \biggl\|_{L_{z}^{4}(\mathbb{R}^{3})} \lesssim 2^{\frac{5Ml}{4}}    |t-t^{\prime}|^{-\frac{1}{2}}  \|G(\cdot,t)\|_{L^{\frac{4}{3}}(\mathbb{R}^{3})}.
\end{equation}
By  (\ref{dualestimate4++}), the Hardy-Littlewood-Sobolev inequality,  (\ref{dualestimate1}) and (\ref{reducetodual}),  we can obtain that
 \[\|F_{j,l,1}\|_{L^{2}(\mathbb{R}^{3}) \rightarrow L^{4}(\mathbb{R}^{3} \times [1/2,4]) }=\|T_{j,l,1}\|_{L^{2}(\mathbb{R}^{3}) \rightarrow L^{4}(\mathbb{R}^{3} \times [1/2,4]) }  \lesssim 2^{\frac{5Ml}{8}}.\]
 It follows from (\ref{sobolevemb}) that
\begin{align}
\|\mathcal{M}_{j,l,1}\|_{L^{2}  \rightarrow L^{4} } \lesssim (1+c^{\frac{1}{4}}) 2^{-\frac{j}{4}}2^{\frac{5Ml}{8}}.\nonumber
\end{align}
By (\ref{jlkdecompose2}), we now sum over $k$ to obtain that
\begin{align}\label{L2L42}
&\|\mathcal{M}_{j,l}\|_{L^{2}
 \rightarrow L^{4} } \nonumber\\
&  \lesssim (1+c^{\frac{1}{4}}) \biggl( 2^{-\frac{j}{4}}2^{\frac{3j}{8}+\frac{j}{4}} 2^{-(j-Ml)N} + 2^{-\frac{j}{4}}2^{\frac{5Ml}{8}} + \sum_{N_{1} < k \le (M-1)l } 2^{-\frac{j}{4}}2^{\frac{3j}{8}+\frac{j+k}{4}} 2^{-(j+k-Ml)N}\biggl) \nonumber\\
&\lesssim (1+c^{\frac{1}{4}}) 2^{-\frac{j}{4}}2^{\frac{5Ml}{8}}, \quad j>Ml.
\end{align}

Furthermore,  to establish  the $L^{\frac{M+2}{4}} \rightarrow L^{\frac{M+2}{2}}$ boundedness for the operator $\mathcal{M}_{loc,j}$, we
interpolate between (\ref{L2L41}) and (\ref{Linftyinftyjl}) when $j \le Ml$, and between (\ref{L2L42}) and (\ref{Linftyinftyjl}) when $j > Ml$. Then applying  inequality (\ref{frequencydecompose1})  yields
\begin{align}\label{MlocjP3}
 &\|\mathcal{M}_{loc,j}  \|_{L^{\frac{M+2}{4}} \rightarrow L^{\frac{M+2}{2}} }
 \nonumber\\
 &\le (1 + c^{\frac{2}{M+2}}) \biggl( \sum_{ l \ge \frac{j}{M}} 2^{-l} 2^{-\frac{2(M-1)}{M+2}l} 2^{\frac{8}{M+2}(\frac{j}{8}+\frac{Ml}{4})} +  \sum_{l_{0} \le l < \frac{j}{M}} 2^{-l} 2^{-\frac{2(M-1)}{M+2}l} 2^{ \frac{8}{M+2} \cdot (-\frac{j}{4} + \frac{5Ml}{8})} \biggl)   \nonumber\\
 & \lesssim (1 + c^{\frac{2}{M+2}}) .
\end{align}

$\blacklozenge$ \textbf{Case $3 \le M \le 5$}. In order to build the $L^{\frac{3}{2}}  \rightarrow L^{3} $ estimate for $\mathcal{M}_{j,l}$, we will first obtain the $L^{1}  \rightarrow L^{\infty} $  estimate for $\mathcal{M}_{j,l}$, then interpolate between this estimate and  the $L^{2} \rightarrow L^{2} $ estimate for $\mathcal{M}_{j,l}$ by (\ref{L2L21}) and (\ref{L2L22}).

 For each $j \le Ml$, $k\in \mathbb{Z}\bigcap [Ml-j,(M-1)l]$, and $j > Ml$, $k\in \biggl(\mathbb{Z}\bigcap (N_{1},(M-1)l]\biggl)\bigcup\{0\}$,  by (\ref{sobolevemb}), Young's inequality and  (\ref{kerneli2}) in Theorem \ref{kernel estimate3}, we have
\begin{equation}
\|\mathcal{M}_{j,l,k}\|_{L^{1} \rightarrow L^{\infty} } \lesssim 2^{-\frac{j}{2}} \cdot \sup_{t \in [1/2,4]} \sup_{y \in \mathbb{R}^{3}}|K_{j,l,k}(y,t)| \lesssim 2^{2j+k}2^{-(j+k-Ml)N}. \nonumber
\end{equation}
When $j>Ml$, it follows from  (\ref{kernelii2}) in Theorem \ref{kernel estimate3} that
\begin{equation}
\|\mathcal{M}_{j,l,1}\|_{L^{1}  \rightarrow L^{\infty} } \lesssim \sup_{t \in [1/2,4]} \sup_{y \in \mathbb{R}^{3}}|K_{j,l,1}(y,t)| \lesssim 2^{j+Ml}. \nonumber
\end{equation}
Then we sum over $k$ according to (\ref{jlkdecompose1}) to obtain that
\begin{equation}
\|\mathcal{M}_{j,l}\|_{L^{1}  \rightarrow L^{\infty} } \lesssim \sum_{j-Ml \le k \le (M-1)l} 2^{2j+k}2^{-(j+k-Ml)N} \lesssim 2^{j+Ml}, \quad j \le Ml, \nonumber
\end{equation}
and by (\ref{jlkdecompose2}),
\begin{align}
\label{Linfty1jl2}
\|\mathcal{M}_{j,l}\|_{L^{1}  \rightarrow L^{\infty} } &\lesssim 2^{2j} \cdot 2^{-(j-Ml)N}+2^{j+Ml}+\sum_{N_{1} < k \le (M-1)l} 2^{2j+k}2^{-(j+k-Ml)N} \nonumber\\
&\lesssim 2^{j+Ml}, \quad j > Ml. \nonumber
\end{align}
Hence, for each $j,l$, there holds
\begin{equation}\label{Linfty1jl}
\|\mathcal{M}_{j,l}\|_{L^{1} \rightarrow L^{\infty} }  \lesssim 2^{j+Ml}.
\end{equation}
Then we  interpolate between (\ref{Linfty1jl}) and (\ref{L2L21}) for $j \le Ml$ to obtain that
\begin{equation}\label{L31}
 \| \mathcal{M}_{j,l}  \|_{L^{\frac{3}{2}}  \rightarrow L^{3} } \lesssim (1+c^{\frac{1}{3}})2^{\frac{j}{3} + \frac{Ml}{3}}, \quad j \le Ml,
\end{equation}
and interpolate between  (\ref{Linfty1jl}) and (\ref{L2L22}) for $j>Ml$ to get
\begin{equation}\label{L32}
 \| \mathcal{M}_{j,l}  \|_{L^{\frac{3}{2}}  \rightarrow L^{3} } \lesssim  (1+c^{\frac{1}{3}}) 2^{ \frac{2Ml}{3}}, \quad j > Ml.
\end{equation}
Finally,  we interpolate  between (\ref{L31}) and (\ref{L2L41}) when $j \le Ml$, and  between (\ref{L32}) and (\ref{L2L42}) when $j > Ml$, then apply (\ref{frequencydecompose1}) to sum over $l$,
\begin{align}\label{MlocjP3+}
  \|\mathcal{M}_{loc,j} \|_{L^{\frac{5M+2}{2M+4}} \rightarrow L^{\frac{5M+2}{M+2}} }
 &\le (1 + c^{\frac{M+2}{5M+2}}) \biggl( \sum_{ l \ge \frac{j}{M}} 2^{-l} 2^{-\frac{(M+2)(M-1)}{5M+2}l} 2^{\frac{8M-16}{5M+2}(\frac{j}{8} + \frac{Ml}{4})}2^{\frac{18-3M}{5M+2}(\frac{j}{3}+\frac{Ml}{3})} \nonumber\\
 &+  \sum_{l_{0} \le l < \frac{j}{M}} 2^{-l} 2^{-\frac{(M+2)(M-1)}{5M+2}l} 2^{\frac{8M-16}{5M+2}\cdot (-\frac{j}{4}+\frac{5Ml}{8})} 2^{\frac{18-3M}{5M+2} \cdot \frac{2Ml}{3}} \biggl) \nonumber\\
 & \lesssim (1 + c^{\frac{M+2}{5M+2}})   .
\end{align}

\textbf{$\bullet$ Estimation for $\mathcal{M}_{loc,j}$ at $P_{4}$.} Recalling that $P_{4} = (\frac{3}{M+2}, \frac{1}{M+2})$ when $M \ge 4$, and $P_{4}=(\frac{3M+6}{8M+4},\frac{M+2}{8M+4})$ when $M =3$, we give a brief idea of the proof. Notice that $0<\frac{1}{M+2} \le \frac{1}{6}$ for $M \ge 4$, and $\frac{1}{6}< \frac{M+2}{8M+4} < \frac{1}{5} $ for $M=3$. Hence, for $M \ge 4$, to obtain the estimation of $\mathcal{M}_{loc,j}$ at $P_{4}$, we first build the   $L^{2}  \rightarrow L^{6} $ estimate of $\mathcal{M}_{j,l}$, and interpolate between this estimate  and the $L^{\infty} \rightarrow L^{\infty} $ estimate of $\mathcal{M}_{j,l}$, then use (\ref{frequencydecompose1}) to sum over $l$. For $M = 3$, we first build the $L^{\frac{5}{3}}  \rightarrow L^{5}  $ estimate for $\mathcal{M}_{j,l}$, then interpolate between this estimate and  the $L^{2}  \rightarrow L^{6}  $ estimate for $\mathcal{M}_{j,l}$, finally the desired estimate is followed by summation over $l$ via (\ref{frequencydecompose1}).

$\blacklozenge$ \textbf{Case $M \ge 4$.} We first establish the  $L^{2}  \rightarrow L^{6}  $ estimate for $\mathcal{M}_{j,l,k}$.  When $j \le Ml$, $k\in \mathbb{Z}\bigcap [Ml-j,(M-1)l]$, or when  $j > Ml$, $k\in \biggl(\mathbb{Z}\bigcap (N_{1},(M-1)l]\biggl)\bigcup\{0\}$, for fixed $t$, $t^{\prime}$,
interpolation between (\ref{dualestimate2}) and (\ref{dualestimate3}) yields
\begin{equation}\label{dualestimate4}
\biggl\|  \widetilde{K_{j,l,k}}(\cdot,t, t^{\prime})  *_{y} G(\cdot,t) (z)   \biggl\|_{L_{z}^{6}(\mathbb{R}^{3})} \lesssim 2^{\frac{5j}{3}+\frac{2k}{3}}  2^{-2(j+k-Ml)N}  |t-t^{\prime}|^{-\frac{1}{3}}  \|G(\cdot,t)\|_{L^{\frac{6}{5}}(\mathbb{R}^{3})}.
\end{equation}
Combining with (\ref{dualestimate4}) and the Hardy-Littlewood-Sobolev inequality, we obtain
\begin{align}
 &\biggl\| \int_{\mathbb{R}} \widetilde{K_{j,l,k}}(\cdot,t, t^{\prime})  *_{y} G(\cdot,t) (z) dt \biggl\|_{L_{z,t^{\prime}}^{6}(\mathbb{R}^{3}\times [1/2,4]) } \nonumber\\
&\lesssim  2^{\frac{5j}{3}+\frac{2k}{3}} 2^{-2(j+k-Ml)N} \biggl\| \int_{\mathbb{R}}  |t-t^{\prime}|^{-\frac{1}{3}}  \|G(\cdot,t)\|_{L^{\frac{6}{5}}(\mathbb{R}^{3})} \rho(t)\rho(t^{\prime}) dt \biggl\|_{L_{t^{\prime}}^{6}([1/2,4])} \nonumber\\
&\lesssim 2^{\frac{5j}{3}+\frac{2k}{3}} 2^{-2(j+k-Ml)N}\|G\|_{L^{6/5}(\mathbb{R}^{3} \times [1/2,4])}.\nonumber
\end{align}
Then it follows from (\ref{dualestimate1}) and (\ref{reducetodual}) that
 \[\|F_{j,l,k}\|_{L^{2}(\mathbb{R}^{3}) \rightarrow L^{6}(\mathbb{R}^{3} \times [1/2,4]) } =\|T_{j,l,k}\|_{L^{2}(\mathbb{R}^{3}) \rightarrow L^{6}(\mathbb{R}^{3} \times [1/2,4]) }\lesssim 2^{\frac{5j}{6}+\frac{k}{3}} 2^{-(j+k-Ml)N}.\]
Using (\ref{sobolevemb}), we obtain the following estimate
\begin{align}
\|\mathcal{M}_{j,l,k}\|_{L^{2}  \rightarrow L^{6} } \lesssim (1+c^{\frac{1}{6}}) 2^{\frac{j}{2}+\frac{k}{3}} 2^{-(j+k-Ml)N}.\nonumber
\end{align}
When $j \le Ml$, we sum over $k$ according to (\ref{jlkdecompose1}) to get
\begin{equation}\label{L2L61}
\|\mathcal{M}_{j,l}\|_{L^{2}\rightarrow L^{6}} \lesssim \sum_{Ml-j \le k \le (M-1)l}(1+c^{\frac{1}{6}}) 2^{\frac{j}{2}+\frac{k}{3}} 2^{-(j+k-Ml)N}\lesssim (1+c^{\frac{1}{6}}) 2^{\frac{j}{6} + \frac{Ml}{3}}.
\end{equation}
When $j > Ml$, we are left  to consider the case where  $k=1$. The inequality (\ref{kernelii3}) in
 Theorem \ref{kernel estimate3}  implies that for each $t \neq t^{\prime}$, the following estimate holds true
\begin{equation}\label{tildekernelL2L4+}
\biggl\|  \widetilde{K_{j,l,1}}(\cdot,t, t^{\prime})  *_{y} G(\cdot,t) (z)   \biggl\|_{L_{z}^{\infty}(\mathbb{R}^{3})} \lesssim 2^{\frac{3j}{2}+ Ml}  |t-t^{\prime}|^{-\frac{1}{2}}  \|G(\cdot,t)\|_{L^{1}(\mathbb{R}^{3})}.
\end{equation}
Interpolation between (\ref{tildekernelL2L2}) with (\ref{tildekernelL2L4+}) implies
\begin{equation}\label{dualestimate4+}
\biggl\|  \widetilde{K_{j,l,1}}(\cdot,t, t^{\prime})  *_{y} G(\cdot,t) (z)   \biggl\|_{L_{z}^{6}(\mathbb{R}^{3})} \lesssim 2^{\frac{2}{3}j+ Ml}    |t-t^{\prime}|^{-\frac{1}{3}}  \|G(\cdot,t)\|_{L^{\frac{6}{5}}(\mathbb{R}^{3})}.
\end{equation}
By (\ref{dualestimate4+}) and the Hardy-Littlewood-Sobolev inequality, there holds
\begin{align}
 \biggl\| \int_{\mathbb{R}} \widetilde{K_{j,l,1}}(\cdot,t, t^{\prime})  *_{y} G(\cdot,t) (z) dt \biggl\|_{L_{z,t^{\prime}}^{6}(\mathbb{R}^{3} \times [1/2,4])}
\lesssim 2^{\frac{2}{3}j+ Ml}  \|G\|_{L^{6/5}(\mathbb{R}^{3} \times [1/2,4])}.\nonumber
\end{align}
Then combining with  (\ref{reducetodual}) and (\ref{dualestimate1}), we get
 \[\|F_{j,l,1}\|_{L^{2}(\mathbb{R}^{3}) \rightarrow L^{6}(\mathbb{R}^{3} \times [1/2,4])}=\|T_{j,l,1}\|_{L^{2}(\mathbb{R}^{3}) \rightarrow L^{6}(\mathbb{R}^{3} \times [1/2,4])}  \lesssim 2^{\frac{j}{3}+ \frac{Ml}{2}}.\]
By (\ref{sobolevemb}),
\begin{align}
\|\mathcal{M}_{j,l,1}\|_{L^{2}  \rightarrow L^{6} } \lesssim (1+c^{\frac{1}{4}}) 2^{\frac{Ml}{2}}. \nonumber
\end{align}
We sum over $k$  by (\ref{jlkdecompose2}) to obtain that
\begin{align}\label{L2L62}
\|\mathcal{M}_{j,l}\|_{L^{2}  \rightarrow L^{6} } &\lesssim (1+c^{\frac{1}{6}}) \biggl(  2^{\frac{j}{2}}\cdot 2^{-(j-Ml)N} +2^{\frac{Ml}{2}} +\sum_{N_{1} < k \le (M-1)l} 2^{\frac{j}{2}+\frac{k}{3}} 2^{-(j+k-Ml)N} \biggl) \nonumber\\
&\lesssim (1+c^{\frac{1}{6}}) 2^{\frac{Ml}{2}}.
\end{align}

Now we interpolate between (\ref{L2L61}) and (\ref{Linftyinftyjl}) when $j \le Ml$, and  between (\ref{L2L62}) and (\ref{Linftyinftyjl}) when $j > Ml$,  and  apply inequality (\ref{frequencydecompose1})  to obtain that
\begin{align}\label{MlocjP4}
& \|\mathcal{M}_{loc,j}  \|_{L^{\frac{M+2}{3}}  \rightarrow L^{M+2} } \nonumber\\
 &\le (1+c^{\frac{1}{M+2}}) \biggl( \sum_{ l \ge \frac{j}{M}} 2^{-l} 2^{-\frac{2(M-1)}{M+2}l} 2^{\frac{6}{M+2}(\frac{j}{6}+ \frac{Ml}{3})}  +  \sum_{l_{0} \le l < \frac{j}{M}} 2^{-l} 2^{-\frac{2(M-1)}{M+2}l} 2^{\frac{6}{M+2} \cdot \frac{Ml}{2}} \biggl) \nonumber\\
 & \lesssim (1+c^{\frac{1}{M+2}}) j  .
\end{align}

$\blacklozenge$ \textbf{Case $M=3$.} We first  interpolate between the $L^{2}  \rightarrow L^{4} $ estimate  in (\ref{L2L41})  and the $L^{1} \rightarrow L^{\infty} $ estimate in (\ref{Linfty1jl})  to obtain
\begin{equation}\label{L51}
 \| \mathcal{M}_{j,l} \|_{L^{\frac{5}{3}} \rightarrow L^{5}} \lesssim 2^{\frac{3j}{10}+\frac{2Ml}{5}} , \quad j \le Ml,
\end{equation}
then we interpolate between the $L^{2} \rightarrow L^{4} $  estimate in (\ref{L2L42})
and the $L^{1}  \rightarrow L^{\infty} $ estimate in (\ref{Linfty1jl})  to get
\begin{equation}\label{L52}
 \|  \mathcal{M}_{j,l}\|_{L^{\frac{5}{3}} \rightarrow L^{5}} \lesssim 2^{\frac{7Ml}{10} } , \quad j>Ml.
\end{equation}
Finally,  we interpolate between (\ref{L2L61}) and (\ref{L51}) when $j \le Ml$, and between (\ref{L2L62})  and (\ref{L52}) when $j>Ml$, and  apply the inequality (\ref{frequencydecompose1})  to get
\begin{align}\label{MlocjP4+}
& \| \mathcal{M}_{loc,j}  \|_{L^{\frac{8M+4}{3M+6}}  \rightarrow L^{\frac{8M+4}{M+2}} } \nonumber\\
 &\le (1+c^{\frac{M+2}{8M+4}}) \biggl( \sum_{ l \ge \frac{j}{M}} 2^{-l} 2^{-\frac{(2M+4)(M-1)}{8M+4}l} 2^{\frac{18M-36}{8M+4} (\frac{j}{6} + \frac{Ml}{3}) } 2^{\frac{40-10M}{8M+4}(\frac{3j}{10} + \frac{2Ml}{5}) } \nonumber\\
 &+  \sum_{l_{0} \le l < \frac{j}{M}} 2^{-l} 2^{-\frac{(2M+4)(M-1)}{8M+4}l} 2^{\frac{18M-36}{8M+4} \cdot \frac{Ml}{2}} 2^{\frac{40-10M}{8M+4} \cdot \frac{7Ml}{10}} \biggl)  \nonumber\\
 & \lesssim (1+c^{\frac{M+2}{8M+4}})  j .
\end{align}

$\bullet$  \textbf{Estimation for $\mathcal{M}_{loc,j}$ at $P_{5}$.}  We recall that $M=5$ and $P_{5} = (\frac{1}{2},\frac{1}{2} \cdot \frac{M-2}{M+2})$. By interpolating between (\ref{L2L41}) and (\ref{L2L61}) when $j \le Ml$, and between (\ref{L2L42}) and (\ref{L2L62}) when $j >Ml$, and applying inequality (\ref{frequencydecompose1}), we can obtain
\begin{align}
& \| \mathcal{M}_{loc,j}  \|_{L^{2}  \rightarrow L^{\frac{2(M+2)}{M-2}} } \nonumber\\
 &\le (1+c^{\frac{1}{2}\cdot \frac{M-2}{M+2}}) \biggl( \sum_{ l \ge \frac{j}{M}} 2^{-l} 2^{-\frac{2(M-1)l}{M+2}} 2^{\frac{18-3M}{M+2} (\frac{j}{6} + \frac{Ml}{3}) } 2^{\frac{4M-16}{M+2}(\frac{j}{8} + \frac{Ml}{4}) }  \nonumber\\
&\quad  +  \sum_{l_{0} \le l < \frac{j}{M}} 2^{-l} 2^{-\frac{2(M-1)}{M+2}l} 2^{\frac{18-3M}{M+2} \cdot \frac{Ml}{2}} 2^{\frac{4M-16}{M+2}  (-\frac{j}{4}+ \frac{5Ml}{8})} \biggl)  \nonumber\\
 & \lesssim (1+c^{\frac{1}{2}\cdot \frac{M-2}{M+2}}) . \nonumber
\end{align}

Then we finish the proof  of Theorem \ref{mainth3}.

\subsection{Proof of Theorem \ref{kernel estimate3}  }\label{proofofmainth4}
We first prove Lemma \ref{kernel estimate4} and Lemma \ref{kernel estimate1} below, which will be helpful in the proof of Theorem \ref{kernel estimate3}. Then we state the proof of Theorem \ref{kernel estimate3} at the end of this subsection.

\begin{lemma}\label{kernel estimate4}
Let $\xi^{\prime} =(\xi_{1}^{\prime},\xi_{2}^{\prime},\xi^{\prime}_{3})$. Define
\begin{align}
\mathcal{K}_{1}(y,t)&=  2^{3j+k} \int_{\mathbb{R}^{3}} e^{i 2^{j}\bigl[y_{1}  \xi_{1}^{\prime}+ 2^{k} y_{2}\xi_{2}^{\prime}+  y_{3} \xi_{3}^{\prime} + t\xi_{3}^{\prime}\Psi_{1}(\frac{\xi^{\prime}_{1}}{\xi^{\prime}_{3}}) \bigl]}   b_{j,k}(\xi^{\prime},t) d\xi_{1}^{\prime}d\xi_{2}^{\prime}d\xi_{3}^{\prime},
\end{align}
where
\begin{equation}\label{Psi1}
\Psi_{1}(\frac{\xi^{\prime}_{1}}{\xi^{\prime}_{3}})= c_{d}(-\frac{\xi^{\prime}_{1}}{\xi^{\prime}_{3}})^{d/(d-1)},
\end{equation}
$b_{j,k}$ is a smooth, compactly supported function with  $\mathrm{supp}$ $b_{j,k}(\cdot,t) \subset \{ \xi^{\prime} \in \mathbb{R}^{3}: |\xi_{3}^{\prime}|  \le |\xi_{1}^{\prime}| \le 2|\xi_{3}^{\prime}|, |\xi^{\prime}_{2}| \le 1 , |\xi^{\prime}_{3}| \sim 1 \}$. Moreover, for each multi-index $\alpha$,  $b_{j,k}(\xi^{\prime},t)$ satisfies
\begin{equation}
|D^{\alpha}_{\xi^{\prime}} b_{j,k}(\xi^{\prime},t)|  \lesssim 1
\end{equation}
  uniformly for each $(\xi^{\prime},t) \in$  $\mathrm{supp}$ $b_{j,k}$.\\
(1) For each $t \in [1/2,4]$, we have the uniform estimate
\begin{equation}\label{L1E1}
\|\mathcal{K}_{1}(\cdot,t)\|_{L^{1}(\mathbb{R}^{3})} \lesssim 2^{\frac{j}{2}}.
\end{equation}
(2) For each $t \in (0,+\infty)$ and $y \in \mathbb{R}^{3}$, there holds
\begin{equation}\label{decay1}
|\mathcal{K}_{1}(y,t)| \lesssim 2^{3j+k}(1+2^{j}t)^{-\frac{1}{2}},
\end{equation}
where the implied constant does not depend on $y$ and $t$.
\end{lemma}
\begin{proof}
\textbf{(1)} We take partition of unity, and decompose $\mathcal{K}_{1}$ as
\begin{equation}\label{decompose1}
\mathcal{K}_{1}(y,t) = \sum_{\kappa \in \mathbb{I}_{1}} \mathcal{K}_{1,\kappa}(y,t),
\end{equation}
 where $\mathbb{I}_{1}:= 2^{-\frac{j}{2}}  \mathbb{Z} \cap [1,2]$,
\[ \mathcal{K}_{1,\kappa}(y,t)=2^{3j+k} \int_{\mathbb{R}^{3}} e^{i 2^{j}\bigl[y_{1}  \xi_{1}^{\prime}+ 2^{k} y_{2}\xi_{2}^{\prime}+  y_{3} \xi_{3}^{\prime} + t \xi^{\prime}_{3}\Psi_{1}(\frac{\xi^{\prime}_{1}}{\xi^{\prime}_{3}}) \bigl]}   b_{j,k,\kappa}(\xi^{\prime},t) d\xi_{1}^{\prime}d\xi_{2}^{\prime}d\xi_{3}^{\prime}, \]
and $b_{j,k, \kappa}(\xi^{\prime},t) = b_{j,k}(\xi^{\prime},t) \tilde{\chi}\biggl( 2^{\frac{j}{2}} (\frac{\xi^{\prime}_{1}}{\xi^{\prime}_{3}}-\kappa)\biggl)$, $\tilde{\chi}$ is a smooth cutoff function supported in $[1,2]$. For fixed $\kappa$, changing variables $\xi^{\prime}_{1} \rightarrow 2^{-\frac{j}{2}}  \eta_{1} + \kappa \eta_{3}$, $\xi^{\prime}_{2} \rightarrow \eta_{2}$, $\xi^{\prime}_{3} \rightarrow \eta_{3}$, we denote the phase function by
\begin{align}
 \Psi_{j,k,y,t}(\eta)= 2^{\frac{j}{2}} y_{1}\eta_{1} + 2^{j+k}y_{2}\eta_{2} + 2^{j}(y_{3}+ \kappa y_{1} )\eta_{3}
  +t2^{j} \eta_{3}\Psi_{1} \biggl( 2^{-\frac{j}{2}}\frac{\eta_{1} }{\eta_{3}}+\kappa \biggl)  . \nonumber
\end{align}
By Taylor expansion of $\Psi_{1} \biggl(2^{-\frac{j}{2}} \frac{\eta_{1} }{\eta_{3}}+\kappa \biggl) $ at $\eta_{1}=0$,
\begin{align}
  \Psi_{j,k,y,t}(\eta) &=2^{\frac{j}{2}} (y_{1} + t\Psi_{1}^{\prime}(\kappa))\eta_{1} + 2^{j+k}y_{2}\eta_{2} + 2^{j}(y_{3}+ \kappa y_{1} + t\Psi_{1}(\kappa) )\eta_{3} \nonumber\\
 &\quad + t\eta_{3} \bigl[\frac{\eta_{1}^{2}}{ 2\eta^{2}_{3}}\Psi_{1}^{\prime \prime}  ( \kappa  ) +   \tilde{R}_{j}(\frac{\eta_{1}}{\eta_{3}},   \kappa) \bigl], \nonumber
\end{align}
where we have put all the remainder terms into $\tilde{R}_{j}$.
Since
\[\biggl|\nabla_{\eta} \biggl(t\eta_{3} \bigl[\frac{\eta_{1}^{2}}{ 2\eta^{2}_{3}}\Psi_{1}^{\prime \prime}  ( \kappa  ) +   \tilde{R}_{j}(\frac{\eta_{1}}{\eta_{3}},   \kappa) \bigl] \biggl)\biggl| \lesssim 1, \]
we have
\[|\nabla_{\eta}\Psi_{j,k,y,t}(\eta)| \gtrsim 2^{\frac{j}{2}} |y_{1} + t\Psi_{1}^{\prime}(\kappa)| +2^{j+k}|y_{2}| +2^{j}|y_{3}+ \kappa y_{1} + t\Psi_{1}(\kappa) | \]
provided that
\[2^{\frac{j}{2}} |y_{1} + t\Psi_{1}^{\prime}(\kappa)| +2^{j+k}|y_{2}| +2^{j}|y_{3}+ \kappa y_{1} + t\Psi_{1}(\kappa) | \gg 1.\]
Then integration by parts yields
\[|\mathcal{K}_{1,\kappa}(y,t)| \le  2^{\frac{5j}{2}+k}    (1+2^{\frac{j}{2}} |y_{1} + t\Psi_{1}^{\prime}(\kappa)| +2^{j+k}|y_{2}| +2^{j}|y_{3}+ \kappa y_{1} + t\Psi_{1}(\kappa) |)^{-4}.\]
Then we have for $t\in [1/2,4]$,
\[\|\mathcal{K}_{1,\kappa}(\cdot,t)\|_{L^{1}} \lesssim 1. \]
Combining with (\ref{decompose1}) and the fact that the number of the elements in $ \mathbb{I}_{1}$ is bounded by $2^{\frac{j}{2}}$, we arrive at  (\ref{L1E1}).

\textbf{(2) } If $t \lesssim 2^{-j}$,
\[|\mathcal{K}_{1}(y,t)| \lesssim 2^{3j +k } \lesssim 2^{3j+k}(1+2^{j}t)^{-\frac{1}{2}}, \]
then (\ref{decay1}) follows. Next we consider the case $t \gg 2^{-j}$,
we rewrite $\mathcal{K}_{1}(y,t)$ as
\begin{align}
\mathcal{K}_{1}(y,t)&=  2^{3j+k} \int_{\mathbb{R}^{2}} e^{i 2^{j}\bigl( 2^{k} y_{2}\xi_{2}^{\prime}+  y_{3} \xi_{3}^{\prime}   \bigl)} \int_{\mathbb{R}}  e^{i 2^{j}\xi_{3}^{\prime}\bigl[y_{1}  \frac{\xi_{1}^{\prime}}{\xi^{\prime}_{3}}   + tc_{d} \Psi_{1}(\frac{\xi^{\prime}_{1}}{\xi^{\prime}_{3}}) \bigl]} b_{j,k}(\xi^{\prime},t) d\xi_{1}^{\prime}d\xi_{2}^{\prime}d\xi^{\prime}_{3}. \nonumber
\end{align}
We change variables $\xi_{1}^{\prime} \rightarrow \eta_{1}\eta_{3}$, $\xi^{\prime}_{2} \rightarrow \eta_{2}$, $\xi_{3}^{\prime} \rightarrow \eta_{3}$,
\begin{align}\label{scaling1}
\mathcal{K}_{1}(y,t)&=  2^{3j+k} \int_{\mathbb{R}^{2}} e^{i 2^{j}\bigl( 2^{k} y_{2}\eta_{2}+  y_{3} \eta_{3}   \bigl)} \int_{\mathbb{R}}  e^{i 2^{j}t\eta_{3}\bigl[\frac{y_{1}}{t}  \eta_{1}  +  c_{d} \Psi_{1}(\eta_{1}) \bigl]}b_{j,k} (\tilde{\eta},t)  \eta_{3} d\eta_{1} d\eta_{2}d\eta_{3},
\end{align}
where $\tilde{\eta}=(\eta_{1}\eta_{3},\eta_{2},\eta_{3})$. We consider the integral
\[I_{1}:= \int_{\mathbb{R}}  e^{i 2^{j}t\eta_{3}\bigl[\frac{y_{1}}{t}  \eta_{1}  +  c_{d} \Psi_{1}(\eta_{1}) \bigl]} b_{j,k} (\tilde{\eta},t) d\eta_{1}.\]
Notice that $|\eta_{1}| \sim 1$ and $|\eta_{3}| \sim 1$, then
\[\biggl|\partial^{2}_{\eta_{1}} \biggl(\eta_{3}\bigl[\frac{y_{1}}{t}  \eta_{1}  +  c_{d} \Psi_{1}(\eta_{1}) \bigl]\biggl)\biggl| \sim |\eta_{1}|^{\frac{2-d}{d-1}} \sim 1. \]
Hence, we can apply von der Corupt's lemma to obtain that
\begin{equation}\label{osc1}
|I_{1}| \lesssim (2^{j}t)^{-\frac{1}{2}}.
\end{equation}
Then (\ref{decay1}) follows from (\ref{scaling1}) and (\ref{osc1}).
\end{proof}

\begin{lemma}\label{kernel estimate1}
Let
\begin{equation}\label{kernel4}
\mathcal{K}_{2}(y,t) =2^{3j} \int_{\mathbb{R}^{3}} e^{i 2^{j} \bigl[y \cdot \xi^{\prime} + t\xi^{\prime}_{3}\bigl( \Psi_{1}(\frac{\xi^{\prime}_{1}}{\xi^{\prime}_{3}})+ 2^{-Ml} \Psi_{2} (\frac{\xi^{\prime}_{1}}{\xi^{\prime}_{3}},\frac{\xi^{\prime}_{2}}{\xi^{\prime}_{1}})\bigl) \bigl] }   b_{j,l}(\xi^{\prime},t) d\xi^{\prime},
\end{equation}
where $j > Ml$, $\Psi_{1}$ can be found in (\ref{Psi1}), and
\begin{equation}\label{Psi2}
\Psi_{2}(\frac{\xi^{\prime}_{1}}{\xi^{\prime}_{3}},\frac{\xi^{\prime}_{2}}{\xi^{\prime}_{1}})  = c_{d,M}  \bigl(-\frac{\xi^{\prime}_{1}}{\xi^{\prime}_{3}}\bigl)^{d/(d-1)} \bigl(-\frac{\xi^{\prime}_{2}}{\xi^{\prime}_{1}}\bigl)^{M/(M-1)} + 2^{- l} R_{l}(\frac{\xi^{\prime}_{1}}{\xi^{\prime}_{3}}, \frac{\xi^{\prime}_{2}}{\xi^{\prime}_{1}}),
\end{equation}
$R_{l}$ is the same function as that in  (\ref{Rl}), and $b_{j, l}$ is a smooth, compactly supported function with  $\mathrm{supp}$ $b_{j,l}(\cdot,t) \subset \{ \xi^{\prime} \in \mathbb{R}^{3}: |\xi_{3}^{\prime}|  \le |\xi_{1}^{\prime}| \le 2|\xi_{3}^{\prime}|,  |\xi_{3}^{\prime}|  \le |\xi_{2}^{\prime}| \le 2|\xi^{\prime}_{3}|,  |\xi^{\prime}_{3}| \sim 1  \}$.
For each multi-index $\alpha$,  $b_{j,l}(\xi^{\prime},t)$ satisfies
\begin{equation}
|D^{\alpha}_{\xi^{\prime}} b_{j,l}(\xi^{\prime},t)|  \lesssim 1
\end{equation}
  uniformly for  each $(\xi^{\prime},t) \in$ $\mathrm{supp}$ $b_{j,l}$. \\
 (1) For each  $t \in [1/2,4]$, we obtain
 \begin{equation}\label{L1E2}
\|\mathcal{K}_{2}(\cdot,t)\|_{L^{1}(\mathbb{R}^{3})} \lesssim 2^{j-\frac{Ml}{2}}.
 \end{equation}
(2) For each $t>0$ and $y \in \mathbb{R}^{3}$, there holds the uniform estimate
\begin{equation}\label{decay2}
|\mathcal{K}_{2}(y,t)| \lesssim 2^{3j}(1+ 2^{j}t)^{-\frac{1}{2}}(1+ 2^{j-Ml}t)^{-\frac{1}{2}}.
\end{equation}
\end{lemma}

\begin{proof}
\textbf{(1)}  We decompose
\begin{equation}\label{decompose2}
\mathcal{K}_{2}(y,t) = \sum_{\kappa \in \mathbb{I}_{1}}\sum_{\tilde{\kappa} \in \mathbb{I}_{2}} \mathcal{K}_{2,\kappa, \tilde{\kappa}}(y,t),
\end{equation}
 where
\[ \mathcal{K}_{2,\kappa, \tilde{\kappa}}(y,t)=2^{3j} \int_{\mathbb{R}^{3}} e^{i 2^{j} \bigl[y \cdot \xi^{\prime} + t\xi^{\prime}_{3}\bigl( \Psi_{1}(\frac{\xi^{\prime}_{1}}{\xi^{\prime}_{3}})+ 2^{-Ml} \Psi_{2} (\frac{\xi^{\prime}_{1}}{\xi^{\prime}_{3}},\frac{\xi^{\prime}_{2}}{\xi^{\prime}_{1}})\bigl) \bigl] }   b_{j,l,\kappa,\tilde{\kappa}}(\xi^{\prime},t) d\xi^{\prime},\]
and
$b_{j,l,\kappa,\tilde{\kappa}}(\xi^{\prime},t) =b_{j,l}(\xi^{\prime},t) \tilde{\chi}\biggl(2^{\frac{j}{2}}(\frac{\xi^{\prime}_{1}}{\xi^{\prime}_{3}}-\kappa)\biggl)\tilde{\chi}\biggl( 2^{\frac{j-Ml}{2}} (\frac{\xi^{\prime}_{2}}{\xi^{\prime}_{3}}-\tilde{\kappa})\biggl)$.
Then we change variables $\xi^{\prime}_{1} \rightarrow 2^{-\frac{j}{2}} \eta_{1} + \kappa \eta_{3}$, $\xi^{\prime}_{2}\rightarrow 2^{-\frac{j-Ml}{2}}\eta_{2} + \tilde{\kappa}\eta_{3}$, $\xi^{\prime}_{3}\rightarrow \eta_{3}$, and denote  the phase function as follows
\begin{align}
\Psi_{j,l,y,t}(\eta)& =2^{\frac{j}{2}}  y_{1}\eta_{1} + 2^{\frac{j+Ml}{2}} y_{2}\eta_{2} + 2^{j}(y_{3}+ \kappa y_{1} + \tilde{\kappa}y_{2})\eta_{3}  \nonumber\\
&\quad +t2^{j} \eta_{3}\Psi_{1} \biggl( 2^{-\frac{j}{2}}\frac{\eta_{1}}{\eta_{3}} +\kappa\biggl) + t 2^{j-Ml}\eta_{3}\Psi_{2} \biggl(2^{-\frac{j}{2}}\frac{\eta_{1}}{\eta_{3}} +\kappa,\frac{2^{-\frac{j-Ml}{2}}\eta_{2} + \tilde{\kappa}\eta_{3}}{2^{-\frac{j}{2}}\eta_{1} + \kappa \eta_{3}}\biggl). \nonumber
\end{align}
Then by Taylor expansion of
\[\Psi_{1} \biggl( 2^{-\frac{j}{2}}\frac{\eta_{1}}{\eta_{3}} +\kappa\biggl)  + 2^{-Ml}\Psi_{2} \biggl(2^{-\frac{j}{2}}\frac{\eta_{1}}{\eta_{3}} +\kappa,\frac{2^{-\frac{j-Ml}{2}}\eta_{2} + \tilde{\kappa}\eta_{3}}{2^{-\frac{j}{2}}\eta_{1} + \kappa \eta_{3}}\biggl)\]
 at $(\eta_{1},\eta_{2})=(0,0)$,
we have
\begin{align}
&\Psi_{j,l,y,t}(\eta) \nonumber\\
&=2^{\frac{j}{2}} \biggl(y_{1}+ t \Psi_{1}^{\prime}(\kappa) + t 2^{-Ml} \partial_{1}\Psi_{2}(\kappa, \frac{\tilde{\kappa}}{\kappa}) -t2^{-Ml} \frac{\tilde{\kappa}}{\kappa^{2}} \partial_{2} \Psi_{2} (\kappa, \frac{\tilde{\kappa}}{\kappa}) \biggl)\eta_{1} \nonumber\\
 &+ 2^{\frac{j+Ml}{2}}  \biggl(y_{2} + t 2^{-Ml} \frac{1}{\kappa} \partial_{2}\Psi_{2}(\kappa,\frac{\tilde{\kappa}}{\kappa}) \biggl) \eta_{2} + 2^{j}\biggl(y_{3}+ \kappa y_{1} + \tilde{\kappa}y_{2} + t\Psi_{1}(\kappa) + t 2^{-Ml}\Psi_{2}(\kappa, \frac{\tilde{\kappa}}{\kappa}) \biggl)\eta_{3} \nonumber\\
 &+ t\frac{\eta_{1}^{2}}{ 2\eta_{3}}  \Psi_{1}^{\prime \prime}  (\kappa) + t\frac{\eta_{2}^{2}}{2\eta_{3}}   \frac{1}{\kappa^{2}} \partial_{22}\Psi_{2}(\kappa, \frac{\tilde{\kappa}}{\kappa})     + R_{j,l}(\eta, t, \kappa,\tilde{\kappa}), \nonumber
\end{align}
where all the remainder terms are included in $R_{j,l}$.
Since
\[ \biggl|\nabla_{\eta} \biggl[t\frac{\eta_{1}^{2}}{ 2\eta_{3}}  \Psi_{1}^{\prime \prime}  (\kappa) + t\frac{\eta_{2}^{2}}{2\eta_{3}}   \frac{1}{\kappa^{2}} \partial_{22}\Psi_{2}(\kappa, \frac{\tilde{\kappa}}{\kappa})     + R_{j,l}(\eta, t, \kappa,\tilde{\kappa}) \biggl]\biggl| \lesssim 1.\]
Integration by parts implies that
\begin{align}\label{kernel3}
  &|\mathcal{K}_{2,\kappa,\tilde{\kappa}}(y,t)|  \lesssim   2^{2j + \frac{Ml}{2}}  \nonumber\\
  &  \times (1+2^{\frac{j}{2}} |y_{1} +c_{1,t}(\kappa, \tilde{\kappa})| +2^{\frac{j+Ml}{2}} |y_{2} + c_{2,t}(\kappa, \tilde{\kappa})| +2^{j}|y_{3} + c_{3,t}(y_{1},y_{2},\kappa, \tilde{\kappa})|)^{-4},
  \end{align}
 where
  $$c_{1,t}(\kappa, \tilde{\kappa})= t \Psi_{1}^{\prime}(\kappa) + t 2^{-Ml} \partial_{1}\Psi_{2}(\kappa, \frac{\tilde{\kappa}}{\kappa}) -t2^{-Ml} \frac{\tilde{\kappa}}{\kappa^{2}} \partial_{2} \Psi_{2} (\kappa, \frac{\tilde{\kappa}}{\kappa})  ,$$
  $$c_{2,t}(\kappa,\tilde{\kappa})= t 2^{-Ml} \frac{1}{\kappa} \partial_{2}\Psi_{2}(\kappa,\frac{\tilde{\kappa}}{\kappa}),$$
$$c_{3,t}(y_{1},y_{2},\kappa, \tilde{\kappa})= \kappa y_{1} + \tilde{\kappa}y_{2} + t\Psi_{1}(\kappa) + t 2^{-Ml}\Psi_{2}(\kappa, \frac{\tilde{\kappa}}{\kappa}) .$$
Then we have
\[\|\mathcal{K}_{2,\kappa,\tilde{\kappa}}(\cdot,t)\|_{L^{1}(\mathbb{R}^{3})} \lesssim 1. \]
By (\ref{decompose2}) and  the facts that the number of elements in  $ \mathbb{I}_{1} $ is bounded by $2^{\frac{j}{2}}$,  the number of  elements in $ \mathbb{I}_{2} $ is bounded by $2^{\frac{j-Ml}{2}}$, we complete the proof of (\ref{L1E2}).

\textbf{(2)} If $t \lesssim 2^{-j}$, then
\[|\mathcal{K}_{2}(y,t)| \lesssim 2^{3j} \lesssim 2^{3j}(1+2^{j}t)^{-1},\]
this yields  (\ref{decay2}) since $(1+2^{j}t)^{-\frac{1}{2}} \le (1+2^{j-Ml}t)^{-\frac{1}{2}} $.

We consider the case $t \gg  2^{-j}$ in what follows, and rewrite
 \begin{equation}
\mathcal{K}_{2}(y,t) =2^{3j} \int_{\mathbb{R}}e^{i2^{j}\xi^{\prime}_{3}y_{3}} \int_{\mathbb{R}^{2}} e^{i 2^{j} t\xi^{\prime}_{3} \bigl[\frac{y_{1}}{t} \cdot \frac{\xi^{\prime}_{1}}{\xi^{\prime}_{3}}+\frac{y_{2}}{t} \cdot \frac{\xi_{2}^{\prime}}{\xi_{3}^{\prime}} + \bigl( \Psi_{1}(\frac{\xi^{\prime}_{1}}{\xi^{\prime}_{3}})+ 2^{-Ml} \Psi_{2} (\frac{\xi^{\prime}_{1}}{\xi^{\prime}_{3}},\frac{\xi^{\prime}_{2}}{\xi^{\prime}_{1}})\bigl) \bigl] }   b_{j,l}(\xi^{\prime},t) d\xi_{1}^{\prime}d\xi_{2}^{\prime}d\xi_{3}^{\prime}. \nonumber
\end{equation}
 We change variables $\xi^{\prime}_{1} \rightarrow \eta_{1}\eta_{3}$, $\xi^{\prime}_{2} \rightarrow \eta_{2}\eta_{3}$, $\xi_{3}^{\prime} \rightarrow \eta_{3}$, then $\mathcal{K}_{2}(y,t) $ equals
 \begin{align}\label{scaling2}
2^{3j} \int_{\mathbb{R}}e^{i2^{j}\eta_{3}y_{3}} \int_{\mathbb{R}} \int_{\mathbb{R}} e^{i 2^{j} t\eta_{3}   \bigl[\frac{y_{1}}{t} \cdot \eta_{1}+\frac{y_{2}}{t} \cdot \eta_{2} +  \Psi_{1}(\eta_{1})+ 2^{-Ml} \widetilde{\Psi}_{2} (\eta_{1},\eta_{2})  \bigl] }   b_{j,l}(\tilde{\eta},t)\eta^{2}_{3} d\eta_{1} d\eta_{2} d\eta_{3} ,
\end{align}
where $\tilde{\eta}=(\eta_{1}\eta_{3},\eta_{2}\eta_{3},\eta_{3})$,
 \[\widetilde{\Psi}_{2} (\eta_{1},\eta_{2})  = c_{d,M}  \bigl(-\eta_{1}\bigl)^{d/(d-1)} \bigl(-\frac{\eta_{2}}{\eta_{1}}  \bigl)^{M/(M-1)}  + 2^{- l} R_{l}(\eta_{1}, \frac{\eta_{2}}{ \eta_{1}}   ). \]
We will consider  the case $2^{-j} \ll t \lesssim 2^{-j+Ml}$ and $t \gg 2^{-j+Ml}$, respectively.

When $2^{-j} \ll t \lesssim 2^{-j+Ml}$, we consider the integral
\[I_{2}:= \int_{\mathbb{R}} e^{i 2^{j} t\eta_{3}   \bigl[\frac{y_{1}}{t} \cdot \eta_{1}+\frac{y_{2}}{t} \cdot \eta_{2} + \bigl( \Psi_{1}(\eta_{1})+ 2^{-Ml} \widetilde{\Psi}_{2} (\eta_{1},\eta_{2})\bigl) \bigl] }   b_{j,l}(\tilde{\eta},t) d\eta_{1}.\]
We notice that $|\eta_{i}|\sim 1$ for $i = 1,2,3$, and  $l \gg 1$, then
\[\biggl| \partial^{2}_{\eta_{1}}\bigl[\frac{y_{1}}{t} \cdot \eta_{1}+\frac{y_{2}}{t} \cdot \eta_{2} +  \Psi_{1}(\eta_{1})+ 2^{-Ml} \widetilde{\Psi}_{2} (\eta_{1},\eta_{2}) \bigl]\biggl| = \biggl|\Psi^{\prime \prime}_{1}(\eta_{1})  + 2^{-Ml} \partial^{2}_{\eta_{1}}\widetilde{\Psi}_{2} (\eta_{1},\eta_{2}) \biggl| \gtrsim |\Psi^{\prime \prime}_{1}(\eta_{1})| \gtrsim 1. \]
Then we can apply von der Corupt's lemma to obtain that
\begin{equation}\label{osc2}
|I_{2}| \lesssim (2^{j}t)^{-\frac{1}{2}}.
\end{equation}
Inequlities (\ref{scaling2}) and (\ref{osc2})   yield
\[|\mathcal{K}_{2}(y,t)| \lesssim 2^{3j}(1+2^{j}t)^{-\frac{1}{2}} \lesssim 2^{3j}(1+2^{j}t)^{-\frac{1}{2}} (1+2^{j-Ml}t)^{-\frac{1}{2}},\]
and (\ref{decay2}) follows.

When $t \gg 2^{-j+Ml}$, we change the order of integration and rewrite $\mathcal{K}_{2}(y,t)$ as
 \begin{align}
2^{3j} \int_{\mathbb{R}}e^{i2^{j}\eta_{3}y_{3}} \int_{\mathbb{R}}  e^{i 2^{j} t\eta_{3}   \bigl[\frac{y_{1}}{t} \cdot \eta_{1} + \Psi_{1}(\eta_{1})  \bigl] } \int_{\mathbb{R}} e^{i 2^{j-Ml} t\eta_{3}   \bigl[ \frac{2^{Ml}y_{2}}{t} \cdot \eta_{2} +  \widetilde{\Psi}_{2} (\eta_{1},\eta_{2})  \bigl] }   b_{j,l}(\tilde{\eta},t) d\eta_{2} d\eta_{1}\eta^{2}_{3} d\eta_{3} .\nonumber
\end{align}
We first consider the integral
\[I_{4}:= \int_{\mathbb{R}} e^{i 2^{j-Ml} t\eta_{3}   \bigl[ \frac{2^{Ml}y_{2}}{t} \cdot \eta_{2} +  \widetilde{\Psi}_{2} (\eta_{1},\eta_{2})  \bigl] }   b_{j,l}(\tilde{\eta},t) d\eta_{2}.\]
We denote $y^{\prime}_{2} = \frac{2^{Ml}y_{2}}{t}$ and
\[\psi(y^{\prime}_{2},\eta_{1},\eta_{2})= y^{\prime}_{2}\eta_{2} + \widetilde{\Psi}_{2}(\eta_{1},\eta_{2}). \]
Since
\[\partial_{\eta_{2}}\psi(y^{\prime}_{2},\eta_{1},\eta_{2})= y^{\prime}_{2}  + \partial_{\eta_{2}}\widetilde{\Psi}_{2}(\eta_{1},\eta_{2}), \]
and $|\partial_{\eta_{2}}\widetilde{\Psi}_{2}(\eta_{1},\eta_{2})| \sim 1$, the phase function may have critical point only if $|y^{\prime}_{2}|\sim 1$. Hence, we will consider two cases (i) $|y^{\prime}_{2}| \ll 1$ or $|y^{\prime}_{2}| \gg 1$;  (ii) $|y^{\prime}_{2}| \sim 1$, respectively.

\textbf{(i)}   $|y^{\prime}_{2}| \ll 1$ or $|y^{\prime}_{2}| \gg 1$. It follows that
\[|\partial_{\eta_{2}}\psi(y^{\prime}_{2},\eta_{1},\eta_{2})| \gtrsim 1. \]
We first apply integration by parts  in $I_{4}$,  then apply von der Corupt's lemma with respect to the variable $\eta_{1}$ for $\mathcal{K}_{2}(y,t)$. It can be obtained that
\[|\mathcal{K}_{2}(y,t)| \lesssim 2^{3j}(2^{j-Ml}t)^{-1} (2^{j}t)^{-\frac{1}{2}},\]
which implies (\ref{decay2}) as $2^{j-Ml}t\ge 1$.

\textbf{(ii)} $|y^{\prime}_{2}| \sim 1$. Since
\[\partial^{2}_{\eta_{2}}\psi(y^{\prime}_{2},\eta_{1},\eta_{2})  = \partial^{2}_{\eta_{2}}\widetilde{\Psi}_{2}(\eta_{1},\eta_{2}) \neq 0, \]
by implicit function theorem, there exists a unique solution $\eta^{c}_{2}(y^{\prime}_{2},\eta_{1})$ to the equation
\[\partial_{\eta_{2}}\psi(y^{\prime}_{2},\eta_{1},\eta_{2}) = 0.\]
By the method of stationary phase,
\[I_{4}= (2^{j-Ml}t)^{-\frac{1}{2}}e^{i2^{j-Ml}t\eta_{3} \psi\bigl(y^{\prime}_{2},\eta_{1},\eta^{c}_{2}(y^{\prime}_{1},\eta_{1})\bigl)}\widetilde{b}_{j,l,y_{2}}(\tilde{\eta},t)+ \widetilde{r}_{j,l,y_{2}}(\tilde{\eta},t),\]
where $\widetilde{b}_{j,l,y_{2}}(\tilde{\eta},t)$ satisfies
\[|D^{\alpha}_{\eta}\widetilde{b}_{j,l,y_{2}}(\tilde{\eta},t)| \lesssim 1\]
for each multi-index $\alpha$,
and $\widetilde{r}_{j,l,y_{2}}(\tilde{\eta},t)$ satisfies
\[|D^{\alpha}_{\eta}\widetilde{r}_{j,l,y_{2}}(\tilde{\eta},t)| \lesssim (2^{j-Ml}t)^{-N}, \quad \forall N \in \mathbb{N}^{+}.\]
Notice that the estimate for the remainder term
\[  \int_{\mathbb{R}}  e^{i 2^{j} t\eta_{3}   \bigl[\frac{y_{1}}{t} \cdot \eta_{1} + \Psi_{1}(\eta_{1})   \bigl] }\widetilde{r}_{j,l,y_{2}}(\tilde{\eta},t)   d\eta_{1}\]
is the same as  in case (i). Next we consider only the estimate for
\[I_{5} := (2^{j-Ml}t)^{-\frac{1}{2}} \int_{\mathbb{R}}  e^{i 2^{j} t\eta_{3}   \bigl[\frac{y_{1}}{t} \cdot \eta_{1} + \Psi_{1}(\eta_{1}) +2^{-Ml}  \psi\bigl(y^{\prime}_{2},\eta_{1},\eta^{c}_{2}(y^{\prime}_{2},\eta_{1})\bigl)  \bigl] }\widetilde{b}_{j,l,y_{2}}(\tilde{\eta},t)   d\eta_{1}. \]

Before we can apply von der Corupt's lemma to dominate  the term $I_{5}$, we need to show \[|\partial_{\eta_{1}}^{2}\psi\bigl(y^{\prime}_{2},\eta_{1},\eta^{c}_{2}(y^{\prime}_{2},\eta_{1})\bigl)| \lesssim 1,\]
so that the term $2^{-Ml}\psi\bigl(y^{\prime}_{2},\eta_{1},\eta^{c}_{2}(y^{\prime}_{2},\eta_{1})\bigl)$ can be regarded  as a small perturbation. Indeed,
\begin{align}
\partial_{\eta_{1}}  \psi\bigl(y^{\prime}_{2},\eta_{1},\eta^{c}_{2}(y^{\prime}_{2},\eta_{1})\bigl)&= y^{\prime}_{2} \partial_{\eta_{1}}\eta^{c}_{2}(y^{\prime}_{2},\eta_{1})+ (\partial_{1}\widetilde{\Psi}_{2})(\eta_{1},\eta^{c}_{2}(y^{\prime}_{2},\eta_{1})) +   \partial_{\eta_{1}}\eta^{c}_{2}(y^{\prime}_{2},\eta_{1}) (\partial_{2}\widetilde{\Psi}_{2})(\eta_{1},\eta^{c}_{2}(y^{\prime}_{2},\eta_{1})) \nonumber\\
&= (\partial_{1}\widetilde{\Psi}_{2})(\eta_{1},\eta^{c}_{2}(y^{\prime}_{2},\eta_{1})), \nonumber
\end{align}
where we have applied the equation
\[y^{\prime}_{2}  +    (\partial_{2}\widetilde{\Psi}_{2})(\eta_{1},\eta^{c}_{2}(y^{\prime}_{2},\eta_{1}))= 0. \]
This  equation also implies that
\[   (\partial_{21}\widetilde{\Psi}_{2})(\eta_{1},\eta^{c}_{2}(y^{\prime}_{2},\eta_{1}))+ \partial_{\eta_{1}}\eta^{c}_{2}(y^{\prime}_{2},\eta_{1})  (\partial_{22}\widetilde{\Psi}_{2})(\eta_{1},\eta^{c}_{2}(y^{\prime}_{2},\eta_{1}))= 0.\]
Since $(\partial_{22}\widetilde{\Psi}_{2})(\eta_{1},\eta^{c}_{2}(y^{\prime}_{2},\eta_{1})) \neq 0$, we have
\[|\partial_{\eta_{1}}\eta^{c}_{2}(y^{\prime}_{2},\eta_{1})|   =\biggl| \frac{ (\partial_{21}\widetilde{\Psi}_{2})\bigl(\eta_{1},\eta^{c}_{2}(y^{\prime}_{2},\eta_{1})\bigl)}
{(\partial_{22}\widetilde{\Psi}_{2})\bigl(\eta_{1},\eta^{c}_{2}(y^{\prime}_{2},\eta_{1})\bigl)} \biggl| \lesssim 1,\]
where the implied constant does not depend on $l$.
It follows that
\begin{align}\label{phaselowerb}
&|\partial^{2}_{\eta_{1}}  \psi\bigl(y^{\prime}_{2},\eta_{1},\eta^{c}_{2}(y^{\prime}_{2},\eta_{1})\bigl)| \nonumber\\
&= |(\partial_{11}\widetilde{\Psi}_{2})(\eta_{1},\eta^{c}_{2}(y^{\prime}_{2},\eta_{1})) + \partial_{\eta_{1}}\eta^{c}_{2}(y^{\prime}_{2},\eta_{1}) (\partial_{12}\widetilde{\Psi}_{2})(\eta_{1},\eta^{c}_{2}(y^{\prime}_{2},\eta_{1}))| \lesssim 1.
\end{align}

Now we are ready to estimate $I_{5}$. By (\ref{phaselowerb}) and the fact that $l$ is sufficiently large,
\[ \bigl| \partial^{2}_{\eta_{1}}\bigl[\frac{y_{1}}{t} \cdot \eta_{1} + \Psi_{1}(\eta_{1}) +2^{-Ml}  \psi\bigl(y^{\prime}_{2},\eta_{1},\eta^{c}_{2}(y^{\prime}_{2},\eta_{1})\bigl)  \bigl] \bigl| \sim |\Psi^{\prime\prime}_{1}(\eta_{1})| \sim 1. \]
The von der Corupt's lemma implies that
\begin{equation}\label{osc3}
|I_{5}| \lesssim (2^{j}t)^{-\frac{1}{2}}  (2^{j-Ml}t)^{-\frac{1}{2}}.
\end{equation}
Then we finish the proof of (\ref{decay2})  by (\ref{scaling2}) and (\ref{osc3}).
\end{proof}

\textbf{
Proof of Theorem   \ref{kernel estimate3}:}
\textbf{(I)}  For  $j \le Ml$, $k\in \mathbb{Z}\bigcap [Ml-j,(M-1)l]$, and  $j > Ml$, $k\in \biggl(\mathbb{Z}\bigcap (N_{1},(M-1)l]\biggl)\bigcup\{0\}$,
we change variables $\xi_{1} \rightarrow 2^{j} \xi^{\prime}_{1}$, $\xi_{2} \rightarrow 2^{j+k} \xi^{\prime}_{2}$, $\xi_{3} \rightarrow 2^{j} \xi^{\prime}_{3}$, then
\begin{align}
K_{j,l,k}(y,t)&=  2^{3j+k} \int_{\mathbb{R}^{3}} e^{i 2^{j}\bigl[y_{1}  \xi_{1}^{\prime} + 2^{k}y_{2}\xi_{2}^{\prime}+  y_{3} \xi_{3}^{\prime} -tc_{d}\xi^{\prime}_{3}(-\frac{\xi^{\prime}_{1}}{\xi^{\prime}_{3}})^{d/(d-1)}\bigl]} \nonumber\\
 &\quad \times a_{j,l,k }( \xi^{\prime}_{j,k},t)\chi_{1}(|\xi^{\prime}_{1}|  )  \chi_{1} \biggl(\frac{\xi^{\prime}_{1}}{ \xi^{\prime}_{3} }\biggl)  \chi_{0}(2^{k-(M-1)l }|\xi^{\prime}_{2}|)  d\xi^{\prime}.
\end{align}
The symbol $a_{j,l,k }( \xi^{\prime}_{j,k},t)$ satisfies (\ref{symbolb1}) when $j \le Ml$, $k\in \mathbb{Z}\bigcap (Ml-j,(M-1)l]$, or when  $j > Ml$, $k\in \biggl(\mathbb{Z}\bigcap (N_{1},(M-1)l]\biggl)\bigcup\{0\}$. When $j \le Ml$ and $k=Ml-j$, $a_{j,l,k }( \xi^{\prime}_{j,k},t)$ satisfies (\ref{symbolb4}).   Then we denote
\[b_{j,l,k}(\xi^{\prime},t)= 2^{(j+k-Ml)N}a_{j,l,k }( \xi^{\prime}_{j,k},t)\chi_{1}(|\xi^{\prime}_{1}|  )  \chi_{1} \biggl(\frac{\xi^{\prime}_{1}}{ \xi^{\prime}_{3} }\biggl)  \chi_{0}(2^{k-(M-1)l }|\xi^{\prime}_{2}|),\]
$b_{j,l,k}$ satisfies the condition in Lemma \ref{kernel estimate4}, i.e.,
\[|D_{\xi^{\prime}}^{\alpha}b_{j,l,k}(\xi^{\prime},t)| \lesssim 1\]
for each multi-index $\alpha$. In addition,
\begin{align}
\widetilde{K_{j,l,k}}(y,t,t^{\prime})&=  2^{3j+k} \int_{\mathbb{R}^{3}} e^{i 2^{j}\bigl[y_{1}  \xi_{1}^{\prime} + 2^{k}y_{2}\xi_{2}^{\prime}+  y_{3} \xi_{3}^{\prime} -(t-t^{\prime})c_{d}\xi^{\prime}_{3}(-\frac{\xi^{\prime}_{1}}{\xi^{\prime}_{3}})^{d/(d-1)}\bigl]} \nonumber\\
 &\quad \times a_{j,l,k }( \xi^{\prime}_{j,k},t)\overline{a_{j,l,k }( \xi^{\prime}_{j,k},t^{\prime}) }\chi^{2}_{1}(|\xi^{\prime}_{1} | )  \chi^{2}_{1} \biggl(\frac{\xi^{\prime}_{1}}{ \xi^{\prime}_{3} }\biggl)  \chi^{2}_{0}(2^{k-(M-1)l }|\xi^{\prime}_{2}|)  d\xi^{\prime}.
\end{align}
We denote
\[\widetilde{b_{j,l,k}}(\xi,t,t^{\prime})= 2^{2(j+k-Ml)N}a_{j,l,k }( \xi^{\prime}_{j,k},t)\overline{a_{j,l,k }( \xi^{\prime}_{j,k},t^{\prime}) }\chi^{2}_{1}(|\xi^{\prime}_{1}|  )  \chi^{2}_{1} \biggl(\frac{\xi^{\prime}_{1}}{ \xi^{\prime}_{3} }\biggl)  \chi^{2}_{0}(2^{k-(M-1)l }|\xi^{\prime}_{2}|),\]
then $\widetilde{b_{j,l,k}}$ satisfies the requirement in Lemma \ref{kernel estimate4}, i.e.,
\[|D_{\xi^{\prime}}^{\alpha}\widetilde{b_{j,l,k}}(\xi^{\prime},t,t^{\prime})| \lesssim 1\]
for each multi-index $\alpha$.

$\bullet$ \textbf{Proof of (\ref{kerneli1}).}  For fixed  $t \in [1/2,4]$, we can apply (\ref{L1E1}) in Lemma \ref{kernel estimate4} to obtain that
\[\|K_{j,l,k}(\cdot,t)\|_{L_{y}^{1}(\mathbb{R}^{3})}   \lesssim 2^{\frac{j}{2}} \cdot 2^{-(j+k-Ml)N}, \]
where the implied constant does not depend on $t$.

$\bullet$  \textbf{Proof of (\ref{kerneli2}).}  We apply (\ref{decay1}) in Lemma \ref{kernel estimate4} for $t\in [1/2,4]$ to obtain that
  \begin{align}
 |K_{j,l,k}(y,t)| \lesssim 2^{-(j+k-Ml)N} 2^{3j+k} (1+2^{j}t)^{-\frac{1}{2}} \lesssim 2^{-(j+k-Ml)N} 2^{\frac{5j}{2}+k}.
 \end{align}

$\bullet$  \textbf{Proof of (\ref{kerneli3}).}  We apply (\ref{decay1})  in Lemma \ref{kernel estimate4}  to get
 \begin{align}
&|\widetilde{K_{j,l,k}}(y,t,t^{\prime})| \lesssim  C_{N} 2^{-2(j+k-Ml)N} 2^{3j+k} (1+2^{j}|t-t^{\prime}|)^{-\frac{1}{2}}.
\end{align}

\textbf{(II)} When $j>Ml$ and $k=1$, we change variables $\xi \rightarrow 2^{j}\xi^{\prime}$, then
 \begin{align}
 K_{j,l,1} (y,t) &=2^{3j} \int_{\mathbb{R}^{3}} e^{i 2^{j} \bigl[y \cdot \xi^{\prime} - t\xi^{\prime}_{3}\bigl( \Psi_{1}(\frac{\xi^{\prime}_{1}}{\xi^{\prime}_{3}})+ 2^{-Ml} \Psi_{2} (\frac{\xi^{\prime}_{1}}{\xi^{\prime}_{3}},\frac{\xi^{\prime}_{2}}{\xi^{\prime}_{1}})\bigl) \bigl] } \nonumber\\  &\quad \times \widetilde{a_{j,l}}(2^{j}\xi^{\prime},t)   \chi_{1}  (|\xi^{\prime}_{1}|)  \chi_{1} \biggl(\frac{\xi^{\prime}_{1}}{ \xi^{\prime}_{3} }\biggl)  \chi_{0}(2^{-(M-1)l }|\xi^{\prime}_{2}|) d\xi^{\prime},
\end{align}
 where
\[\widetilde{a_{j,l}}(2^{j}\xi^{\prime},t)= \sum_{-N_{0} \le k \le N_{1}} \biggl( \widetilde{a_{j,l,k}}(2^{j}\xi^{\prime},t) +e^{-i t2^{j-Ml} \xi^{\prime}_{3}   \Psi_{2} (\frac{\xi^{\prime}_{1}}{\xi^{\prime}_{3}},\frac{\xi^{\prime}_{2}}{\xi^{\prime}_{1}})  } \widetilde{r_{j,l,k}}(2^{j}\xi^{\prime},t) \biggl),\]
and $\widetilde{a_{j,l,k}}$ satisfies (\ref{symbolb2}), $\widetilde{r_{j,l,k}}$ satisfies (\ref{symbolb2+}). Moreover, according to  (\ref{symbolb2+}), for fixed  multi-index $\alpha$,
\begin{align}
&\biggl|D^{\alpha}_{\xi^{\prime}}\biggl( e^{-i t2^{j-Ml} \xi^{\prime}_{3}   \Psi_{2} (\frac{\xi^{\prime}_{1}}{\xi^{\prime}_{3}},\frac{\xi^{\prime}_{2}}{\xi^{\prime}_{1}})  } \widetilde{r_{j,l,k}}(2^{j}\xi^{\prime},t)  \biggl) \biggl| \lesssim 2^{-(j-Ml)N}, \quad \forall N \in \mathbb{N}^{+}. \nonumber
\end{align}
Hence, we denote
\[b_{j,l}(\xi^{\prime},t) = 2^{\frac{j-Ml}{2}} \widetilde{a_{j,l}}(2^{j}\xi^{\prime},t)     \chi_{1}  (|\xi^{\prime}_{1}|)  \chi_{1} \biggl(\frac{\xi^{\prime}_{1}}{ \xi^{\prime}_{3} }\biggl)  \chi_{0}(2^{-(M-1)l }|\xi^{\prime}_{2}|),\]
 $ b_{j,l}(\xi^{\prime},t)$ satisfies the condition in Lemma \ref{kernel estimate1}, i.e.,
\[|D_{\xi^{\prime}}^{\alpha}[b_{j,l}(\xi^{\prime},t)]| \lesssim 1\]
for each multi-index $\alpha$.
Moreover,
 \begin{align}
\widetilde{K_{j,l,1}}(y,t,t^{\prime}) &=2^{3j} \int_{\mathbb{R}^{3}} e^{i 2^{j} \bigl[y \cdot \xi^{\prime} - (t-t^{\prime})\xi^{\prime}_{3}\bigl( \Psi_{1}(\frac{\xi^{\prime}_{1}}{\xi^{\prime}_{3}})+ 2^{-Ml} \Psi_{2} (\frac{\xi^{\prime}_{1}}{\xi^{\prime}_{3}},\frac{\xi^{\prime}_{2}}{\xi^{\prime}_{1}})\bigl) \bigl] } \nonumber\\  &\quad \times \widetilde{a_{j,l}}(2^{j}\xi^{\prime},t)\overline{\widetilde{a_{j,l}}(2^{j}\xi^{\prime},t)} \chi^{2}_{1}  (|\xi^{\prime}_{1}|)  \chi^{2}_{1} \biggl(\frac{\xi^{\prime}_{1}}{ \xi^{\prime}_{3} }\biggl)  \chi^{2}_{0}(2^{-(M-1)l }|\xi^{\prime}_{2}|) d\xi^{\prime}.
\end{align}
Let
\[\widetilde{b_{j,l} }(\xi,t,t^{\prime})= 2^{j-Ml}
\widetilde{a_{j,l}}(2^{j}\xi^{\prime},t)\overline{\widetilde{a_{j,l}}(2^{j}\xi^{\prime},t)} \chi^{2}_{1}  (|\xi^{\prime}_{1}|)  \chi^{2}_{1} \biggl(\frac{\xi^{\prime}_{1}}{ \xi^{\prime}_{3} }\biggl)  \chi^{2}_{0}(2^{-(M-1)l }|\xi^{\prime}_{2}|) .\]
Then
\[| D_{\xi^{\prime}}^{\alpha}\widetilde{b_{j,l}}(\xi^{\prime},t)| \lesssim 1\]
for each multi-index $\alpha$.

$\bullet$  \textbf{Proof of (\ref{kernelii1}).} For each $t \in [1/2,4]$, we apply  (\ref{L1E2}) in Lemma \ref{kernel estimate1} to obtain that
  \[\|K_{j,l,1}(\cdot,t)\|_{L_{y}^{1}(\mathbb{R}^{3})} \lesssim 2^{-\frac{j-Ml}{2}} 2^{j-\frac{Ml}{2}}  \lesssim 2^{\frac{j}{2}}, \]
where the implied constant does not depend on $t$.

$\bullet$  \textbf{Proof of (\ref{kernelii2}).} For each $t \in [1/2,4]$, we apply (\ref{decay2}) in Lemma \ref{kernel estimate1} to get
\[|K_{j,l,1}(y,t)| \lesssim 2^{-\frac{j-Ml}{2}} \cdot 2^{3j}(1+2^{j}t)^{-\frac{1}{2}}(1+2^{j-Ml}t)^{-\frac{1}{2}} \lesssim 2^{\frac{3j}{2} +Ml}.\]

$\bullet$  \textbf{Proof of (\ref{kernelii3}).} By (\ref{decay2}) in Lemma \ref{kernel estimate1}, we get
\begin{align}
|\widetilde{K_{j,l,1}}(y,t,t^{\prime})| &\lesssim 2^{-(j-Ml)}2 ^{3j}(1+2^{j}|t-t^{\prime}|)^{-\frac{1}{2}}(1+2^{j-Ml}|t-t^{\prime}|)^{-\frac{1}{2}} \nonumber\\
&\lesssim 2^{2j+Ml}(1+2^{j}|t-t^{\prime}|)^{-\frac{1}{2}}(1+2^{j-Ml}|t-t^{\prime}|)^{-\frac{1}{2}}. \nonumber
\end{align}

Now we  have completed the proof of Theorem \ref{kernel estimate3}.

\section{Proofs of Theorem \ref{mainth1+} and Theorem \ref{mainth2+}}\label{proofofmainth}

We only show the details for proofs of Theorem \ref{mainth1+} and Theorem \ref{mainth2+} when  $\lambda < \infty$, the case $\lambda = \infty$ can be proved similarly.

We first perform a reduction. By a dyadic decomposition and scaling argument,
\begin{equation}\label{dyadicdec1}
\|\mathcal{M}_{loc,\lambda}\|_{L^{p}  \rightarrow L^{q}  } \le \sum_{k \ge 1} 2^{\frac{(m+2)k}{p}-\frac{(m+2)k}{q}-2k}\|\mathcal{M}_{loc,\lambda,k}\|_{L^{p}  \rightarrow L^{q} },
\end{equation}
where the local maximal operator $\mathcal{M}_{loc,\lambda,k}$ is defined by
\[\mathcal{M}_{loc,\lambda,k}f(y)=  \sup_{t\in [1,2]}\biggl|\widetilde{A^{\lambda}_{t,k}}f(y)\biggl| \]
and the averaging operator
\begin{align}\label{h0}
 \widetilde{A^{\lambda}_{t,k}}f(y)
&=  \int_{\mathbb{R}^{2}}f\bigl(y-t(x_{1}, x_{2}, \Phi(x_{1},x_{2})+c2^{mk} )\bigl) \nonumber\\
&\quad \times \eta_{0}(\frac{x_{2}-\lambda x_{1}}{\epsilon x_{1}})\widetilde{\eta}(x_{1},x_{2}) \rho_{0}(2^{-k}x_{1},2^{-k}x_{2}) dx_{1}dx_{2}.
\end{align}
Here, $\eta_{0}$, $\tilde{\eta}$ and $\rho_{0}$ are smooth cut-off functions such that $\mathrm{supp}$ $\eta_{0} \subset B(0,1)$, $\mathrm{supp}$ $\rho_{0}\subset B(0,1)$, and $\mathrm{supp}$ $\tilde{\eta} \subset \{x \in \mathbb{R}^{2}:  |x| \sim 1\}$.

\textbf{Proof of Theorem \ref{mainth1+}.}
  By (\ref{dyadicdec1}), Theorem \ref{mainth1+} follows if we can prove the following estimates. When $c=0$,
\begin{equation}\label{lemma1eq1}
 \|\mathcal{M}_{loc,\lambda, k}f \|_{L^{p}  \rightarrow L^{q} } \lesssim 1
\end{equation}
holds for $(\frac{1}{p},\frac{1}{q}) \in \Delta_{0}$ and  $\frac{n_{\lambda}+1}{p}-\frac{n_{\lambda}+1}{q}-1<0$.
When $c \neq 0$,
\begin{equation}\label{lemma1eq2}
  \|\mathcal{M}_{loc,\lambda, k}f \|_{L^{p}  \rightarrow L^{q} } \lesssim 2^{\frac{mk}{q}}
\end{equation}
provided that $(\frac{1}{p},\frac{1}{q}) \in \Delta_{0}$ and  $\frac{n_{\lambda}+1}{p}-\frac{1}{q}-1<0$.
 Next we show the proof of (\ref{lemma1eq1}) and (\ref{lemma1eq2}).

 After some linear transformations (which do not change the $L^{p}  \rightarrow L^{q} $ norm of the maximal operator up to constants), we are reduced to estimating the maximal operator defined by
\[\sup_{t\in [1,2]}\biggl|\int_{\mathbb{R}^{2}}f\biggl(y-t\bigl(x_{1}, x_{2},  \Phi(x_{1},x_{2}+\lambda x_{1})+c2^{mk} \bigl)\biggl) \eta_{0}(\frac{x_{2}}{\epsilon x_{1}})\widetilde{\eta}(x_{1},x_{2}+\lambda x_{1})\rho_{k}(x_{1}, x_{2})dx_{1}dx_{2}\biggl|,\]
where we denote $\rho_{k}(x_{1},x_{2})= \rho_{0}(2^{-k}x_{1}, 2^{-k}(x_{2}+\lambda x_{1}))$.
Recall that
\[\Phi(x_1,x_2)=(x_{2}-\lambda x_{1})^{n_{\lambda}}P(x_{1},x_{2}),\]
where $P$ does not vanish along $L_{\lambda}\backslash \{(0,0)\}$. By Taylor expansion of $P(x_{1},x_{2}+\lambda x_{1})$ at $x_{2} =0$,
\begin{align}
\Phi(x_{1},x_{2}+\lambda x_{1})&= c_{\Phi}x^{m-n_{\lambda}}_{1}x^{n_{\lambda}}_{2} + x^{n_{\lambda}+1}_{2} U(x_{1},x_{2}),
\end{align}
where we have put all the remainder terms into $U(x_{1},x_{2})$.

By an isometric transformation, we obtain
\begin{align}
 \|\mathcal{M}_{loc,\lambda, k}f \|_{L^{p}  \rightarrow L^{q} }\le   \sum_{l \ge l_{0}}  2^{\frac{(n_{\lambda} +1)l}{p}-\frac{(n_{\lambda} +1)l}{q}-l}  \|\mathcal{M}_{loc,\lambda, k,l}f \|_{L^{p}  \rightarrow L^{q} }, \nonumber
\end{align}
where $l_{0} \ge \log\frac{1}{\epsilon}$, $\mathcal{M}_{loc,\lambda,k,l}$ denotes the local maximal operator along the hypersurface
\begin{equation}\label{hypersurface1}
S^{\lambda}_{k,l}= \biggl\{ \biggl(x_{1},x_{2}, c_{\Phi}x^{m-n_{\lambda}}_{1}x^{n_{\lambda}}_{2} + 2^{-l}x^{n_{\lambda}+1}_{2} U(x_{1},2^{-l}x_{2}) +2^{mk+ n_{\lambda}l}c \biggl): |x_{1}| \sim |x_{2}| \sim 1 \biggl\}.
\end{equation}
Notice that $0<n_{\lambda}<m$,  the surface $S^{\lambda}_{k,l}$ has non-vanishing Gaussian curvature, then
\begin{align}\label{nonvanishcurvature}
 \|\mathcal{M}_{loc,\lambda,k,l} \|_{L^{p} \rightarrow L^{q} } \lesssim c2^{\frac{mk }{q}} 2^{\frac{n_{\lambda}l}{q}} +1
\end{align}
provided that $(\frac{1}{p}, \frac{1}{q}) \in \Delta_{0}$. Inequality (\ref{nonvanishcurvature}) can be proved by standard arguments.  For completeness, we give a simple explanation for the proof of (\ref{nonvanishcurvature}) in the appendix, one can see Subsection \ref{exmplain} below.

 Now let's proceed to the proof of (\ref{lemma1eq1}) and (\ref{lemma1eq2}).   Inequality (\ref{nonvanishcurvature}) yields
\begin{equation}\label{lemma1eq3}
 \|\mathcal{M}_{loc,\lambda,k} \|_{L^{p}  \rightarrow L^{q} } \le   \sum_{l \ge l_0}  2^{\frac{(n_{\lambda} +1)l}{p}-\frac{(n_{\lambda} +1)l}{q}-l} (c2^{\frac{mk }{q}} 2^{\frac{n_{\lambda}l}{q}} +1).
\end{equation}
It follows that when $c=0$, the right side of (\ref{lemma1eq3}) is convergent provided that $\frac{n_{\lambda}+1}{p}-\frac{n_{\lambda}+1}{q}-1<0$.
Then inequality (\ref{lemma1eq1}) holds true for $(\frac{1}{p},\frac{1}{q}) \in \Delta_{0}$ and  $\frac{n_{\lambda}+1}{p}-\frac{n_{\lambda}+1}{q}-1<0$.
When $c \neq 0$, the right side of (\ref{lemma1eq3}) can be dominated by $c2^{\frac{mk}{q}}$ if  $\frac{n_{\lambda}+1}{p}-\frac{1}{q}-1<0$. Hence, there holds inequality (\ref{lemma1eq2})
provided that $(\frac{1}{p},\frac{1}{q}) \in \Delta_{0}$ and  $\frac{n_{\lambda}+1}{p}-\frac{1}{q}-1<0$. Then we complete the proof of Theorem \ref{mainth1+}.

For completeness, we give some details for the proof of Remark \ref{remarkGammaj}.

\textbf{Proof of Remark \ref{remarkGammaj}.} Through a dyadic decomposition and a scaling argument, we obtain
\begin{equation}\label{dyadicdec1+}
\|\mathcal{M}_{loc,\Gamma_{j}}\|_{L^{p}  \rightarrow L^{q}  } \le \sum_{k \ge 1} 2^{\frac{(m+2)k}{p}-\frac{(m+2)k}{q}-2k}\|\mathcal{M}_{loc,\Gamma_{j},k}\|_{L^{p}  \rightarrow L^{q}  },
\end{equation}
where the local maximal operator $\mathcal{M}_{loc,\Gamma_{j},k}$ is defined by
\begin{align}
\mathcal{M}_{loc,\Gamma_{j},k}f(y)
&=  \sup_{t \in [1,2]} \biggl|\int_{\mathbb{R}^{2}}f\biggl(y -t\bigl(x_{1}, x_{2},  \Phi(x_{1},x_{2})+c2^{mk}\bigl)\biggl) \widetilde{\eta_{j}}(x_{1},x_{2}) dx_{1}dx_{2}\biggl|, \nonumber
\end{align}
$\widetilde{\eta_{j}}$ is a smooth cut-off functions such that  $\mathrm{supp}$ $\widetilde{\eta_{j}} \subset \{(x_{1},x_{2}) \in \Gamma_{j}:  |(x_{1},x_{2})| \sim 1\}$. Notice that for each $(x_{1},x_{2}) \in$ $\mathrm{supp}$ $\widetilde{\eta_{j}}$, $H\Phi(x_{1},x_{2}) \neq 0$. By the method we have applied to prove (\ref{nonvanishcurvature}), there holds
\begin{equation}\label{dyadicdec1++}
 \|\mathcal{M}_{loc,\Gamma_{j},k}\|_{L^{p}  \rightarrow L^{q} } \lesssim (c2^{\frac{mk}{q}}+1)
\end{equation}
for each $(\frac{1}{p},\frac{1}{q}) \in \Delta_{0}$. Then Remark \ref{remarkGammaj} follows from (\ref{dyadicdec1+}) and (\ref{dyadicdec1++}).

Now we are left with the proof of Theorem \ref{mainth2+}.

\textbf{Proof of Theorem \ref{mainth2+}.}
Let
\begin{equation}
\overline{S}:=\{(x_{1},x_{2},\Phi(x_{1},x_{2})+c2^{mk}): |x_{2}-\lambda x_{1}|< \epsilon |x_{1}|, 1/2 \le |(x_{1},x_{2})| \le 2 \}. \nonumber
\end{equation}
By (\ref{dyadicdec1}), Theorem \ref{mainth2+} follows if we can prove that the local maximal operator $\mathcal{M}_{loc,\lambda,k}$ related to $\overline{S}$ satifies
\begin{equation}\label{lemma2eq1}
 \|\mathcal{M}_{loc,\lambda,k}    \|_{L^{p}  \rightarrow L^{q} } \lesssim c2^{\frac{mk}{q}}+1
\end{equation}
if $(\frac{1}{p}, \frac{1}{q}) \in \Delta_{M^{\lambda}}$.

In order to apply Theorem \ref{mainth3} to obtain (\ref{lemma2eq1}), we first show how the hypersurface $\overline{S}$ is related to the hypersurface $\widetilde{S}$ defined by (\ref{tildeS}). We provide a detailed proof for the case when $x_1>0$, while the proof for $x_1<0$ follows through similar argument.
Changing variables $x_{1} \rightarrow \tau \cos \theta$, $x_{2} \rightarrow \tau  (\sin \theta + \lambda  \cos \theta )$, we have
\[\overline{S}= \bigl\{(\tau \cos \theta, \tau (\sin \theta + \lambda \cos \theta ),\tau^{m} \Phi(\cos \theta, \sin \theta + \lambda  \cos \theta )+c2^{mk}): \tau \in [1/2,2], \theta \in (-\theta_{0},\theta_{0})\bigl\},\]
 where $\theta_{0}$ is a positive constant with $\theta_{0} \ll 1$. Recall that $\Phi(x_{1},\lambda x_{1}) \neq 0$ for each $x_{1} \neq 0$,  we have $\Phi(\cos \theta, \sin \theta + \lambda  \cos \theta ) \neq 0$ for each  $\theta \in (-\theta_{0},\theta_{0})$ provided that  $\theta_{0}$ is sufficiently small.
  Without loss of generality, we may assume that $\Phi(\cos \theta, \sin \theta + \lambda  \cos \theta )>0$ for each $\theta \in (-\theta_{0},\theta_{0})$.
  Then we change variable
  \[\tau \rightarrow r/\Phi^{1/m}(\cos \theta, \sin \theta + \lambda  \cos \theta ) ,\]
we have
 \begin{equation}
 \overline{S}=\bigl\{\bigl(r z_{1}(\theta), r (z_{2}(\theta) + \lambda z_{1}(\theta) ),r^{m}+c2^{mk}\bigl): r \in [c_{1}(\theta),c_{2}(\theta)], \theta \in (-\theta_{0},\theta_{0})\bigl\}. \nonumber
 \end{equation}
 Here,
 \[z_{1}(\theta) =  \frac{\cos \theta} {\Phi^{1/m}(\cos \theta, \sin \theta + \lambda  \cos \theta )}, \quad  z_{2}(\theta)  =  \frac{\sin \theta} {\Phi^{1/m}(\cos \theta, \sin \theta + \lambda  \cos \theta )},\]
 and
 \[c_{1}(\theta)= \frac{1}{2} \cdot \Phi^{1/m}(\cos \theta, \sin \theta + \lambda  \cos \theta ), \quad  c_{2}(\theta)=2 \cdot  \Phi^{1/m}(\cos \theta, \sin \theta + \lambda  \cos \theta ).\]
 It is clear that  $c_{i}(\theta) \sim 1$ for $i=1,2$, where the implied constant does not depend on $\theta$.

 Notice that each  $(z_{1},z_{2})=(z_{1}(\theta),z_{2}(\theta))$ satisfies the equation
\begin{equation}\label{levelset}
\Phi(z_{1}, z_{2}  + \lambda z_{1}) = 1.
\end{equation}
By the arguments at  the beginning of Subsection \ref{CasB}, we can obtain a unique  function $\Psi(z_{2})$ defined  for $z_{2}$ containing in a small neighborhood of the origin such that
\[\Phi(\Psi(z_{2}), z_{2}  + \lambda \Psi(z_{2})) = 1.\]
Then
 \[\overline{S}= \bigl\{\bigl(r  \Psi(z_{2}(\theta)), r (z_{2}(\theta) + \lambda \Psi(z_{2}(\theta)) ),r^{m} +c 2^{mk}\bigl): r \in [c_{1}(\theta),c_{2}(\theta)], \theta \in (-\theta_{0},\theta_{0})\bigl\}.\]
 By $z_{2}^{\prime}(\theta) \neq 0$ for each $\theta \in (-\theta_{0},\theta_{0})$, we may change the variable $\theta \rightarrow z^{-1}_{2}(s)$,
where $z^{-1}_{2}$ denotes the inverse function of $z_{2}$ defined in $(-\widetilde{\epsilon_{0}},\widetilde{\epsilon_{0}})$, $0< \widetilde{\epsilon_{0}} \ll 1$.
Finally, we obtain
  \begin{equation}\label{hypersurface2}
 \overline{S}= \bigl\{\bigl(r  \Psi(s), r (s  + \lambda \Psi(s) ),r^{m}+c2^{mk}\bigl): r \in \bigl[ c_{1} (z^{-1}_{2}(s)), c_{2} (z^{-1}_{2}(s))\bigl], s \in (-\widetilde{\epsilon_{0}},\widetilde{\epsilon_{0}})\bigl\}. \nonumber
 \end{equation}

 Now let's turn to estimating the operator $\mathcal{M}_{loc,\lambda,k}$. Based on the above argument, we derive the following estimate
 \begin{align}
  \mathcal{M}_{loc,\lambda,k}f(y) &\lesssim \sup_{t\in [1,2]} \int^{\widetilde{\epsilon_{0}}}_{-\widetilde{\epsilon_{0}}} \int^{c_{2} (z^{-1}_{2}(s))}_{c_{1} (z^{-1}_{2}(s))}
  |f|(y-t(r  \Psi(s), r (s  + \lambda \Psi(s) ),r^{m}+c2^{mk}))   drds \nonumber\\
   &\lesssim \sup_{t\in [1,2]}  \int^{\widetilde{\epsilon_{0}}}_{-\widetilde{\epsilon_{0}}} \int^{\widetilde{c_{2}}}_{\widetilde{c_{1}}} |f|(y-t(r  \Psi(s), r (s  + \lambda \Psi(s) ),r^{m}+c2^{mk})) drds , \nonumber
  \end{align}
  where  $\widetilde{c_{1}}, \widetilde{c_{2}}$ are positive constants.
 By a linear transformation in the $(y_{1},y_{2})$-plane, it is sufficient to consider the  $L^{p} \rightarrow L^{q}$ estimate of  the local maximal operator related  to
\begin{equation}\label{hypersurface2}
\tilde{S}_{k}= \{(r\Psi(s),rs,r^{m}+c2^{mk}): r \in [\widetilde{c_{1}},\widetilde{c_{2}}], s \in (-\widetilde{\epsilon_{0}},\widetilde{\epsilon_{0}})\}.
\end{equation}
We further perform a similar linear transformation to the one above and reduce  to the situation when $\Psi^{\prime}(0)=0$.
Notice that $\Psi$ satisfies Proposition \ref{level set type}. By Taylor expansion of $\Psi(s)$ at $s=0$, we have
\[\Psi(s)=\Psi(0)+s^{M^{\lambda}}g(s), \]
where $\Psi(0) \neq 0$, $3 \le M^{\lambda} \le m$,  $g$ is a smooth function containing  all the remainder terms, and  $g(0) \neq 0$. Then (\ref{lemma2eq1}) follows from Theorem \ref{mainth3}.

\section{Proofs of Theorem \ref{apth1} and Theorem \ref{apth2}}\label{nececondition}

\subsection{Proof of Theorem \ref{apth1}}
 We  prove only  the case when $\lambda \neq \infty$. The case $\lambda = \infty$ follows by similar arguments. We recall that when $\lambda \neq \infty$ and $\Gamma(\lambda)$ is of Type A, $\Phi$ can be written as
\[\Phi(x_1,x_2)=(x_{2}-\lambda x_{1})^{n_{\lambda}}P(x_{1},x_{2}),\]
where $P$ does not vanish along $L_{\lambda}\backslash \{(0,0)\}$. By Taylor expansion of $P(x_{1},x_{2})$ at each point along $L_{\lambda} $, we have
\begin{align}
\Phi(x_{1},x_{2})&= c_{\Phi}(x_{2}-\lambda x_{1})^{n_{\lambda}} x^{m-n_{\lambda}}_{1}+ (x_{2}-\lambda x_{1}) ^{n_{\lambda}+1} \widetilde{P}(x_{1},x_{2}),
\end{align}
where  $\widetilde{P}(x_{1},x_{2})$ contains all higher order remainder terms.

\textbf{(1)} We need to show that $\mathcal{M}_{loc,\lambda}$ cannot be $L^{p}\rightarrow L^{q}$ bounded unless $\frac{1}{p} \le \frac{2}{3}$, $\frac{1}{q} \le \frac{1}{p} \le \frac{3}{q}$ and $\frac{1}{q} \ge \frac{2}{p}-1$.

$\bullet$ The necessity of  $\frac{1}{q} \le \frac{1}{p}$.  By selecting the function $f$ as the characteristic function on balls $B(0,\delta^{-1})$ and the set $T$ as the $\delta^{-1}$  neighborhood of $S_{\lambda}$, where $\delta >0$ can be sufficiently small, if $\mathcal{M}_{loc,\lambda}$ is $L^{p} \rightarrow L^{q} $ bounded, then we can obtain
$|T|^{\frac{1}{q}} \lesssim |B(0,\delta^{-1})|^{\frac{1}{p}}$. This implies the necessity of $\frac{1}{q} \le \frac{1}{p}$.

$\bullet$ The necessity of $\frac{1}{p} \le \frac{3}{q}$. We choose the function $f$ as the characteristic function on the set $T$, where the set $T$ is defined as the $\delta$ neighborhood of $S_{\lambda}$. It is important to note that for any $y$ belonging to the ball $B(0,\delta)$ and $(x_{1},x_{2}) \in \Gamma(\lambda)$,  $y+\bigl(x_{1},x_{2},\Phi(x_{1},x_{2})+c\bigl)$ is an element of $T$. Hence the $L^{p}  \rightarrow L^{q} $ boundedness of $\mathcal{M}_{loc,\lambda}$ implies the validity of $|B(0,\delta)|^{\frac{1}{q}} \lesssim |T|^{\frac{1}{p}}$, which means that the condition $\frac{1}{p} \le \frac{3}{q}$ is necessary.

$\bullet$ The necessity of  $\frac{1}{p} \le \frac{2}{3}$. We choose $f = \chi_{B(0,\delta)}$ which is the characteristic function on the ball $B(0,\delta)$, and the set
\[T=\{t(x_{1},x_{2},\Phi(x_{1},x_{2})+c): t \in [1,2],  \epsilon_{1}/2 <x_{1}< \epsilon_{1}, \epsilon_{2}/2 + \lambda x_{1} <x_{2} <\epsilon_{2} + \lambda x_{1} \},\]
where $0< \delta   \ll \epsilon_{2} \ll \epsilon \ll \epsilon_{1} \ll 1$.
We notice that
\begin{align}\label{transversality}
&\Phi(x_{1},x_{2})-x_{1}\Phi_{1}(x_{1},x_{2})-x_{2}\Phi_{2}(x_{1},x_{2})  \nonumber\\
&=(x_{2}-\lambda x_{1})^{n_{\lambda}} \biggl[(1-m)c_{\Phi}x^{m-n_{\lambda}}_{1}  + (x_{2}-\lambda x_{1}) \overline{P}(x_{1},x_{2}) \biggl],
\end{align}
where   $\overline{P}(x_{1},x_{2})=-n_{\lambda}\widetilde{P}(x_{1},x_{2})-x_{1}\widetilde{P}_{1}(x_{1},x_{2})-x_{2}\widetilde{P}_{2}(x_{1},x_{2})$.
By (\ref{transversality}) and $m \ge 2$, we have
\begin{align}
|T|  = \biggl| \int_{1}^{2} \int_{\epsilon_{1}}^{\epsilon_{1}/2} \int_{ \epsilon_{2}/2 + \lambda x_{1}}^{ \epsilon_{2} + \lambda x_{1}} t^{2} \bigl[\Phi(x_{1},x_{2})-x_{1}\Phi_{1}(x_{1},x_{2})-x_{2}\Phi_{2}(x_{1},x_{2})\bigl]dx_{2}dx_{1}dt \biggl|  \sim \epsilon_{2}^{n_{\lambda}+1}, \nonumber
\end{align}
 and the implied constant does not depend on $\delta$.

   For each $y \in T$, there exists a $t(y) \in [1,2]$, $z_{1} \in (\epsilon_{1}/2 , \epsilon_{1})
   $, and $z_{2} \in (\epsilon_{2}/2  + \lambda z_{1}, \epsilon_{2} + \lambda z_{1} )$ (where $z_1$ and $z_2$ depend on $y$) such that $y =t(y)(z_{1},z_{2},\Phi(z_{1},z_{2})+c) \in T$, there holds
 \begin{align*}
 \mathcal{M}_{loc,\lambda}f(y) &\ge \int_{\Gamma(\lambda)}\chi_{B(0,\delta)}(y-t(y)(x_{1},x_{2},\Phi(x_{1},x_{2})+c))dx_{1}dx_{2}  \nonumber\\
 &\ge  \int_{\{(x_{1},x_{2}) \in \Gamma(\lambda):|x_{i}-z_{i}| < \delta, i=1,2\}}\chi_{B(0,\delta)}(y-t(y)(x_{1},x_{2},\Phi(x_{1},x_{2})+c))dx_{1}dx_{2} \nonumber\\
  &\ge \delta^{2}.
 \end{align*}
  Then it follows from the $L^{p}  \rightarrow L^{q} $ boundedness of $\mathcal{M}_{loc,\lambda}$ that
 $|T|^{\frac{1}{q}}\delta^{2} \lesssim |B(0,\delta)|^{\frac{1}{p}}$, which requires $\frac{1}{p} \le \frac{2}{3}$ since $\delta$ can be sufficiently small.

$\bullet$ The necessity of  $\frac{1}{q}  \ge \frac{2}{p}-1$. Let
 \[x_{1,0}=\epsilon_{1},\quad \quad  x_{2,0} =\epsilon_{2} + \lambda \epsilon_{1}, \quad \quad 0<\epsilon_{2} \ll   \epsilon_{1},\]
 and define the vectors
 \[e_{1}=\bigl(1,0,\Phi_{1}(x_{1,0},x_{2,0})\bigl), \quad e_{2}=\bigl(0,1,\Phi_{2}(x_{1,0},x_{2,0})\bigl), \quad
 e_{3}=\bigl(-\Phi_{1}(x_{1,0},x_{2,0}),-\Phi_{2}(x_{1,0},x_{2,0}), 1\bigl).\]
 We choose $f$ as the characteristic function of the set
 \[D=\{z \in \mathbb{R}^{3}: |z \cdot e_{1}| \le 10 \delta, \quad |z \cdot e_{2}| \le 10 \delta, \quad |z \cdot e_{3}|\le 10 \delta^{2}\} \]
  with $0< \delta \ll \epsilon_{2}$. Denote
 \[a=(x_{1,0}, x_{2,0}, \Phi(x_{1,0},x_{2,0})+c)\]
 and for $t \in [1,2]$,
 \begin{align}
   R_{t}   &= \{y \in \mathbb{R}^{3}: |\bigl(y-ta\bigl) \cdot e_{1}| \le \delta, \quad   |\bigl(y-ta\bigl) \cdot e_{2}| \le \delta, \quad |\bigl(y-ta\bigl) \cdot e_{3}| \le \delta^{2}\}. \nonumber
 \end{align}
  Note that for each $(x_{1},x_{2})$ in the set $\{(x_{1},x_{2}): |x_{i}-x_{i,0}| \le \delta/2, i=1,2\} \subset \Gamma(\lambda)$ and $t \in [1,2]$, we have
 $t(x_{1},x_{2},\Phi(x_{1},x_{2})+c) \in R_{t}$.
  Consequently,  for each  $y \in R_{t}$,
 \[y-t(x_{1},x_{2},\Phi(x_{1},x_{2})+c) \in D.\]
 The $L^{p}  \rightarrow L^{q} $ boundedness of $\mathcal{M}_{loc,\lambda}$  implies
 \begin{equation}\label{neceA13}
 |\cup_{t \in [1,2]} R_{t}|^{\frac{1}{q}} \delta^{2} \lesssim |D|^{\frac{1}{p}}.
  \end{equation}
By (\ref{transversality}), for each $t,t^{\prime} \in [1,2]$ satisfying $|t-t^{\prime}| \sim \delta^{2}$, we obtain
  \[|(t-t^{\prime})a \cdot e_{3}| =|t-t^{\prime}|\biggl| \Phi(x_{1,0},x_{2,0})-x_{1,0}\Phi_{1}(x_{1,0},x_{2,0})-x_{2,0}\Phi_{2}(x_{1,0},x_{2,0})\biggl| \sim \epsilon^{n_{\lambda}}_{2} \delta^{2},\]
  where the implied constant does not depend on $\delta$.   Then we can obtain that   $|\cup_{t \in [1,2]} R_{t}| \sim  \delta^{2}$.
 Hence, (\ref{neceA13}) implies that $\delta^{2+\frac{2}{q}}\le \delta^{\frac{4}{p}}$. Taking $\delta \rightarrow 0$, this  gives  $1+\frac{1}{q} \ge \frac{2}{p}$.


\textbf{(2)}
Recall that $h_{\lambda}=\max\{m/2,n_{\lambda}\}$. We consider the cases $h_{\lambda}=m/2$ and $h_{\lambda}= n_{\lambda}$.

$\bullet$  We first consider the case $h_{\lambda}=m/2$. We choose $f$ as the characteristic function of the set $D= (-10\delta,10 \delta) \times (-10\delta,10\delta) \times (-10\delta^{m}, 10 \delta^{m})$. Additionally, we define the set $ \tilde{U}= (0,\delta) \times (0, \delta)$. We set  $\delta \ll \epsilon \ll 1$.

When $c =0$, for each $(x_{1},x_{2}) \in \tilde{U}\cap \Gamma(\lambda)$ and each $y \in R:=(0, \delta) \times (0,\delta) \times (0,\delta^{m})$, we have $y- \bigl(x_{1},x_{2},\Phi(x_{1},x_{2}) \bigl) \in D$. Then it follows from   the $L^{p}  \rightarrow L^{q} $ boundedness of $\mathcal{M}_{loc,\lambda}$  that $|R|^{\frac{1}{q}}|\tilde{U} \cap \Gamma(\lambda)|  \lesssim |D|^{\frac{1}{p}}$. As a result, we have obtained the necessity of $\frac{m/2+1}{p}-\frac{m/2+1}{q}-1 \le 0$.

 When $c \neq 0$, we choose the set $R_{t}=(0,\delta) \times (0,\delta) \times (tc,\delta^{m} + tc)$, $t \in [1,2]$. Then for each $y \in R_{t}$, $t \in [1,2]$ and $(x_{1},x_{2}) \in \tilde{U} \cap \Gamma(\lambda)$, we have $y-t\bigl(x_{1},x_{2},\Phi(x_{1},x_{2})+c\bigl) \in D$. Then the $L^{p} \rightarrow L^{q} $ boundedness of $\mathcal{M}_{loc,\lambda}$ implies $|\cup_{t \in [1,2]}R_{t}|^{\frac{1}{q}}|\tilde{U} \cap \Gamma(\lambda)| \lesssim |D|^{\frac{1}{p}}$, which shows the necessity of $\frac{m/2+1}{p} - \frac{1}{q} -1 \le 0$, since $|\cup_{t\in[1,2]}R_{t}| \sim_{\epsilon} \delta^{2}$.

$\bullet$  Next, we consider the case $h_{\lambda} = n_{\lambda}$. In this case, we denote the set $D$ by
\[D:=\{(z_{1},z_{2},z_{3}): |z_{1}|<10 \epsilon_{1},|z_{2}-\lambda z_{1}|< 10\delta,|z_{3}|<10 \delta^{n_{\lambda}} \},\]
and we choose $f$ as the characteristic function of the set $D$. We  define $\tilde{U}: =\{(x_{1},x_{2}): \epsilon_{1}/2<x_{1} <\epsilon_{1}, |x_{2}-\lambda x_{1}|<\delta\}$. We set $0<\delta \ll \epsilon \ll \epsilon_{1} \ll 1$.

When $c=0$, we choose $  R: =\{(y_{1},y_{2},y_{3}): 0<y_{1}<\epsilon_{1}, 0<y_{2}-\lambda y_{1}<\delta, 0<y_{3}<\delta^{n_{\lambda}}\}$. Then for each $(x_{1},x_{2}) \in \tilde{U}\cap \Gamma(\lambda)$ and   each  $y \in R $, we have $y- \bigl(x_{1},x_{2},\Phi(x_{1},x_{2}) \bigl) \in D$. Then   the $L^{p}  \rightarrow L^{q} $ boundedness of $\mathcal{M}_{loc,\lambda}$ yields   that $|R|^{\frac{1}{q}}|\tilde{U} \cap \Gamma(\lambda)|  \lesssim |D|^{\frac{1}{p}}$. Since $|R| \sim |D| \sim \delta^{n_{\lambda}+1}$  and $|\tilde{U} \cap \Gamma(\lambda)| \sim \delta$, where the implied constants do  not depend on $\delta$, we  have obtained the necessity of $\frac{n_{\lambda}+1}{p}-\frac{n_{\lambda}+1}{q}-1 \le 0$ because $\delta$ can be sufficiently small.

 When $c \neq 0$, we denote  $R_{t}: =\{(y_{1},y_{2},y_{3}): 0<y_{1}<\epsilon_{1}, 0<y_{2}-\lambda y_{1}<\delta, tc<y_{3}<\delta^{n_{\lambda}}+tc\}$.  Then for each $(x_{1},x_{2}) \in \tilde{U}\cap \Gamma(\lambda)$ and each $y \in R_{t} $, we have $y- t\bigl(x_{1},x_{2},\Phi(x_{1},x_{2})+c \bigl) \in D$. The $L^{p}  \rightarrow L^{q} $ boundedness of $\mathcal{M}_{loc,\lambda}$ yields   that $|\cup_{t\in [1,2]}R_{t}|^{\frac{1}{q}}|\tilde{U} \cap \Gamma(\lambda)|  \lesssim |D|^{\frac{1}{p}}$. Since  $|\cup_{t\in [1,2]}R_{t}|  \sim \delta$ (with implied constant independent of $\delta$),  we   obtain the necessity of $\frac{n_{\lambda}+1}{p}-\frac{1}{q}-1 \le 0$.

 Then we have finished the proof of Theorem \ref{apth1}.  $\hfill$

\subsection{Proof of Theorem \ref{apth2}}
By similar arguments as in the proof of Theorem \ref{apth1}, we can get the necessity of conditions $\frac{1}{q} \le \frac{1}{p} \le \frac{3}{q}$, $\frac{m/2+1}{p}-\frac{1}{q}-1 \le 0$ for $c \neq 0$,  and $\frac{m/2+1}{p}-\frac{m/2+1}{q}-1 \le 0$ for $c =0$. We therefore  omit their proofs here.

Then we are left to prove the necessity of $\frac{1}{p} \le \frac{M^{\lambda}+1}{2M^{\lambda}}$ and $1+ \frac{1}{q} \ge \frac{3M^{\lambda}+2}{M^{\lambda}+2} \cdot \frac{1}{p}$ for $c =0$. We only consider the case $\lambda \neq \infty$, because the case $\lambda =\infty$ can be proved by similar arguments.  We recall that $M^{\lambda} = \omega_{\lambda} +2$, where $\omega_{\lambda}$ is the root multiplicity of $x_{2} = \lambda x_{1}$ in $H\Phi(x_{1},x_{2})$, i.e., $H\Phi(x_{1},x_{2})$ can be written as
\[H\Phi(x_{1},x_{2}) = (x_{2}-\lambda x_{1})^{\omega_{\lambda}}Q(x_{1},x_{2}),\]
where the function $Q$ satisfies $Q(x_{1},\lambda x_{1}) \neq 0$  for each $x_{1} \neq 0$. Let $D^{2}\Phi(x_{1},x_{2})$ be the Hessian matrix of $\Phi$. The fact that  $\Phi$ does not vanish along $L_{\lambda} \backslash \{(0,0)\}$ implies that $\nabla\Phi$ is not the zero vector. The homogeneity of $\Phi$ also yields
\begin{equation}\label{Phi1homogeneous}
\Phi_{1}(x_{1},x_{2}) = \frac{1}{m-1}x_{1}\Phi_{11}(x_{1},x_{2})+\frac{1}{m-1}x_{2}\Phi_{12}(x_{1},x_{2})
\end{equation}
and
\begin{equation}\label{Phi2homogeneous}
\Phi_{2}(x_{1},x_{2}) = \frac{1}{m-1}x_{1}\Phi_{21}(x_{1},x_{2})+\frac{1}{m-1}x_{2}\Phi_{22}(x_{1},x_{2}).
\end{equation}
 Clearly,  $D^{2}\Phi(x_{1}, x_{2})$ is not the zero matrix along $L_{\lambda}\backslash \{(0,0)\}$.
  Combining with the fact that $H\Phi$ vanishes along $L_{\lambda}$, we get that  $D^{2}\Phi(x_{1},\lambda x_{1})$ is a rank-one matrix  for  $x_{1} \neq 0$. Hence,  the matrix $D^{2}\Phi(x_{1},\lambda x_{1})$ has exactly one eigenvalue zero. We denote by $\mathbf{v}(x_{1}) =(v_{1}(x_{1}), v_{2}(x_{1})  )$  the unit eigenvector  associated with this zero eigenvalue. It is easy to check that
\begin{align}
\mathbf{v}(x_{1}) \cdot \nabla\Phi(x_{1},\lambda x_{1})& =  \frac{1}{m-1}x_{1} \biggl[v_{1}(x_{1})\Phi_{11}(x_{1},\lambda x_{1})+v_{2}(x_{1})\Phi_{21}(x_{1},\lambda x_{1})\biggl] \nonumber\\
& \quad \quad +\frac{1}{m-1}\lambda x_{1} \biggl[v_{1}(x_{1})\Phi_{12}(x_{1},\lambda x_{1})+v_{2}(x_{1})\Phi_{22}(x_{1},\lambda x_{1})\biggl] \nonumber\\
&=0.
\end{align}
Thus, $\mathbf{v}(x_{1})$ is orthogonal to $\nabla\Phi(x_{1},\lambda x_{1})$. According to whether or not $\nabla\Phi(x_{1},\lambda x_{1})$ (or equivalently $\mathbf{v}(x_{1})$) is parallel  to one of the coordinate axes, we consider the following two cases.

\textbf{Case I.}
$\nabla\Phi(x_{1},\lambda x_{1})$ is parallel to one of the coordinate axes for each $x_{1} \neq 0$. This means that
exactly one of $\Phi_{1}(x_{1},\lambda x_{1})$ and $\Phi_{2}(x_{1},\lambda x_{1})$ vanishes for each $x_{1} \neq 0$. Without loss of generality, we may assume that $\Phi_{1}(x_{1},\lambda x_{1}) \neq 0$  and $\Phi_{2}(x_{1},\lambda x_{1})=0$. Then we have $\mathbf{v}(x_{1})=(0,1)$ for each $x_{1} \neq 0$. For $(x_{1},x_{2}) \in \Gamma(\lambda)$, we perform a Taylor expansion on $\Phi(x_{1},x_{2})$ at $(x_{1},\lambda x_{1})$ along the vector $(0,1)$, and obtain that
\begin{align}\label{claimPhi}
\Phi(x_{1},\lambda x_{1} + \tau ) = \Phi(x_{1},\lambda x_{1}) + \frac{\tau^{2}}{2} \Phi_{22}(x_{1},\lambda x_{1}+\theta \tau ),
\end{align}
where the constant  $\theta \in (0,1)$ depends on $x_{1}$ and $\tau$.  We claim that
\begin{equation}\label{claimPhi22}
\Phi_{22}(x_{1},\lambda x_{1}+\theta \tau) = \theta^{\omega_{\lambda}}\tau^{\omega_{\lambda}}\overline{P}(x_{1},\lambda x_{1}+\theta \tau),
\end{equation}
where the function $\overline{P}$ satisfies $\overline{P}(x_{1},\lambda x_{1}) \neq 0$ for each $x_{1} \neq 0$. We note that if the claim holds true, then the hypersurface $S_{\lambda}$ can be parameterized by
\[S_{\lambda}=\bigl\{\bigl(x_{1},\lambda x_{1}+\tau, \Phi(x_{1},\lambda x_{1})+\frac{1}{2}\theta^{\omega_{\lambda}} \tau^{\omega_{\lambda}+2}\overline{P}(x_{1},\lambda x_{1}+\theta \tau)\bigl): x_{1}>0, -\epsilon <\tau< \epsilon \bigl\}. \]
We first use this claim to prove the necessary condition $\frac{1}{p}< \frac{M^{\lambda} +1 }{2M^{\lambda}}$ and $1+ \frac{1}{q} \ge \frac{3M^{\lambda}+2}{M^{\lambda}+2} \cdot \frac{1}{p}$, and subsequently prove the claim itself.

 $\bullet$ Necessity of $\frac{1}{p}< \frac{M^{\lambda} +1 }{2M^{\lambda}}$. Let
 \[\tilde{U} = \{(x_{1},\tau): x_{1} \in (\frac{\epsilon_{1}}{2},\epsilon_{1}), \tau \in (0, \delta)\}, \quad \quad 0 < \delta \ll  \epsilon \ll \epsilon_{1} \ll 1.\]
  Set
\[\mathcal{T} =\{t(x_{1}, \lambda x_{1}+\tau,\Phi(x_{1},\lambda x_{1})): (x_{1},\tau) \in \tilde{U}, t \in [1,2]\}. \]
By the homogeneous property, $\Phi(x_{1},\lambda x_{1} )=x^{m}_{1}\Phi(1,\lambda) \neq 0$, and $m \ge 2$. We have
 \[|\mathcal{T}| =\biggl|\int^{2}_{1} \int^{\epsilon_{1}}_{\frac{\epsilon_{1}}{2}} \int^{\delta}_{0}t^{2}(1-m)x^{m}_{1}\Phi(1,\lambda ) d\tau dx_{1}dt \biggl| \sim \delta,\]
  where the implied constant is independent of  $\delta$.

We define
\[D= [-10 \delta^{M^{\lambda}}, 10 \delta^{M^{\lambda}}] \times  [-10\delta, 10 \delta] \times [-10   \delta^{M^{\lambda}}, 10    \delta^{M^{\lambda}}].\]
Let $f= \chi_{D}$ be the characteristic function of  $D$. Then
\begin{equation}\label{ubound}
\|f\|_{L^{p}} \sim \delta^{\frac{2M^{\lambda}+1}{p}}.
\end{equation}
Notice that for each $y \in \mathcal{T}$, there exists  $s_{y} \in [1,2]$ and $(z_{1},\tilde{\tau}) \in \tilde{U}$ (which depend on $y$) such that $y=s_{y}(z_{1},\lambda z_{1} + \tilde{\tau},\Phi(z_{1},\lambda z _{1}))$. Then there holds
\begin{align}
&\mathcal{M}_{loc,\lambda}f(y) \nonumber\\
&\ge \int_{\tilde{U}} \chi_{D}\biggl(s_{y}\bigl(z_{1}-x_{1},\lambda z_{1}+ \tilde{\tau}-\lambda x_{1}-\tau,\Phi(z_{1},\lambda z_{1}) -\Phi(x_{1},\lambda x_{1} + \tau)\bigl)\biggl)dx_{1}d\tau \nonumber\\
&\ge  \int_{\tilde{U}(z_{1},\tilde{\tau})} \chi_{D}\biggl(s_{y}\bigl(z_{1}-x_{1}, \lambda z_{1}+ \tilde{\tau}-\lambda x_{1}-\tau,\Phi(z_{1},\lambda z_{1}) -\Phi(x_{1},\lambda x_{1} + \tau)\bigl)\biggl)dx_{1}d\tau ,
\end{align}
where $\tilde{U}(z_{1},\tilde{\tau})= \{(x_{1},\tau) \in \tilde{U}: |x_{1}-z_{1}| \le \delta^{M^{\lambda}} /2, |\tau-\tilde{\tau}| \le \delta/2 \}$. It is clear that $|\tilde{U}(z_{1},\tilde{\tau})| \gtrsim \delta^{M^{\lambda}+1}$. Moreover, for each $(x_{1},\tau) \in \tilde{U}(z_{1},\tilde{\tau})$, there holds
\begin{equation}\label{Nececheck1}
|s_{y}(z_{1}-x_{1})| \le 10 \delta^{M^{\lambda}}, \quad |s_{y}(\lambda z_{1}+ \tilde{\tau}-\lambda x_{1}-\tau)| \le 10  \delta.
\end{equation}
By (\ref{claimPhi22}) and (\ref{claimPhi}),
\begin{align}\label{Nececheck2}
\biggl|s_{y} \bigl[ \Phi(z_{1},\lambda z_{1}) -\Phi(x_{1},\lambda x_{1}+\tau ) \bigl]  \biggl|
& \le  \biggl| \Phi(z_{1},\lambda z_{1}) - \Phi(x_{1},\lambda x_{1}) \biggl| + \biggl|\theta^{\omega_{\lambda}}\tau^{\omega_{\lambda}+2} \overline{P}(x_{1},\lambda x_{1}+\theta \tau )\biggl| \nonumber\\
&\le 10   \delta^{M^{\lambda}}.
\end{align}

The estimates (\ref{Nececheck1}) and (\ref{Nececheck2}) yield that
\[s_{y}\bigl(z_{1}-x_{1}, \lambda z_{1} + \tilde{\tau}-\lambda x_{1}-\tau,\Phi(z_{1},\lambda z_{1}) -\Phi(x_{1}, \lambda x_{1} + \tau )\bigl) \in D.\]
Hence for each $y \in \mathcal{T}$, there holds
\[\mathcal{M}_{loc,\lambda}f(y) \ge |\tilde{U}(z_{1},\tilde{\tau})| \gtrsim \delta^{M^{\lambda}+1}. \]
Furthermore,
\begin{equation}\label{lbound}
 \|\mathcal{M}_{loc,\lambda}f  \|_{L^{p}} \ge  \|\mathcal{M}_{loc,\lambda}f\|_{L^{p}(\mathcal{T})} \gtrsim_{\epsilon} \delta^{M^{\lambda}+1+\frac{1}{p}}.
\end{equation}
If $\mathcal{M}_{loc,\lambda}$ is $L^{p}$ bounded, then the inequalities (\ref{ubound}) and (\ref{lbound})  imply that
\[\delta^{M^{\lambda}+1+\frac{1}{p}} \lesssim_{\epsilon} \delta^{\frac{2M^{\lambda}+1}{p}}.\]
Then it follows that $\frac{1}{p} \le \frac{M^{\lambda}+1}{2M^{\lambda}}$ since $\delta$ can be sufficiently small.

$\bullet$ Necessity of $1+ \frac{1}{q} \ge \frac{3M^{\lambda}+2}{M^{\lambda}+2} \cdot \frac{1}{p}$.
We choose $\delta$ and $\epsilon_{1}$ such that $0<\delta \ll \epsilon \ll \epsilon_{1} \ll 1$. Let
\[e_{1}:=\bigl(1,0, m\epsilon_{1}^{m-1}\Phi(1,\lambda)\bigl), \quad \quad e_{2}:=\bigl(0,1,0\bigl), \quad \quad e_{3}:=\bigl(-m\epsilon_{1}^{m-1}\Phi(1,\lambda),0,1\bigl).\]
Let
\[a:=\bigl(\epsilon_{1},\lambda \epsilon_{1},\Phi(\epsilon_{1},\lambda \epsilon_{1})\bigl).\]
We define
\[R_{t}:=\{z\in \mathbb{R}^{3}: |(z-ta)\cdot e_{1}| \le \delta^{M^{\lambda}/2},\quad |(z-ta)\cdot e_{2}| \le \delta, \quad |(z-ta)\cdot e_{3}| \le \delta^{M^{\lambda}} \}.\]
For each $(x_{1},\tau)\in \widetilde{U}: = (\epsilon_{1}-\frac{\delta^{M^{\lambda}/2}}{2},\epsilon_{1}+ \frac{\delta^{M^{\lambda}/2}}{2}) \times (0,\frac{\delta}{2})$, we have
\begin{equation}\label{SsubsetRt}
t\bigl (x_{1}, \lambda x_{1} + \tau, \Phi(x_{1},\lambda x_{1}+\tau)\bigl) \in R_{t}.
\end{equation}
Indeed, applying (\ref{claimPhi22}), we can obtain that
\begin{align}
&\biggl|t\biggl(\bigl(x_{1},\lambda x_{1}+\tau, \Phi(x_{1},\lambda x_{1}+\tau)\bigr)-a\biggr) \cdot e_{1}\biggl| \nonumber\\
&\le 2 \biggl(|x_{1}-\epsilon_{1}|  + |m\epsilon^{m-1}_{1}\Phi(1,\lambda)|\bigl|\Phi(x_{1},\lambda x_{1}) +\frac{\theta^{\omega_{\lambda}}\tau^{M^{\lambda}}}{2}\overline{P}(x_{1},\lambda x_{1}+\tau) - \Phi(\epsilon_{1},\lambda \epsilon_{1}) \bigl|  \biggl) \nonumber\\
&\le \delta^{M^{\lambda}/2},
\end{align}
\begin{align}
\biggl|t\biggl(\bigl(x_{1}, \tau, \Phi(x_{1},\lambda x_{1}+\tau)\bigr)-a\biggr) \cdot e_{2}\biggl|
\le 2   |\lambda(x_{1}-\epsilon_{1}) + \tau| \le \delta,
\end{align}
and
\begin{align}
&\biggl|t\biggl(\bigl(x_{1},\tau, \Phi(x_{1},\lambda x_{1}+\tau)\bigr)-a\biggr) \cdot e_{3}\biggl| \nonumber\\
&\le 2 \biggl|-(x_{1}-\epsilon_{1})m\epsilon^{m-1}_{1}\Phi(1,\lambda )  +  \bigl[\Phi(x_{1},\lambda x_{1}) +\frac{\theta^{\omega_{\lambda}}\tau^{M^{\lambda}}}{2}\overline{P}(x_{1},\lambda x_{1}+\tau) - \Phi(\epsilon_{1},\lambda \epsilon_{1}) \bigl]  \biggl|\nonumber\\
&\le \delta^{M^{\lambda}}.
\end{align}

Let
\[D:=\{z\in \mathbb{R}^{3}: |z \cdot e_{1}| \le 10 \delta^{M^{\lambda}/2},\quad |z \cdot e_{2}| \le 10 \delta, \quad |z \cdot e_{3}| \le 10 \delta^{M^{\lambda}} \}.\]
Then for each $y \in R_{t}$ and $(x_{1},\tau) \in \widetilde{U}$, by (\ref{SsubsetRt}), we have
\[y-t\bigl(x_{1}, \lambda x_{1} + \tau, \Phi(x_{1},\lambda x_{1}+\tau)\bigl) \in D. \]
We choose $f$ as the characteristic function of $D$. The $L^{p} \rightarrow L^{q}$ boundedness of $\mathcal{M}_{loc,\lambda}$ yields
\begin{equation}\label{necepq2}
|\cup_{t\in[1,2]}R_{t}|^{\frac{1}{q}}|\widetilde{U}|\lesssim |D|^{\frac{1}{p}}.
\end{equation}
For each $t,t^{\prime} \in [1,2]$ satisfying $|t-t^{\prime}| \sim \delta^{M^{\lambda}}$, we have
\begin{equation}
|e_{3} \cdot (t-t^{\prime})a| = |t-t^{\prime}||\Phi(\epsilon_{1},\lambda \epsilon_{1})-\epsilon_{1} m \epsilon_{1}^{m-1}\Phi(1,\lambda )|=|(m-1)\epsilon^{m}_{1}\Phi(1,\lambda)|\delta^{M^{\lambda}}.\nonumber
\end{equation}
Since $\Phi$ does not vanish along $L_{\lambda}\backslash \{(0,0)\}$ and $m \ge 2$, we can obtain that
\[|\cup_{t\in[1,2]}R_{t}| \sim \delta^{\frac{M^{\lambda}}{2}+1}.\]
It follows from (\ref{necepq2}) that
\[\delta^{\frac{M^{\lambda}}{2}+1 +(\frac{M^{\lambda}}{2}+1)\frac{1}{q}} \lesssim \delta^{(M^{\lambda}+ \frac{M^{\lambda}}{2}+1)\frac{1}{p}}.\]
Then we obtain $1+\frac{1}{q} \ge \frac{3M^{\lambda}+2}{M^{\lambda}+2}\cdot \frac{1}{p} $ because $\delta$ can be sufficiently small.

Next, we will prove (\ref{claimPhi22}). It suffices to prove  that $\partial^{(k)}_{2}\Phi$ vanishes along  $L_{\lambda} \backslash \{(0,0)\}$ for $2 \le k \le \omega_{\lambda} +1$, and $\partial^{(\omega_{\lambda}+2)}_{2}\Phi$ does not vanish along  $L_{\lambda} \backslash \{(0,0)\}$. Then we can obtain (\ref{claimPhi22}) by   Taylor's formula. Moreover, we have  \begin{equation}\label{Mm}
M^{\lambda}=\omega_{\lambda}+2 \le m
\end{equation}
because the degree of $x_{2}$ in the polynomial  $\Phi$ cannot exceed $m$.

We will prove $\Phi_{22}(x_{1},\lambda  x_{1} ) = 0$ by contradiction. Assuming that $\Phi_{22}(x_{1}, \lambda  x_{1} ) \neq 0$,
since    $\mathrm{H}\Phi (x_{1}, \lambda x_{1})=0$,   by (\ref{Phi1homogeneous}) and (\ref{Phi2homogeneous}), we have
\begin{align}\label{Phi1Phi2}
&\Phi_{1}(x_{1}, \lambda  x_{1} ) \nonumber\\
&=   \frac{\Phi_{12}(x_{1}, \lambda x_{1} )}{\Phi_{22}(x_{1}, \lambda  x_{1} )} \biggl( \frac{1}{m-1} x_{1} \Phi_{12}(x_{1}, \lambda  x_{1}) +\frac{1}{m-1}\lambda x_{1} \Phi_{22}(x_{1}, \lambda x_{1} )\biggl) \nonumber\\
&=   \frac{\Phi_{12}(x_{1}, \lambda  x_{1} )}{\Phi_{22}(x_{1}, \lambda x_{1} )} \Phi_{2}(x_{1}, \lambda  x_{1} ).
\end{align}
  Since $\Phi_{2}(x_{1},\lambda  x_{1}  )= 0$, (\ref{Phi1Phi2}) implies   $\Phi_{1}(x_{1},\lambda  x_{1} ) = 0$. This contradicts  the fact that $\Phi_{1}$ does not vanish along $L_{\lambda} \backslash \{(0,0)\}$.


By  $\Phi_{22}(x_{1},\lambda x_{1}) = 0$ and  $\Phi_{2}(x_{1},\lambda x_{1}) = 0$, we may assume  that
\[\Phi_{2}(x_{1},x_{2}) = (x_{2}-\lambda x_{1})^{\tilde{M}} \tilde{Q}(x_{1},x_{2}), \]
where $\tilde{M} \ge 2$ and $\tilde{Q}$ does not vanish along $L_{\lambda} \backslash \{(0,0)\}$. Then we have
\begin{align}
H\Phi(x_{1},x_{2}) 
= (x_{2}-\lambda x_{1})^{\tilde{M}-1} Q(x_{1},x_{2}), \nonumber
\end{align}
where
\begin{align}\label{HPhiexpansion}
Q(x_{1},x_{2})&= \Phi_{11}(x_{1},x_{2})\biggl[\tilde{M} \tilde{Q}(x_{1},x_{2}) +  (x_{2}-\lambda x_{1}) \tilde{Q}_{2}(x_{1},x_{2}) \biggl]-(x_{2}-\lambda  x_{1})^{\tilde{M}-1} \times \nonumber\\
& \biggl[-\tilde{M}\lambda \tilde{Q}(x_{1},x_{2}) + (x_{2}-\lambda x_{1}) \tilde{Q}_{1}(x_{1},x_{2}) \biggl]^{2}.
 \end{align}
  To prove the claim, it suffices to show that  $Q$ does not vanish along $L_{\lambda} \backslash \{(0,0)\}$,  which implies $\omega_{\lambda}= \tilde{M}-1$. Then $M^{\lambda}= \omega_{\lambda}+2 $ is the smallest positive integer $k$ such that $ \partial^{(k)}_{2}\Phi$ does not vanish along $L_{\lambda} \backslash \{(0,0)\}$.

 To show that $Q$ does not vanish along $L_{\lambda} \backslash \{(0,0)\}$, by (\ref{HPhiexpansion}),  it suffices to prove that   $\Phi_{11}$ is non-vanishing along $L_{\lambda} \backslash \{(0,0)\}$. Indeed,  by Lemma 3.3 in \cite{IKM1}, since $\Phi_{1}$ does not vanish along  $L_{\lambda} \backslash \{(0,0)\}$, while $H\Phi$ vanishes along $L_{\lambda}\backslash \{(0,0)\}$, we can obtain  that $\Phi_{11}$ does not vanish along $L_{\lambda} \backslash \{(0,0)\}$.  This completes the proof of the claim.

\textbf{Case II.} $\nabla\Phi(x_{1},\lambda x_{1})$ is not parallel to any of the coordinate axes for each $x_{1} \neq 0$. This means that for each $x_{1} \neq 0$, neither $\Phi_{1}(x_{1},\lambda x_{1})$ nor $\Phi_{2}(x_{1},\lambda x_{1})$  vanishes.

We first do some reduction. By some linear transformations which do not change the $L^{p} \rightarrow L^{q}$  norm of $\mathcal{M}_{loc,\lambda}$ up to constants, it suffices to estimate the local maximal operator $\widetilde{\mathcal{M}}_{loc,0}$ over the hypersurface
\[\widetilde{S_{0}}:=\{(x_{1},x_{2},\widetilde{\Phi}(x_{1},x_{2})): |x_{2}|<\epsilon |x_{1}|\}, \quad \quad \epsilon \ll 1, \]
where
\[\widetilde{\Phi}(x_{1},x_{2})=\Phi\biggl(x_{1}-\frac{\Phi_{2}(x_{1},\lambda x_{1})}{\Phi_{1}(x_{1},\lambda x_{1})+\lambda \Phi_{2}(x_{1},\lambda x_{1})}x_{2}, x_{2}+ \lambda\bigl[x_{1}-\frac{\Phi_{2}(x_{1},\lambda x_{1})}{\Phi_{1}(x_{1},\lambda x_{1})+\lambda \Phi_{2}(x_{1},\lambda x_{1})}x_{2} \bigl]  \biggl).\]
We note that $\Phi(x_{1},x_{2})$ does not vanish along $L_{\lambda} \backslash \{(0,0)\}$. The homogeneous property of $\Phi$ implies  $\Phi_{1}(x_{1},\lambda x_{1})+\lambda \Phi_{2}(x_{1},\lambda x_{1}) \neq 0$, and
\[\frac{\Phi_{2}(x_{1},\lambda x_{1})}{\Phi_{1}(x_{1},\lambda x_{1})+\lambda \Phi_{2}(x_{1},\lambda x_{1})}=\frac{\Phi_{2}(1,\lambda  )}{\Phi_{1}(1,\lambda  )+\lambda \Phi_{2}(1,\lambda )}\]
is a non-zero constant.  Clearly,  $\widetilde{\Phi}$ is a polynomial homogeneous of degree $m$,
\[\widetilde{\Phi}(x_{1},0) \neq 0, \quad \quad \widetilde{\Phi}_{2}(x_{1},0) = 0\]
for each $x_{1}\neq 0$. Moreover, by Lemma \ref{level set leema}, we have
\[H\widetilde{\Phi}(x_{1},x_{2})=x^{\omega_{\lambda}}_{2}Q\biggl(x_{1}-\frac{\Phi_{2}(1,\lambda )}{\Phi_{1}(1,\lambda )+\lambda \Phi_{2}(1,\lambda )}x_{2}, x_{2}+ \lambda\bigl[x_{1}-\frac{\Phi_{2}(1,\lambda )}{\Phi_{1}(1,\lambda )+\lambda \Phi_{2}(1,\lambda  )}x_{2} \bigl]  \biggl),\]
which implies that the root multiplicity of $x_{2}=0$ in $H\widetilde{\Phi}$ equals $\omega_{\lambda}$. Hence, by the arguments  from Case I,   $\widetilde{\mathcal{M}}_{loc,0}$ cannot be $L^{p} \rightarrow L^{q}$ bounded unless $\frac{1}{p} \le \frac{M^{\lambda}+1}{2M^{\lambda}}$ and $1+\frac{1}{q} \le \frac{3M^{\lambda}+2}{M^{\lambda}+2} \cdot \frac{1}{p}$.

This completes the proof of Theorem  \ref{apth2}.  $\hfill$

\section{Proofs of Proposition \ref{muphihphi} and Proposition \ref{level set type}}\label{proofofproperties}
We will prove Proposition \ref{muphihphi} and Proposition \ref{level set type} in Subsection \ref{prop1.3}
 and Subsection \ref{prop1.5}, respectively.

\subsection{Proof of Proposition \ref{muphihphi}}\label{prop1.3}
For convenience, we denote
\[E_{\Phi}=\{n_{\lambda}: \Gamma(\lambda) \text{ is of Type A}\} \cup \{ m/2\},\]
and denote $\mu_{\Phi}$ by the maximal element of $E_{\Phi}$.

We have the following observations before the proof of Proposition \ref{muphihphi}.
For fixed $x \in S^{1}$, if $\Phi(x) \neq 0$,  it is clear that  $\textsf{ordd} \Phi (x) =0$.
Otherwise, without loss of generality, we may assume that $x \in S^{1} \cap \{x_{2}=\lambda_{h}x_{1}\}$, where $x_{2}=\lambda_{h}x_{1}$ corresponds to the  factor $(x_{2}-\lambda_{h}x_{1})^{n_{h}}$ in (\ref{Phi}) with $\lambda_{h} \in \mathbb{R}$, $n_{h} \ge 1$. A simple caculation implies that  $\textsf{ordd} \Phi (x) =n_{h}$. Moreover, for a homogeneous polynomial $\Phi$, the supremum in the definition of $\textsf{ordd}\Phi$ can always be achieved at some $x \in S^{1}$. Conversely, for fixed straight line  passing through the origin which is contained in $Z_{\Phi}$,  without loss of generality, we may again consider $x_{2}=\lambda_{h}x_{1}$ corresponding to the  factor $(x_{2}-\lambda_{h}x_{1})^{n_{h}}$ in (\ref{Phi}), there is an   $x \in S^{1} \cap \{x_{2}= \lambda_{h}x_{1}\}$ such that $\textsf{ordd} \Phi (x) =n_{h}$. In particular, if $E_{\Phi} \backslash \{ m/2 \} \neq \emptyset$, then for each $\kappa \in E_{\Phi} \backslash \{ m/2 \}$, there exists an  $x_{\kappa} \in S^{1}$ such that $\textsf{ordd} \Phi (x_{\kappa}) =\kappa $.

We are ready to prove Proposition \ref{muphihphi}.
 We first consider the case when $h_{ \Phi} = m/2$. If $E_{\Phi} = \{\frac{m}{2}\}$, then it is clear that
$\mu_{\Phi}=h_{\Phi}= \frac{m}{2}$; if $E_{\Phi} \backslash \{\frac{m}{2}\} \neq \emptyset$, there holds
\[\frac{m}{2} \ge \sup_{x \in S^{1}} \textsf{ordd} \Phi(x) \ge \sup_{\kappa \in E_{\Phi} \backslash \{m/2\}} \textsf{ordd} \Phi(x_{\kappa})= \sup_{\kappa \in E_{\Phi} \backslash \{m/2\}}  \kappa, \]
then it follows that $\mu_{\Phi} = \frac{m}{2} = h_{\Phi}$.

In the case when $h_{ \Phi} > m/2$, there exists an $x  \in S^{1} \cap Z_{\Phi}$ such that $\textsf{ordd} \Phi = \textsf{ordd} \Phi(x)> m/2$. Without loss of generality, we may assume that $x \in S^{1} \cap \{x_{2}=\lambda_{h}x_{1}\}$, where $x_{2}=\lambda_{h}x_{1}$ corresponds to the  factor $(x_{2}-\lambda_{h}x_{1})^{n_{h}}$ in (\ref{Phi}). Then we write $\Phi$ as
\[\Phi(x_{1},x_{2}) = (x_{2}-\lambda_{h}x_{1})^{n_{h}}Q(x_{1},x_{2}),\]
where $Q$ is a homogeneous polynomial such that $Q(x_{1},\lambda_{h}x_{1}) \neq 0$ for $x_{1} \neq 0$. It is easy to claim that when $n_{h} \ge 3$, $H\Phi(x_{1},\lambda_{h}x_{1}) =0$.

When $n_{h} > \frac{m}{2}$ and $m \ge 4$, then we must have $n_{h} \ge 3$. Therefore, when $m \ge 4$, we have $n_{h} \in E_{\Phi}$ and
\[n_{h} = \sup_{x \in S^{1}} \textsf{ordd} \Phi(x) \ge \sup_{\kappa \in E_{\Phi} \backslash \{m/2\}} \textsf{ordd} \Phi(x_{\kappa})= \sup_{\kappa \in E_{\Phi} \backslash \{m/2\}}  \kappa.\]
Hence there holds $\mu_{\Phi} = h_{\Phi} = n_{h}$.

Now we are left with the case when  $m=3$ or $m=2$.  When $m=3$ and $n_{h}=3$, the conclusion follows directly by the previous claim. When $m=3$ and $n_{h}=2$, we may write
\[\Phi(x_{1},x_{2}) = (x_{2}-\lambda_{h}x_{1})^{2}Q(x_{1},x_{2}),\]
where $Q$ is a linear polynomial that is homogeneous of degree one. A direct calculation implies that  $H\Phi(x_{1},\lambda_{h}x_{1}) =0$, then we have $\mu_{\Phi} = h_{\Phi} = 2$. When $m=2$, we must have $n_{h}=2$ because $h_{\phi}> \frac{m}{2}$, then
\[\Phi(x_{1},x_{2}) = c_{\Phi}(x_{2}-\lambda_{h}x_{1})^{2}.\]
It is obvious that $H\Phi \equiv 0$, so $n_{h} \in E_{\Phi}$ and $\mu_{\Phi} = h_{\Phi} =2$. This finishes the proof of Proposition \ref{muphihphi}.

\subsection{Proof of  Proposition \ref{level set type}}\label{prop1.5}
The following preliminary lemma will be useful in the proof of Proposition \ref{level set type}.
\begin{lemma}\label{level set leema}
Assume that $f(x_{1},x_{2})$ is a smooth real-valued function in $\mathbb{R}^{2}$. \\
(1) For $\lambda \in \mathbb{R}$, we denote $f_{\lambda}(x_{1},x_{2})=f(x_{1},x_{2}+\lambda x_{1})$,  then
\[Hf_{\lambda}(x_{1},x_{2})=(Hf)(x_{1},x_{2}+\lambda x_{1}).\]
(2) Similarly, for $\tau \in \mathbb{R}$, let $f_{\tau}(x_{1},x_{2})=f(x_{1}+\tau x_{2},x_{2})$, then
\[Hf_{\tau}(x_{1},x_{2})=(Hf)(x_{1}+\tau x_{2},x_{2}).\]
\end{lemma}
\begin{proof}
We only prove part (1), part (2) can be proved similarly. Since
\[\partial_{1}f_{\lambda}(x_{1},x_{2})= (\partial_{1}f)(x_{1},x_{2}+\lambda x_{1})+ \lambda (\partial_{2}f)(x_{1},x_{2}+\lambda x_{1}),\]
we have
\[\partial_{11}f_{\lambda}(x_{1},x_{2})= (\partial_{11}f)(x_{1},x_{2}+\lambda x_{1})+ 2\lambda (\partial_{12}f)(x_{1},x_{2}+\lambda x_{1})+\lambda^{2} (\partial_{22}f)(x_{1},x_{2}+\lambda x_{1}),\]
and
\[\partial_{12}f_{\lambda}(x_{1},x_{2})= (\partial_{12}f)(x_{1},x_{2}+\lambda x_{1})+ \lambda (\partial_{22}f)(x_{1},x_{2}+\lambda x_{1}).\]
In addition,
\[\partial_{2}f_{\lambda}(x_{1},x_{2})= (\partial_{2}f)(x_{1},x_{2}+\lambda x_{1}),\]
then
\[\partial_{22}f_{\lambda}(x_{1},x_{2})= (\partial_{22}f)(x_{1},x_{2}+\lambda x_{1}).\]
Now it is clear that
\begin{align}
&Hf_{\lambda}(x_{1},x_{2})\nonumber\\
&= \partial_{11}f_{\lambda}(x_{1},x_{2}) \partial_{22}f_{\lambda}(x_{1},x_{2})-[\partial_{12}f_{\lambda}(x_{1},x_{2})]^{2} \nonumber\\
&=\bigl[(\partial_{11}f)(x_{1},x_{2}+\lambda x_{1})+ 2\lambda (\partial_{12}f)(x_{1},x_{2}+\lambda x_{1})+\lambda^{2} (\partial_{22}f)(x_{1},x_{2}+\lambda x_{1}) \bigl]\nonumber\\
& \quad \times (\partial_{22}f)(x_{1},x_{2}+\lambda x_{1}) - \bigl[ (\partial_{12}f)(x_{1},x_{2}+\lambda x_{1})+ \lambda (\partial_{22}f)(x_{1},x_{2}+\lambda x_{1}) \bigl]^{2}              \nonumber\\
& = (\partial_{11}f)(x_{1},x_{2}+\lambda x_{1}) (\partial_{22}f)(x_{1},x_{2}+\lambda x_{1}) - (\partial_{12}f)(x_{1},x_{2}+\lambda x_{1})^{2} \nonumber\\
&=(Hf)(x_{1},x_{2}+\lambda x_{1}). \nonumber
\end{align}
This completes the proof of Lemma \ref{level set leema}.
\end{proof}

Now we are ready to prove Proposition \ref{level set type}. Let
\[\widetilde{\Phi}_{\lambda}(z_{1},z_{2}) = \Phi_{\lambda}(z_{1}+ \Psi^{\prime}(0)z_{2},z_{2}), \quad \widetilde{\Psi}(z_{2})=\Psi(z_{2}) -\Psi^{\prime}(0)z_{2}. \]
  Notice that for each positive integer $k \ge 2$, $\Psi^{(k)}(z_{2})= \widetilde{\Psi}^{(k)}(z_{2})$. $M^{\lambda}$ in Proposition \ref{level set type}  is just the  smallest integer   such that $\widetilde{\Psi}^{(M^{\lambda})}(0) \neq 0$.

We need to show that $M^{\lambda}< \infty$.  It is clear that there holds  the equation
\begin{equation}\label{reduce level set}
\widetilde{\Phi}_{\lambda}\bigl(\widetilde{\Psi}(z_{2}),z_{2}\bigl) =1, \quad  |z_{2}| < \epsilon,
\end{equation}
and  $\tilde{\Psi}^{\prime}(0)=0$. By taking derivative with respect to the variable $z_{2}$ in (\ref{reduce level set}) and setting $z_{2}=0$, we obtain
\[ \partial_{2}\widetilde{\Phi}_{\lambda}(z^{0}_{1},0)=0.\]
By the homogeneity of $\widetilde{\Phi}_{\lambda}$, we have
 \[\partial_{2}\widetilde{\Phi}_{\lambda}(z_{1},0)=0, \quad \forall z_{1} \neq 0.\]
By taking the $k$-th order derivative with respect to the variable $z_{2}$ on both side of the equation (\ref{reduce level set}) and putting $z_{2}=0$, we can obtain that
\[\widetilde{\Psi}^{(k)}(0)\partial_{1}\widetilde{\Phi}_{\lambda}(z^{0}_{1},0)  + (\partial^{k}_{2}\widetilde{\Phi}_{\lambda})(z^{0}_{1},0) =0, \quad 1 \le k \le M^{\lambda}. \]
Then $M^{\lambda}$ is the smallest integer such that
$(\partial^{M^{\lambda}}_{2}\widetilde{\Phi}_{\lambda})(z^{0}_{1},0) \neq 0$. We have $M^{\lambda} < \infty$ because $\tilde{\Phi}_{\lambda}$ is a homogeneous polynomial which is not of the form $cz^{m}_{1}$, $c \in \mathbb{R}$.

We first obtain the expression  of $\widetilde{\Phi}_{\lambda}$ and then the expression of $H\widetilde{\Phi}_{\lambda}$ via $M^{\lambda}$.
By homogeneity, from $(\partial^{M^{\lambda}}_{2}\widetilde{\Phi}_{\lambda})(z^{0}_{1},0) \neq 0$, we have
\[(\partial^{M^{\lambda}}_{2}\widetilde{\Phi}_{\lambda})(z_{1},0) \neq 0, \quad \forall z_{1} \neq 0. \]
 By Taylor's  formula  of $\widetilde{\Phi}_{\lambda}(z_{1},z_{2})$ at $z_{2}=0$, we have
\begin{equation}\label{tildePhilambda}
\widetilde{\Phi}_{\lambda}(z_{1},z_{2}) = c_{\Phi}z^{m}_{1} + z^{M^{\lambda}}_{2}P_{\lambda}(z_{1},z_{2}),
\end{equation}
where the constant $c_{\Phi} \neq 0$, and  $P_{\lambda}(z_{1},0) \neq 0$ for each $z_{1} \neq 0$.
It follows from (\ref{tildePhilambda}) and $m \ge 2$ that
\begin{equation}\label{HPhi1}
H\widetilde{\Phi}_{\lambda}(z_{1},z_{2}) = z^{M^{\lambda}-2}_{2}Q_{\lambda}(z_{1},z_{2}),
\end{equation}
 where $Q_{\lambda}(z_{1},0) \neq 0$ for each $z_{1} \neq 0$.

 Next, we will obtain the expression of $H\widetilde{\Phi}_{\lambda}$  by (\ref{HPhi0}), changing  variables and Lemma \ref{level set leema}, then compare the result with (\ref{HPhi1}) to get the exact value of $M^{\lambda}$.

$\bullet$ For $\lambda < \infty$, we recall that
\[\Phi_{\lambda}(z_{1},z_{2}) =\Phi(z_{1},z_{2}+ \lambda z_{1}).\]
By (\ref{HPhi0}) and part (1) in Lemma \ref{level set leema},
\[H\Phi_{\lambda}(z_{1},z_{2})= (H\Phi) \bigl(z_{1} ,z_{2}+\lambda
z_{1}  \bigl)  = z^{\omega_{\lambda}}_{2} \widetilde{Q}_{\lambda}(z_{1},z_{2}),\]
where  $\widetilde{Q}_{\lambda}(z_{1},z_{2})$ is a homogeneous polynomial and  $\widetilde{Q}_{\lambda}(z_{1},0) \neq 0$ for each $z_{1} \neq 0$.
Then according to  part (2) in Lemma \ref{level set leema}, we have
\begin{align}\label{HPhi2}
H\widetilde{\Phi}_{\lambda}(z_{1},z_{2}) &= (H\Phi_{\lambda}) \bigl(z_{1}+\Psi^{\prime}(0)z_{2},z_{2} \bigl) = z^{\omega_{\lambda}}_{2} \overline{Q}_{\lambda}(z_{1},z_{2}),
\end{align}
where
 \[\overline{Q}_{\lambda}(z_{1},z_{2})=  \widetilde{Q}_{\lambda}(z_{1}+\Psi^{\prime}(0)z_{2},z_{2}).\]
 Then we have  $\overline{Q}(z_{1},0) \neq 0$ for each $z_{1} \neq 0$.
It follows from (\ref{HPhi1}) and (\ref{HPhi2}) that $M^{\lambda} = \omega_{\lambda}+2$.

$\bullet$ For $\lambda = \infty$, we recall that
\[\Phi_{\infty}(z_{1},z_{2}) =\Phi(z_{2}, z_{1}).\]
Then
\[H\Phi_{\infty}(z_{1},z_{2})= (H\Phi) \bigl(z_{2} ,
z_{1}  \bigl)  = z^{\alpha_{1}}_{2} \widetilde{Q}_{\infty}(z_{1},z_{2}),\]
where $\widetilde{Q}_{\infty}(z_{1},z_{2})$ is a homogeneous polynomial and  $\widetilde{Q}_{\infty}(z_{1},0) \neq 0$ for each $z_{1} \neq 0$. Then by part (2) in Lemma \ref{level set leema},
\begin{align}\label{HPhi3}
H\widetilde{\Phi}_{\infty}(z_{1},z_{2}) &= (H\Phi_{\infty}) \bigl( z_{1}+\Psi^{\prime}(0)z_{2},z_{2}  \bigl) = z^{\alpha_{1}}_{2} \overline{Q}_{\infty}( z_{1},z_{2}),
\end{align}
with $\overline{Q}_{\infty}( z_{1},z_{2})= \widetilde{Q}_{\infty}( z_{1}+\Psi^{\prime}(0)z_{2},z_{2})$. We have   $\overline{Q}_{\infty}(z_{1},0) \neq 0$ for each $z_{1} \neq 0$, then we obtain that $M^{\lambda} = \alpha_{1}+2$ by (\ref{HPhi1}) and (\ref{HPhi3}).

\section{Appendix}\label{appendix}
This appendix consists of two parts.
 Firstly, for the completeness of the paper, we explain how to prove the H\"{o}lder continuity property starting from the $L^{p} \rightarrow L^{q}$ bounded estimates of the local maximal operators in this paper. Once we obtain the H\"{o}lder continuity property of the local maximal operator, we can immediately obtain the weighted estimate for the corresponding  global maximal operator by combining it with Theorem 1.20 and Corollary 1.21 in \cite{LWZ}. Secondly,   we briefly explain how to obtain inequality (\ref{nonvanishcurvature}), which has   appeared in the proof of Theorem \ref{mainth1+}.

\subsection{H\"{o}lder continuity property}
This subsection is a brief overview of how we can prove the H\"{o}lder continuity property of the local maximal operator considered in this article, that is to prove that
\begin{equation} \label{continuity:ineq}
\biggl\|\sup_{t\in [1,2]} \biggl| A_{t}f(y+z) - A_{t}f(y)  \biggl| \biggl\|_{L^{q}} \lesssim |z|^{\epsilon} \|f\|_{L^{p}}
\end{equation} 
for some real number $\epsilon >0$ and any $z \in \mathbb{R}^{3}$, $|z| \le 1$, whenever $(\frac{1}{p},\frac{1}{q}) \neq (0,0)$ and belongs to the region $\Delta$ where $\mathcal{M}_{loc}$  is $L^{p} \rightarrow L^{q} $ bounded.
Since the hypersurface  studied in this paper are of finite type, the Fourier transform of their surface measures decays to a certain extent. Therefore, we can utilize the argument in Section 1.1.3 of reference \cite{LWZ} to deduce the problem to obtaining the inequality
\begin{align}
\biggl\|\sup_{t\in[1,2], h \in [0, \delta]} \biggl| A_{t+h}f  - A_{t}f  \biggl|  \biggl\|_{L^{q} } \lesssim \delta^{\epsilon_1} \|f\|_{L^{p} } \nonumber
\end{align}
for every $\delta \in (0,1)$ and some $\epsilon_{1}>0$. By the argument in the introduction, we reduce the problem to proving that
\begin{align}\label{goal}
\biggl\|\sup_{t\in[1,2], h \in [0, \delta]} \biggl| A_{t+h,\lambda}f  - A_{t,\lambda}f  \biggl|  \biggl\|_{L^{q}} \lesssim \delta^{\epsilon_1} \|f\|_{L^{p}}
\end{align}
for each $(\frac{1}{p},\frac{1}{q}) \in \Delta_{\lambda}$  such that $\mathcal{M}_{loc,\lambda}$ is $L^{p} \rightarrow L^{q}$ bounded, where $A_{t,\lambda}$ is defined by (\ref{avAlambda}). Noting that for each $(\frac{1}{p},\frac{1}{q}) \in \Delta_{\lambda}$,  we have
\[\biggl\|\sup_{t\in[1,2], h \in [0, \delta]} \biggl| A_{t+h,\lambda}f  - A_{t,\lambda}f  \biggl|  \biggl\|_{L^{q}} \lesssim \|f\|_{L^{p}},\]
then we can get inequality (\ref{goal}) by interpolating between this estimate and the results in the following theorem.

\begin{theorem}
The estimate (\ref{goal}) holds true  whenever \\
$\bullet$ $\Gamma(\lambda)$ is of Type A, $c=0$, $(\frac{1}{p},\frac{1}{q}) = (\frac{1}{2},\frac{1}{2})$; $c \neq 0$, $(\frac{1}{p},\frac{1}{q}) = (\frac{1}{8h_{\Phi}},\frac{1}{8h_{\Phi}})$.  \\
$\bullet$ $\Gamma(\lambda)$ is of Type B, $c=0$, $(\frac{1}{p},\frac{1}{q}) = (\frac{1}{2},\frac{1}{2})$; $c \neq 0$, $(\frac{1}{p},\frac{1}{q}) = (\frac{1}{8h_{\Phi}},\frac{1}{
8h_{\Phi}})$.
\end{theorem}
\begin{proof}
  After a dyadic decomposition and a scaling argument, we have
\begin{align}\label{holder1}
&\biggl\|\sup_{t\in[1,2], h \in [0, \delta]} \biggl| A_{t+h,\lambda}f  - A_{t,\lambda}f  \biggl|  \biggl\|_{L^{p}} \nonumber\\
 &\lesssim \sum_{k \ge 1} 2^{-2k} \biggl\|\sup_{t\in[1,2], h \in [0, \delta]} \biggl|  A^{\lambda}_{t+h,k} f  - A^{\lambda}_{t,k} f  \biggl|  \biggl\|_{L^{p}},
\end{align}
where $ A^{\lambda}_{t,k} $ is defined by
\begin{align}
  \int_{\mathbb{R}^{2}}f\bigl(y-t(2^{-k}x_{1}, 2^{-k}x_{2}, 2^{-mk} \Phi(x_{1},x_{2})+c  )\bigl)   \eta_{0}(\frac{x_{2}-\lambda x_{1}}{\epsilon x_{1}})\tilde{\eta}(x_{1},x_{2}) \rho_{0} (2^{-k}x_{1},2^{-k}x_{2})dx_{1}dx_{2}, \nonumber
\end{align}
 where $\eta_{0}$ and  $\rho_{0}$ are smooth cut-off functions supported in $B(0,1)$, $\tilde{\eta}$ is a smooth cut-off function with $\mathrm{supp}$ $\tilde{\eta} \subset \{x \in \mathbb{R}^{2}:  |x| \sim 1\}$.  By an isometric transformation, we are reduced to proving that
 \begin{align}\label{holderl2}
 \biggl\|\sup_{t\in[1,2], h \in [0, \delta]} \biggl| \widetilde{A^{\lambda}_{t+h,k}}f  - \widetilde{A^{\lambda}_{t,k}}f  \biggl|  \biggl\|_{L^{p} } \lesssim  2^{\epsilon(p)k}  \delta^{\epsilon_{1}} \|f\|_{L^{p}}
\end{align}
for some $\epsilon_{1}>0$ and some $p>1$, $\epsilon (p)<2$. Here, $ \widetilde{A^{\lambda}_{t,k}}$ is defined by (\ref{h0}).

$\bullet$  If $\Gamma(\lambda)$ is of Type A,  (\ref{holderl2}) follows if we can prove that
 \begin{align}\label{holderl2dec}
   \sum_{l \ge l_{0}} 2^{-l}   \biggl\|\sup_{t\in[1,2], h \in [0, \delta]} \biggl| \widetilde{A^{\lambda}_{t+h,k,l}}f - \widetilde{A^{\lambda}_{t,k,l}}f  \biggl|  \biggl\|_{L^{p}} \lesssim  2^{\epsilon(p)k} \delta^{\epsilon_{1}} \|f\|_{L^{p}},
\end{align}
where   $\widetilde{A^{\lambda}_{t,k,l}}$ denotes the averaging operator along the hypersurface
$S^{\lambda}_{k,l}$ defined by (\ref{hypersurface1})  whose  Gaussian curvature is non-vanishing.
 We further introduce the operator $\widetilde{A_{t,k,l,j}^{\lambda}}$ for each  integer $j \ge 0$, such that the Fourier transform of $\widetilde{A_{t,k,l,j}^{\lambda}}f$ is supported in $\{\xi \in \mathbb{R}^{3}: |\xi| \sim 2^{j}\}$ when $j \ge1$, and the  Fourier transform of $\widetilde{A_{t,k,l,0}^{\lambda}}f$ is supported in $\{\xi \in \mathbb{R}^{3}: |\xi|  \le 1\}$.
Sobolev  embedding and Plancherel's theorem imply that
\begin{align}\label{holderl2h1}
&\biggl\|\sup_{t\in[1,2], h \in [0, \delta]} \biggl| \widetilde{A^{\lambda}_{t+h,k,l,j}}f - \widetilde{A^{\lambda}_{t,k,l,j}}f \biggl|  \biggl\|_{L^{2}} \nonumber\\
&\lesssim \biggl(\int_{\mathbb{R}^{3}} \sup_{t\in [1,2]}\biggl(\int^{\delta}_{0}  \biggl| \frac{\partial}{\partial h} \bigl[ \widetilde{A^{\lambda}_{t+h,k,l,j}}f(y) \bigl]  \biggl|^{2}   dh \biggl)^{\frac{1}{2}} \biggl(\int^{\delta}_{0}\biggl| \widetilde{A^{\lambda}_{t+h,k,l,j}}f(y) - \widetilde{A^{\lambda}_{t,k,l,j}}f(y) \biggl|^{2} dh \biggl)^{1/2} dy \biggl)^{\frac{1}{2}} \nonumber\\
&\lesssim \biggl(\int_{\mathbb{R}^{3}}  \int^{\delta}_{0}  \sup_{t\in [1,2]} \biggl| \frac{\partial}{\partial h} \bigl[ \widetilde{A^{\lambda}_{t+h,k,l,j}}f(y) \bigl]  \biggl|^{2}   dh dy \biggl)^{\frac{1}{4}}   \biggl( \int_{\mathbb{R}^{3}}\int^{\delta}_{0} \sup_{t\in[1,2]} \biggl| \widetilde{A^{\lambda}_{t+h,k,l,j}}f(y) - \widetilde{A^{\lambda}_{t,k,l,j}}f(y) \biggl|^{2} dh  dy \biggl)^{\frac{1}{4}}  \nonumber\\
&\lesssim  \biggl( \int^{\delta}_{0} \int_{\mathbb{R}^{3}}  \sup_{t\in [1,2]} \biggl| \frac{\partial}{\partial h}\bigl[ \widetilde{A^{\lambda}_{t+h,k,l,j}}f(y)  \bigl]  \biggl|^{2}   dy dh \biggl)^{\frac{1}{4}}   \biggl( \int^{\delta}_{0} \int_{\mathbb{R}^{3}} \sup_{t\in[1,2+\delta]}  \biggl| \widetilde{A^{\lambda}_{t,k,l,j}}f(y)   \biggl|^{2}  dy dh  \biggl)^{\frac{1}{4}}  \nonumber\\
&\lesssim (1+ c2^{mk+n_{\lambda}l} ) \delta^{\frac{1}{2}} \|f\|_{L^{2}},
\end{align}
and
\begin{align}\label{holderl2h2}
\biggl\|\sup_{t\in[1,2], h \in [0, \delta]} \biggl| \widetilde{A^{\lambda}_{t+h,k,l,j}}f - \widetilde{A^{\lambda}_{t,k,l,j}}f \biggl|  \biggl\|_{L^{2} }  &\lesssim \biggl\|\sup_{t\in[1,2+\delta]} \biggl|\widetilde{A^{\lambda}_{t,k,l,j}}f \biggl|  \biggl\|_{L^{2} } \nonumber\\
&\lesssim (1+c2^{\frac{mk+n_{\lambda}l}{2}}) 2^{-\frac{j}{2}} \|f\|_{L^{2} }.
\end{align}
We notice that we have applied the fact that $\delta \in (0,1)$.

When $c=0$, interpolation between (\ref{holderl2h1}) and (\ref{holderl2h2}) implies that for $\epsilon_{1} \in (0,1)$, there holds
\begin{align}
\biggl\|\sup_{t\in[1,2], h \in [0, \delta]} \biggl| \widetilde{A^{\lambda}_{t+h,k,l,j}}f  - \widetilde{A^{\lambda}_{t,k,l,j}}f  \biggl|  \biggl\|_{L^{2}} \lesssim \delta^{\frac{\epsilon_{1}}{2}} 2^{-\frac{(1-\epsilon_{1})j}{2}} \|f\|_{L^{2}}. \nonumber
\end{align}
Summing over $j$, we get
\begin{align}
\biggl\|\sup_{t\in[1,2], h \in [0, \delta]} \biggl| \widetilde{A^{\lambda}_{t+h,k,l}}f  - \widetilde{A^{\lambda}_{t,k,l}}f  \biggl|  \biggl\|_{L^{2}} \lesssim \delta^{\frac{\epsilon_{1}}{2}} \|f\|_{L^{2}}. \nonumber
\end{align}
It follows   that
\begin{align}
   \sum_{l \ge l_{0}} 2^{-l}   \biggl\|\sup_{t\in[1,2], h \in [0, \delta]} \biggl| \widetilde{A^{\lambda}_{t+h,k,l}}f  - \widetilde{A^{\lambda}_{t,k,l}}f  \biggl|  \biggl\|_{L^{2}} \lesssim   \sum_{l \ge l_{0}} 2^{-l} \delta^{\frac{\epsilon_{1}}{2}}   \|f\|_{L^{2}} \lesssim   \delta^{\frac{\epsilon_{1}}{2}} \|f\|_{L^{2}}. \nonumber
\end{align}
Then (\ref{holderl2dec}) holds true when $p=2$  and $\epsilon(p)=0$.

When $c \neq 0$,
inequalities (\ref{holderl2h1}) and (\ref{holderl2h2}) imply that for $\epsilon_{1} \in (0,1)$, there holds
\begin{align}
\biggl\|\sup_{t\in[1,2], h \in [0, \delta]} \biggl| \widetilde{A^{\lambda}_{t+h,k,l,j}}f  - \widetilde{A^{\lambda}_{t,k,l,j}}f  \biggl|  \biggl\|_{L^{2} } \lesssim 2^{mk+n_{\lambda}l} \delta^{\frac{\epsilon_{1}}{2}} 2^{-\frac{(1-\epsilon_{1})j}{2}} \|f\|_{L^{2} }, \nonumber
\end{align}
this yields
\begin{align}\label{holderl2h4}
\biggl\|\sup_{t\in[1,2], h \in [0, \delta]} \biggl| \widetilde{A^{\lambda }_{t+h,k,l}}f  - \widetilde{A^{\lambda }_{t,k,l}}f \biggl|  \biggl\|_{L^{2} } \lesssim  2^{mk+n_{\lambda}l} \delta^{\frac{\epsilon_{1}}{2}}  \|f\|_{L^{2} }.
\end{align}
Moreover, we have the following trivial estimate,
\begin{align}\label{holderlinftyh}
\biggl\|\sup_{t\in[1,2], h \in [0, \delta]} \biggl| \widetilde{A^{\lambda }_{t+h,k,l}}f  - \widetilde{A^{\lambda }_{t,k,l}}f  \biggl|  \biggl\|_{L^{\infty} } \lesssim  \|f\|_{L^{\infty} }.
\end{align}
Interpolation between (\ref{holderl2h4}) and (\ref{holderlinftyh}) yields
\begin{align}
\biggl\|\sup_{t\in[1,2], h \in [0, \delta]} \biggl| \widetilde{A^{\lambda}_{t+h,k,l}}f  - \widetilde{A^{\lambda }_{t,k,l}}f  \biggl|  \biggl\|_{L^{p} } \lesssim 2^{\frac{2mk+2n_{\lambda}l}{p}} \delta^{\frac{\epsilon_{1}}{p}}  \|f\|_{L^{p} }. \nonumber
\end{align}
Setting $p=8h_{\Phi}$ and combining with $h_{\Phi}> \max\{\frac{m}{2},n_{\lambda} \}$, we have
\begin{align}
   \sum_{l \ge l_{0}} 2^{-l}   \biggl\|\sup_{t\in[1,2], h \in [0, \delta]} \biggl| \widetilde{A^{\lambda}_{t+h,k,l}}f  - \widetilde{A^{\lambda}_{t,k,l}}f  \biggl|  \biggl\|_{L^{8h_{\Phi}}}
   &\lesssim   \sum_{l \ge l_{0}} 2^{-l} 2^{\frac{2mk+2n_{\lambda}l}{8h_{\Phi}}} \delta^{\frac{\epsilon_{1}}{8h_{\Phi}}}  \|f\|_{L^{8h_{\Phi}}} \nonumber\\
  & \lesssim  2^{\frac{mk}{4h_{\Phi}}} \delta^{\frac{\epsilon_{1}}{8h_{\Phi}}} \|f\|_{L^{2}}. \nonumber
\end{align}
Hence,  (\ref{holderl2dec})   holds when  $p =8h_{\Phi}$ and $\epsilon(p)= \frac{m}{4h_{\phi}}<2$.

$\bullet$   If $\Gamma(\lambda)$ is of Type B, we introduce the averaging operator $\widetilde{A_{t,k,j}^{\lambda}}$ along the hypersurface $\tilde{S_{k}}$ defined by (\ref{hypersurface2}) for each integer $j \ge 0$, such that the Fourier transform of $\widetilde{A_{t,k,j}^{\lambda}}f$ is supported in $\{\xi \in \mathbb{R}^{3}: |\xi| \sim 2^{j}\}$ when $j \ge 1$, and the Fourier transform of $\widetilde{A_{t,k,0}^{\lambda}}f$ is supported in $\{\xi \in \mathbb{R}^{3}: |\xi| \le 1\}$.
Sobolev embedding and the maximal estimate (\ref{maximalL2}) yield
\begin{align}\label{holderl2tau1}
&\biggl\|\sup_{t\in[1,2], h \in [0, \delta]} \biggl| \widetilde{A^{\lambda}_{t+h,k,j}}f - \widetilde{A^{\lambda}_{t,k,j}}f \biggl|  \biggl\|_{L^{2}} \nonumber\\
&\lesssim  \biggl( \int^{\delta}_{0} \int_{\mathbb{R}^{3}}  \sup_{t\in [1,2]} \biggl| \frac{\partial}{\partial h}\bigl[ \widetilde{A^{\lambda}_{t+h,k,l,j}}f(y) \bigl]  \biggl|^{2}   dy dh \biggl)^{\frac{1}{4}}   \biggl( \int^{\delta}_{0} \int_{\mathbb{R}^{3}} \sup_{t\in[1,2+\delta]}  \biggl| \widetilde{A^{\lambda}_{t,k,l,j}}f(y)   \biggl|^{2}  dy dh  \biggl)^{\frac{1}{4}}  \nonumber\\
&\lesssim (1+c2^{mk}) \delta^{\frac{1}{2}} 2^{(\frac{1}{2}-\frac{1}{M})j} \|f\|_{L^{2} },
\end{align}
and
\begin{align}\label{holderl2tau2}
\biggl\|\sup_{t\in[1,2], h \in [0, \delta]} \biggl| \widetilde{A^{\lambda}_{t+h,k,j}}f - \widetilde{A^{\lambda}_{t,k,j}}f \biggl|  \biggl\|_{L^{2}}  &\lesssim \biggl\|\sup_{t\in[1,2+\delta]} \biggl|\widetilde{A^{\lambda}_{t,k,l,j}}f \biggl|  \biggl\|_{L^{2} } \nonumber\\
&\lesssim (1+c2^{\frac{mk}{2}})  2^{-\frac{j}{M}} \|f\|_{L^{2}}.
\end{align}
Notice that we only need to apply (\ref{maximalL2}) for $j \ge 1$. For $j=0$,  the above estimates follow directly from Sobolev embedding and Plancherel's theorem.

When $c=0$, interpolating between (\ref{holderl2tau1}) and (\ref{holderl2tau2}), we obtain that   for $\epsilon_{1} \in (0,\frac{2}{M})$, there holds
\begin{align}
\biggl\|\sup_{t\in[1,2], h \in [0, \delta]} \biggl| \widetilde{A^{\lambda}_{t+h,k,j}}f  - \widetilde{A^{\lambda}_{t,k,j}}f \biggl|  \biggl\|_{L^{2}} \lesssim \delta^{\frac{\epsilon_{1}}{2}} 2^{(\frac{\epsilon_{1}}{2}-\frac{1}{M})j} \|f\|_{L^{2}}. \nonumber
\end{align}
Summing over $j$, we get
\begin{align}
\biggl\|\sup_{t\in[1,2], h \in [0, \delta]} \biggl| \widetilde{A^{\lambda}_{t+h,k,j}}f  - \widetilde{A^{\lambda}_{t,k,j}}f  \biggl|  \biggl\|_{L^{2}} \lesssim \delta^{\frac{\epsilon_{1}}{2}} \|f\|_{L^{2}}, \nonumber
\end{align}
then   (\ref{holderl2}) remains valid   with  $p=2$  and $\epsilon(p)=0$.


When $c \neq 0$,
interpolating between (\ref{holderl2tau1}) and (\ref{holderl2tau2}), we obtain that   for $\epsilon_{1} \in (0,\frac{2}{M})$, there holds
\begin{align}
\biggl\|\sup_{t\in[1,2], h \in [0, \delta]} \biggl| \widetilde{A^{\lambda}_{t+h,k,j}}f  - \widetilde{A^{\lambda}_{t,k,j}}f  \biggl|  \biggl\|_{L^{2}} \lesssim 2^{mk} \delta^{\frac{\epsilon_{1}}{2}} 2^{(\frac{\epsilon_{1}}{2}-\frac{1}{M})j} \|f\|_{L^{2}}. \nonumber
\end{align}
It follows that
\begin{align}\label{holderl2tau4}
\biggl\|\sup_{t\in[1,2], h \in [0, \delta]} \biggl| \widetilde{A^{\lambda}_{t+h,k}}f  - \widetilde{A^{\lambda}_{t,k}}f  \biggl|  \biggl\|_{L^{2}} \lesssim 2^{mk}\delta^{\frac{\epsilon_{1}}{2}} \|f\|_{L^{2}}.
\end{align}
Interpolating between (\ref{holderl2tau4}) and the trivial estimate
\[
\biggl\|\sup_{t\in[1,2], h \in [0, \delta]} \biggl| \widetilde{A^{\lambda}_{t+h,k}}f  - \widetilde{A^{\lambda}_{t,k}}f  \biggl|  \biggl\|_{L^{\infty}} \lesssim   \|f\|_{L^{\infty}},\]
we can obtain that for each $2<p<\infty$,
\begin{align}
\biggl\|\sup_{t\in[1,2], h \in [0, \delta]} \biggl| \widetilde{A^{\lambda}_{t+h,k}}f  - \widetilde{A^{\lambda}_{t,k}}f  \biggl|  \biggl\|_{L^{p}} \lesssim 2^{\frac{2mk}{p}}\delta^{\frac{\epsilon_{1}}{p}} \|f\|_{L^{p}}. \nonumber
\end{align}
Combining with the fact that $h_{\Phi}> \frac{m}{2}$,  we obtain (\ref{holderl2})  with  $p =8h_{\Phi}$ and $\epsilon(p)= \frac{m}{4h_{\phi}}<2$.
\end{proof}

\subsection{A complementary discussion  of the proof of Theorem \ref{mainth1+}}\label{exmplain}
In the proof of Theorem \ref{mainth1+}, the estimate (\ref{nonvanishcurvature}) arises from a general conclusion. \\
\textbf{Theorem A. }\textit{Let $U$ be a compact set in $\mathbb{R}^{2}$, $\Psi$ be a function with the non-degenerate Hessian matrix, $S=\{(x_{1},x_{2},\Psi(x_{1},x_{2}) + c): (x_{1},x_{2}) \in U\}$ be a smooth hypersurface, and $c \in \mathbb{R}$. Then the $L^{p}  \rightarrow L^{q} $ norm of the local maximal operator along $S$ is bounded by $(1+c^{1/q})$ provided that $(\frac{1}{p}, \frac{1}{q}) \in \Delta_{0}$.}

 Here we outline the main idea behind the proof of this conclusion.
Upon Littlewood-Paley decomposition, when $\xi$ is constrained to $\{\xi: |\xi| \sim 2^{j}\}$, $j \ge 1$,  if $|\xi_{1}| + |\xi_{2}| \gg |\xi_{3}|$, it can be proven that  the Fourier transform along $S$ decays rapidly to complete the corresponding maximal function estimate. Next, we present the application of the method of the stationary phase  in
\[\widehat{d\mu_{S}}(t\xi)=\int_{\mathbb{R}^{2}}e^{-it[x_{1}\xi_{1}+x_{2}\xi_{2}+\Phi(x_{1},x_{2})\xi_{3}]} \eta(x_{1},x_{2})dx_{1}dx_{2}, \quad |\xi_{1}|+|\xi_{2}|\sim 2^{j}, |\xi_{1}|+|\xi_{2}| \lesssim |\xi_{3}|,\]
where $t \in [1,2]$, $\eta$ is a smooth cut-off function supported in $U$. According to Theorem 1.2.1 in \cite{sogge2}, the main domination part of $\widehat{d\mu_{S}}(t\xi)$ can be expressed as the sum of a finite number of terms of the following form
\[e^{-it\tilde{\Phi}(\xi)}\frac{a(t,\xi)}{(1+t|\xi|)},\]
where the function $\widetilde{\Phi}(\xi)$ is homogeneous of degree one and its Hessian matrix has a rank of $2$, $a(t,\xi)$ is a symbol of order zero in $\xi$.  Through Sobolev  embedding, we transform the $L^{p} \rightarrow L^{q}$ estimate of the maximal operator into the $L^{p}(\mathbb{R}^{3}) \rightarrow L^{q}(\mathbb{R}^{3} \times [1,2])$ estimate  of the following Fourier integral operator
\[\int_{\mathbb{R}^{3}}e^{iy\cdot \xi-it\tilde{\Phi}(\xi)}\frac{a(t,\xi)}{(1+t|\xi|)} \hat{f}(\xi) d\xi.\]
It is worth noting that the surface $\{(\eta,\tilde{\Phi}(\eta)): \eta \in \textmd{supp} a(t,\cdot)\}$ has two non-vanishing principal curvatures, which implies that the method of estimating for the corresponding Fourier integral operator is consistent with the proof described in Section 4 of reference \cite{WS}.

\subsection*{Acknowledgements}
This work is supported by the National Key R\&D Program of China (No.2023YFA1010800); Natural Science Foundation of China (No.12301113; No.12271435).

 All authors would like to extend their sincere gratitude to Prof. Detlef M\"{u}ller for his careful revisions to the overall proof of the manuscript, which significantly enhanced its readability. They also greatly appreciate Prof. Sanghyuk Lee and Dr. Stefan Buschenhenke for their valuable suggestions regarding some counterexamples in our paper.




\end{document}